\documentclass[12pt,reqno]{amsart}
\usepackage{}
\usepackage{a4wide}
\numberwithin{equation}{section}
\usepackage{mathrsfs}
\usepackage{amsfonts}
\usepackage{amsmath,extpfeil}
\usepackage{stmaryrd}
\usepackage{amssymb}
\usepackage{amsthm}
\usepackage{mathrsfs}
\usepackage{url}
\usepackage{amsfonts}
\usepackage{amscd}
\usepackage{indentfirst}
\usepackage{enumerate}
\usepackage{amsmath,amsfonts,amssymb,amsthm}
\usepackage{amsmath,amssymb,amsthm,amscd}
\usepackage{graphicx,mathrsfs}
\usepackage{appendix}
\usepackage[numbers,sort&compress]{natbib}
\usepackage{color}

\makeatletter
\@namedef{subjclassname@2020}{\textup{2020} Mathematics Subject Classification}
\makeatother

\newtheorem{prop}{Proposition}[section]
\newtheorem{theorem}{Theorem}[section]
\newtheorem{lemma}{Lemma}[section]
\newtheorem{corollary}{Corollary}[section]

\newtheorem*{definition}{Definition}
\newtheorem{remark}{Remark}[section]    
\numberwithin{equation}{section}

\newcommand{\beq}{\begin{equation}}
\newcommand{\eeq}{\end{equation}}

\allowdisplaybreaks[4]
\begin{document}

\title[Guillemin boundary Monge-Amp\`ere equation]
{Monge-Amp\`ere equation with Guillemin boundary condition in high dimension}
\author[G. Huang and W. Shen]
{Genggeng Huang and Weiming Shen}
\address[Genggeng Huang]
{School of Mathematical Sciences, Fudan University, Shanghai 200433, China.}
\email{genggenghuang@fudan.edu.cn}

\address[Weiming Shen]{School of Mathematical Sciences,
Capital Normal University,
Beijing, 100048, China}
\email{wmshen@aliyun.com}

\thanks{The first author is partially supported by NSFC 12141105. The second author is partially supported by NSFC 12371208.}


\date{}

\begin{abstract}
The Guillemin boundary condition naturally appears in the study of K\"ahler geometry of toric  manifolds.
In the present paper, the following  Guillemin boundary value problem is investigated
\begin{eqnarray}
		&&\det D^2 u=\frac{h(x)}{\prod_{i=1}^N l_i(x)},\quad\text{in}\quad P\subset\mathbb R^n,\label{intro1-0}\\
		&&u(x)-\sum_{i=1}^N l_i(x)\ln  l_i(x)\in C^\infty(\overline{P}).    \label{intro1-1-0}
	\end{eqnarray}
	Here
  \begin{equation*}
  0<h(x)\in C^\infty(\overline{P}),\quad P=\cap_{i=1}^N \{l_i(x)>0\}
  \end{equation*}
   is a simple convex polytope in $\mathbb R^n$.  The solvability of \eqref{intro1-0}-\eqref{intro1-1-0} is given under the necessary and sufficient condition.
 The key issue in the proof is to obtain the boundary regularity of $u(x)-\displaystyle \sum_{i=1}^N l_i(x)\ln  l_i(x)$. Due to the difficulty caused by the structure of the equation itself and the singularity of $\partial P$, special attention is required to understand the influence of different singularity types at various positions on $\partial P$  and how these impact the behavior of $u$ in its vicinity.
 \end{abstract}

\maketitle
\tableofcontents

\section{Introduction}
This paper is devoted to the study of the regularity of the following boundary value problem:
\begin{eqnarray}
		&&\det D^2 u=\frac{h(x)}{\prod_{i=1}^N l_i(x)},\quad\text{in}\quad P\subset\mathbb R^n,\label{intro1}\\
		&&u(x)-\sum_{i=1}^N l_i(x)\ln  l_i(x)\in C^\infty(\overline{P}).    \label{intro1-1}
	\end{eqnarray}
	Here,
  \begin{equation*}
  0<h(x)\in C^\infty(\overline{P}),\quad P=\cap_{i=1}^N \{l_i(x)>0\}
  \end{equation*}
   is a simple convex polytope in $\mathbb R^n$.  The functions $l_i(x)$ are affine functions for $i=1,\cdots,N$.
\begin{definition}\label{def1}
A polytope in $\mathbb R^n$ is called a {\bf simple polytope} if and only if every vertex is contained in the minimal number of only $n$ co-dimension $1$ faces.
\end{definition}
\eqref{intro1-1} is called the Guillemin boundary condition. The Guillemin boundary condition arises from the K\"ahler geometry of toric manifolds.  A Delzant polytope is a special type of simple convex polytope that satisfies additional conditions, see Definition 2.1 in \cite{LiSheng2023}.
Let $X$ be a toric manifold of dimension $n$.  Its image under the moment map is a Delzant polytope(Page 8, \cite{Guillemin2012}) $P$ in $\mathbb R^n$.
Then one can determine a K\"ahler potential $\phi$ and hence a K\"ahler metric $\omega$ on $X$ by some convex function  $u:\overline{P}\rightarrow \mathbb R$. This convex function $u$, i.e., the symplectic potential, should satisfy the following beautiful equation  given by Abreu \cite{Abreu1998}
\begin{equation}\label{Abreu1}
\sum_{i,j=1}^n\frac{\partial^2 u^{ij}}{\partial x_i\partial x_j}=-2R
\end{equation}
where $(u^{ij})$ is the inverse of $(u_{ij})$ and $R$ is the scalar curvature of the corresponding K\"ahler metric $\omega$.
In order to obtain a smooth K\"ahler metric $\omega$, Guillemin \cite{Guillemin1994} proved that the solution of \eqref{Abreu1} should satisfy
\begin{equation}\label{Guillemin-1}
u(x)=\sum_{i=1}^N  l_i(x)\ln l_i(x)+v(x),\quad v\in C^\infty(\overline{P}).
\end{equation}
This is the same boundary condition as \eqref{intro1-1}. Donaldson \cite{Donaldson2008,Donaldson2009} called \eqref{Guillemin-1} the Guillemin boundary condition for \eqref{Abreu1}. We adopt Donaldson's notation for the Monge-Amp\`ere equation \eqref{intro1}. Hence,  \eqref{intro1}-\eqref{intro1-1} is called the Guillemin boundary problem for Monge-Amp\`ere equation.
In a series of papers \cite{Donaldson2008,Donaldson2009}, for $n=2$, Donaldson also proved the existence of solutions of \eqref{Abreu1}-\eqref{Guillemin-1} by the continuity method for $R$ being a constant under some suitable stability condition. This gives the existence of constant scalar curvature K\"ahler metric on toric surfaces.   Chen, Li and Sheng \cite{ChenLiSheng2014} extended this result to the case of extremal metrics for toric surfaces. Recently, Chen and Cheng  \cite{ChenCheng12021, ChenCheng22021} proved the properness conjecture for the existence of constant scalar curvature K\"ahler (cscK) metrics on a compact K\"ahler manifold.
Combining the work of Hisamoto \cite{Hisamoto2016}, Chen-Cheng's result confirmed the existence of cscK metrics under uniform stability conditions for toric manifolds
of any dimension.
There are also some studies on the interior estimate for the Abreu
equation. See the papers Donaldson \cite{Donaldson2005} for $n=2$ and Chen-Han-Li-Sheng \cite{ChenHanLiSheng2014}.
For more studies on K\"ahler geometry of toric manifolds, refer to \cite{LiSheng2023,WangZhu2004,WangZhou2014,ZhouZhu2008} and the references therein.

From the perspective of PDE, the solvability of \eqref{Abreu1} and \eqref{Guillemin-1} remains open for general scalar curvature $R$. The Abreu equation is a fourth-order PDE with singularities on the boundary $\partial P$. Hence it is difficult to solve the Abreu equation.
 Furthermore, due to the singularity of the boundary of $P$,  it is often difficult to derive boundary estimates for the solutions of the relevant equations defined on $P$. Previous results related to \eqref{Abreu1}-\eqref{Guillemin-1} have essentially avoided boundary estimates by utilizing geometric structures of the toric manifolds. Therefore, from the perspective of PDE, both with respect to the analysis on toric manifolds and with respect to the study of the Abreu equations, it is of great significance to establish a purely analytical theory about \eqref{Abreu1}-\eqref{Guillemin-1}.
As a first step, we can view the Abreu equation as a system of two second-order elliptic equations for $(u,\varphi)$,
\begin{eqnarray}
&&\det D^2 u=\varphi^{-1},\label{Abreu-d-1}\\
&& U^{ij}\varphi_{ij}=-2R \label{Abreu-d-2}
\end{eqnarray}
where $U^{ij}$ is the cofactor matrix of $(u_{ij})$. \eqref{Abreu-d-1} is a Monge-Amp\`ere equation and \eqref{Abreu-d-2} is a linearized Monge-Amp\`ere equation. It is natural to investigate \eqref{Abreu-d-1} and \eqref{Abreu-d-2} separately as a first attempt to solving the Abreu equation.
From Guillemin boundary condition \eqref{Guillemin-1}, one knows that $\varphi$ is asymptotic to  $\displaystyle \prod_{i=1}^N l_i(x)$ near $\partial P$. Hence \eqref{Abreu-d-1} coupled with \eqref{Guillemin-1} is exactly \eqref{intro1}-\eqref{intro1-1} as introduced at the beginning of the present paper.
Meanwhile, we will relax the requirement that $P$ is a Delzant polytope and only require $P$ to be a simple convex polytope.
 \par Rubin \cite{Rubin2015} initiated the study of \eqref{intro1}-\eqref{intro1-1} for $n=2$. Rubin first proved that the Dirichlet problem of  \eqref{intro1} admits a unique solution $u\in C^\alpha(\overline{P})$ for a large class of Dirichlet data $u=\phi$ on $\partial P$.  Equations \eqref{intro1} and \eqref{intro1-1} reduce to ODEs on $\partial P$. These ODEs and the values of $u$ on the vertices of $P$  determine the value of $u$ on $\partial P$.  For such boundary values, Rubin showed that the solution
 \begin{equation}\label{Ru-1}
 u-\sum_{i=1}^N l_i(x)\ln l_i(x)\in C^\infty(\overline{P}\backslash \{p_1,\cdots,p_{\hat N}\}).
 \end{equation}
 \eqref{Ru-1} is nearly the Guillemin boundary condition \eqref{intro1-1} except for the vertices of polygon $P$. Therefore, Rubin \cite{Rubin2015} also raised a question: ``whether $ u-\displaystyle \sum_{i=1}^N l_i(x)\ln l_i(x)$ is smooth up to the corners''. This question was recently solved by the first named author \cite{Huang2023}.  It is interesting and important to ask whether \eqref{intro1}-\eqref{intro1-1} is solvable for $n\ge 3$. However,
  due to the complexity of the structure of \eqref{intro1}-\eqref{intro1-1} and the diversity of the singularity types of $\partial P$
 in high dimensions, there has been no progress on the open problem  whether \eqref{intro1}-\eqref{intro1-1} is solvable for $n\ge 3$ until now.  The present paper gives an affirmative answer to this question.
 \par The Guillemin boundary value problem \eqref{intro1}-\eqref{intro1-1} has some interesting features. The first one is that this problem is only well-defined for simple polytopes(see Proposition \ref{prop1}). This phenomenon becomes clear in high dimension $n\ge 3$ since all the convex polygons are simple for $n=2$.
  The second one is that the boundary condition \eqref{intro1-1} is vague without any explicitly prescribed boundary data. We can only prescribe explicit boundary values at the vertices of $P$. However, if one restricts the discussion on $(n-\mathfrak k)-$face of $P$, for $\mathfrak k=1,\cdots,n-1$, then the boundary values should satisfy a Monge-Amp\`ere equation with the Guillemin boundary condition on this $(n-\mathfrak k)-$face(see Lemma \ref{lemcompactible2}).
Therefore, one can build the boundary condition of the problem \eqref{intro1}-\eqref{intro1-1} by solving the Guillemin boundary value problem inductively from edges to $(n-1)-$face of $P$. In other words, the solvability of \eqref{intro1}-\eqref{intro1-1} in dimension $n$ relies on the solvability of \eqref{intro1}-\eqref{intro1-1} in dimensions $1,2,\cdots,n-1$. This seems to be a new phenomenon in the boundary value problem of partial differential equations. The last one is that the singular term for the solution of $u$ is exactly $\displaystyle \sum_{i=1}^N l_i(x)\ln l_i(x)$.   For a singular elliptic equation, it is natural that the solution has singularity. There are also many examples that the singular structures of the solutions are not simple.   In fact, one can find that there are obstructions to the global regularity of these solutions, and the solution may not be smooth up to the boundary even though the boundary of the domain is analytic.  We refer the readers to \cite{HanJiangShen2020,JianLuWang2022,HanJiang2023} and the references therein.

\par The difficulties in solving the Guillemin boundary value problem \eqref{intro1}-\eqref{intro1-1} lie in that the equation is highly singular near the boundary of $P$ and the domain $P$ itself is only Lipschitz. When one considers the problem \eqref{intro1}-\eqref{intro1-1} near some $(n-1)-$face, it is important to investigate the following model equation:
\begin{equation}\label{model-1}
\det D^2 u=\frac 1{x_1},\quad \text{for}\quad x_1>0.
\end{equation}
We can view \eqref{model-1} as a special case of the following type of equation:
\begin{equation}\label{model-2}
\det D^2 u=d_{\partial \Omega}^\alpha,\quad \text{in}\quad \Omega
\end{equation}
where $d_{\partial \Omega}$ is the distance function from $x$ to $\partial \Omega$. In general, the discussion of \eqref{model-2} is restricted to cases where $\Omega$ is uniformly convex. For $\alpha=n$, \eqref{model-2} is related to the global regularity of the Monge-Amp\`ere eigenfunction. We refer readers to  \cite{HongHuangWang2011,Savin2014,LeSavin2017} for more discussion as well as for the cases when $\alpha>0$. For  $\alpha<0$, one cannot expect the global $C^{2}$ estimates. Recently, Savin-Zhang \cite{SavinZhang2020} proved $u\in C^{1,1+\alpha}(\overline{\Omega})$ for $\alpha\in (-1,0)$ and provided some descriptions of the shapes of the level sets of $u$ near the boundary for $\alpha\in (-2,-1]$. $\alpha=-(n+2)$ is another interesting power that corresponds to the hyperbolic affine hypersphere. The global $C^{2,\alpha}$ graph regularity was established by Lin-Wang \cite{LinWang1998} and the optimal graph regularity was obtained by Jian-Wang \cite{JianWang2013} for the hyperbolic affine hypersphere. However, none of the above results can be applied to \eqref{model-1} since $\alpha=-1$ and the boundary is flat. In the work of \cite{Rubin2015}, for $n=2$, this difficulty was overcome by introducing the partial Legendre transformation which transforms \eqref{model-1} into a quasilinear elliptic equation. Then by a suitable perturbation argument, the smoothness of $v=u-x_1\ln x_1$ was derived.  For the high dimensional case, the partial Legendre transformation of  \eqref{model-1} is still a fully non-linear elliptic equation. New ideas are needed for the high dimensional case.

\par The singularity of the domain also makes this problem more difficult. Near the $(n-\mathfrak k)-$face, the model equation for \eqref{intro1} is
\begin{equation}\label{model-3}
\det D^2 u=\left(\prod_{i=1}^{\mathfrak k} x_i\right)^{-1},\quad \text{in}\quad x_1,\cdots,x_{\mathfrak k}>0,\quad 2\le\mathfrak k\le n.
\end{equation}
The singularity of \eqref{model-3} becomes more complex as the boundary of the domain becomes more singular. The literature on the global $C^{2,\alpha}$  regularity for solutions of Monge-Amp\`ere equation in polytopes is rather limited. Recently, Le-Savin \cite{LeSavin2021} studied the global $C^{2,\alpha}$ regularity for the Dirichlet problem of Monge-Amp\`ere equation in a planar polygon. Jhaveri \cite{Jhaveri2019} investigated the the global $C^{2,\alpha}$ regularity for the Monge-Amp\`ere equation with the second boundary condition in the planar polygon domain and the high dimensional case was shown by Chen-Liu-Wang in \cite{ChenLiuWang2023}. Nevertheless, the results of \cite{LeSavin2021,Jhaveri2019,ChenLiuWang2023} are all for uniformly elliptic case and the global smoothness of the solutions is not true in general(see \cite{Jhaveri2019,HuangShen2023}).
The first-named author\cite{Huang2023} gave an affirmative answer to the existence of solutions for \eqref{intro1}-\eqref{intro1-1} under the necessary and sufficient compatibility condition \eqref{comp-c1} when $n=2$. This is a surprising phenomenon. For the uniformly elliptic case, \cite{HuangShen2023} showed that  singular terms will appear repeatedly in the asymptotic expansion of the solutions near the corner in general.  The following theorem is our main result in the present paper.

\begin{theorem}\label{mainthm0}
Let $P$ be a simple convex polytope in $\mathbb R^n$, and let $h(x)\in C^\infty(\overline P)$ be a positive function. Moreover, at any vertex $p_i$ of $P$, the given function $h(x)$ satisfies  the compatibility condition:
\begin{equation}\label{comp-c1}
h(p_i)=\prod_{l_k(p_i)\ne 0}l_k(p_i)\cdot\left[\det\left(Dl_{i_1}\cdots Dl_{i_n}\right) \right]^2(p_i)
\end{equation}
where $l_{i_1}(p_i)=\cdots=l_{i_n}(p_i)=0$.
Then, for any given real numbers $\alpha_1,\cdots,\alpha_{\widehat N}$, \eqref{intro1} and \eqref{intro1-1} admit a unique solution $u$ provided \eqref{comp-c1} is fulfilled and $u(p_i)=\alpha_i$, $i=1,\cdots,\widehat N$.
\end{theorem}
\begin{remark}\label{remark0-1}
It is interesting to note that we only impose the values on vertices to ensure the solvability of problem \eqref{intro1} and \eqref{intro1-1}.
\end{remark}

In the following, let us briefly explain our main ideas of the proof for Theorem \ref{mainthm0}.
It is rather difficult to use the continuity method to solve \eqref{intro1}-\eqref{intro1-1}. However, the existence of an Alexandrov solution of \eqref{intro1} with Dirichlet boundary condition is easy to obtain. By induction and previous discussion on the features of the Guillemin boundary problem, we can obtain an Alexandrov solution $u$ of \eqref{intro1}-\eqref{intro1-1} with $u(p_i)=\alpha_i$, $i=1,\cdots,\widehat N$. The main task for us is to raise the regularity of $u-\displaystyle \sum_{i=1}^N l_i(x)\ln l_i(x)$ from $C^\alpha(\overline{P})$ to $C^\infty(\overline{P})$. Unlike the two dimensional case, the interior regularity for $u$ also needs some considerations due to the well-known counterexample of Pogorelov. It is important to note that the gradient of $u$ blows up on $\partial P$. This gives the strict convexity of the Alexandrov solution $u$ as well as the interior regularity.

For the boundary regularity, we first restrict our discussion near the $(n-1)$-face. Since all our estimates are local, by an affine transformation, we only need to consider
\begin{equation}\label{0-s3-01}
\begin{split}
\det D^2 u=\frac{h(x)}{x_1},\quad \text{in}\quad Q_3=(0,3)\times (-3,3)^{n-1},\\
\end{split}
\end{equation}
and $u(0,x'')$ is a smooth, uniformly convex function solving
\begin{equation}\label{0-s3-02}
\det D^2_{x''} u(0,x'')=h(0,x''),\quad \text{in}\quad Q_3''=(-3,3)^{n-1}.
\end{equation}
In this case, we denote $x''=(x_2,\cdots,x_n)$ and $v(x)=u(x)-x_1\ln  x_1.$ It is sufficient to show that
\begin{equation*}
 v(x)\in C^\infty(\overline{Q_2}).
\end{equation*}
By constructing suitable barrier functions, one can show that $v$ is Lipschitz, i.e.,
\begin{equation*}
  |v(x)-v(0,x'')|\leq Cx_1,\quad \forall x\in Q_2.
\end{equation*}
Then, by a blow-up argument and  the Pogorelov estimate, a weighted $C^{1,1}$-estimate is derived for $v$. By this weighted $C^{1,1}$ estimate,  one can perform partial Legendre transformation to obtain the weighted $C^{2,\alpha}$ estimate. Following the ideas in \cite{JianWang2013} with suitable modifications, the $C^\infty$ regularity of $v$ can be achieved by a bootstrap argument.

\par It remains to consider higher co-dimensional faces.  For any fixed $\mathfrak k\ge 2$, consider the following type of equation:
\begin{equation}\label{0-locpro2-1}
\begin{split}
& \det D^2 u=\frac{h(x)}{\prod_{i=1}^{\mathfrak k} x_{i}},\quad \text{in}\quad Q_{3},\\
& \det D^2_{ij} u \left|_{\{x_l=0,l\in \Gamma,i,j\notin \Gamma\}}=\frac{h}{\prod_{i=1,i\notin \Gamma}^{\mathfrak k}x_i}\right|_{\{x_l=0,l\in \Gamma\}}.
\end{split}
\end{equation}
Here $\Gamma$ can be any subset of $\{1,2,\cdots,\mathfrak k\}$ and $Q_{3}=(0,3)^{\mathfrak k}\times (-3,3)^{n-\mathfrak k}$. Also we make the following convention such that for $\mathfrak k=n$, $\Gamma=\{1,\cdots,n\}$, the second equation of \eqref{0-locpro2-1} means
\begin{equation*}
h(0)=1.
\end{equation*}
 Let
\begin{equation}\label{0-locpro2-2}
v(x)=u(x)-\sum_{i=1}^\mathfrak k x_i\ln  x_i,\quad v(x)\in C^\infty(\overline{Q_{3}}\backslash \{x_1=\cdots=x_\mathfrak k=0\}).
\end{equation}
Assume
\begin{equation}\label{0-locpro2-3}
0<h(x)\in C^\infty(\overline{Q_3}),\quad u(x)\in C(\overline{Q_{3}}),\quad v|_{\partial_0 Q_3} \in C^\infty,
\end{equation}
where
\begin{equation*}
\partial_0 Q_3=\{x\in \partial Q_3| x_i=0, \text{ for some } i=1,\cdots,\mathfrak k\}.
\end{equation*}
For this part of proof, we use notations
\begin{equation*}
x'=(x_1,\cdots,x_\mathfrak k),\quad x''=(x_{\mathfrak k+1},\cdots,x_n)
\end{equation*}
and  assume $v(0',x'')$ is uniformly convex in $x''$.
We aim to prove
\begin{equation*}
v(x)\in C^\infty(\overline{Q_2})
\end{equation*}
in the present case. The equation for $v_{1\cdots \mathfrak k}$ plays an important role in the proof. After a suitable coordinate transformation, the boundary regularity of $v(x)$ can be reduced to the interior regularity of $v_{1\cdots \mathfrak k}$. The key asymptotic estimate is as follows:
\begin{equation}\label{0-asy-k}
|v(x',x'')-F(x',x'')|\le Cx_1\cdots x_{\mathfrak k},
\end{equation}
where $F\in C^\infty(\mathbb R^n)$ is a smooth extension of $v|_{\partial_0 Q_3}$. To obtain asymptotic estimate  \eqref{0-asy-k}, we first need to prove the following weak asymptotic estimate
\begin{equation}\label{0-asy-k-1}
|v(x',x'')-F(x',x'')|\le C  (x_1\cdots x_{\mathfrak k})^\frac{1}{\mathfrak k}+o(|x'|).
\end{equation}
Estimate \eqref{0-asy-k-1} implies $v$ is  Lipschitz up to $(n-\mathfrak k)-$faces. Then, a blow-up argument leads to study of the following Liouville-type theorem.
\begin{theorem}\label{0-thm-liouk}
Let $\mathfrak u\in C(\overline{(\mathbb R^+)^{\mathfrak k}\times \mathbb R^{n-\mathfrak k}})$ be an Alexandrov solution of the following equation
\begin{equation}\label{0-liouk-1}
\begin{split}
& \det D^2 \mathfrak u=\frac{1}{\prod\limits_{a=1}^\mathfrak k x_a},\quad \text{in}\quad (\mathbb R^+)^{\mathfrak k}\times \mathbb R^{n-\mathfrak k},\\
& \mathfrak u|_{\partial((\mathbb R^+)^{\mathfrak k}\times \mathbb R^{n-\mathfrak k})}=\sum_{a=1}^\mathfrak k x_a\ln  x_a+\frac 12 |x''|^2,\\
& \left|\mathfrak u(x)-\sum_{a=1}^{\mathfrak k} x_a\ln  x_a-\frac 12 |x''|^2\right|\le C   (x_1\cdots x_{\mathfrak k})^\frac{1}{\mathfrak k}.
\end{split}
\end{equation}
Then
\begin{equation*}
\mathfrak u(x)=\sum_{a=1}^{\mathfrak k} x_a\ln  x_a+\frac 12|x''|^2.
\end{equation*}
\end{theorem}
The proof of Theorem \ref{0-thm-liouk} relies on the induction of the regularity of $v$ on $\mathfrak k$. For $n=\mathfrak k=2$, Theorem \ref{0-thm-liouk} can be proved by the construction of sub-solutions and super-solutions under a weaker growth condition near infinity (Lemma 2.4 of \cite{Huang2023}).  It would be interesting  whether Theorem \ref{0-thm-liouk} is still true if the last condition of \eqref{0-liouk-1} is replaced by
\begin{equation*}
\left|\mathfrak u(x)-\sum_{a=1}^{\mathfrak k} x_a\ln  x_a\right|=O(|x''|^2)+o(|x'|\ln |x'|),\quad \text{as}\quad |x|\rightarrow +\infty.
\end{equation*}
Theorem \ref{0-thm-liouk} implies the global $C^1$-regularity of $v$. The refinement of the asymptotic estimate for \eqref{0-asy-k-1} essentially relies on  the global $C^1$-regularity of $v$. Notice that
\begin{equation*}
|\nabla(v-F)(x',x'')|=o(1),\quad \text{near}\quad x'=0.
\end{equation*}
One can successfully construct sub-solutions and super-solutions to prove
\begin{equation}\label{0-asy-k-2}
|v(x',x'')-F(x',x'')|\le C\sum_{i\ne j,i,j=1}^{\mathfrak k} x_i x_j.
\end{equation}
 The estimate \eqref{0-asy-k-2} is important for obtaining the following estimates:
\begin{equation}\label{0-refine-growth1-0}|(v-F)(x',x'')|\lesssim C_{\epsilon} \frac{r^{l-1-\epsilon}}{r^{\mathfrak k-1}}x_{1} \cdots x_{\mathfrak k}, \end{equation}
for $\epsilon \in (0,1)$, $2\le l\le \mathfrak k-1$
on $\big(\partial[0,r]^{\mathfrak k}\big)\times[-2,2]^{n-{\mathfrak k}}$. By the estimate \eqref{0-asy-k-2} and the induction method, one can obtain the desired asymptotic estimate \eqref{0-asy-k}.   Then we will combine the  asymptotic estimate \eqref{0-asy-k} with elliptic PDE estimates and use inductive assumptions to improve the regularity of $v$ step by step. The appropriate order of induction and the use of coordinate transformations to find equations with a good structure for the derivative of $v$ play important roles in the proof.

The present paper is organized as follows. Section 2 is devoted to proving the existence of Alexandrov solution and the interior regularity of the Guillemin boundary problem for the Monge-Amp\`ere equation. In Section 3, the $C^\infty$ regularity of $u(x)-\displaystyle \sum_{i=1}^N l_i(x)\ln  l_i(x)$ up to the $(n-1)-$face of $P$ is proved. This estimate also serves as the starting point for our global $C^\infty$ regularity result. In Section 4, by refining asymptotic estimates and using an induction on $\mathfrak k$, we can prove the $C^\infty$ regularity of $u(x)-\displaystyle\sum_{i=1}^N l_i(x)\ln  l_i(x)$ up to the $(n-\mathfrak k)-$face of $P$ for $2\le \mathfrak k\le n$.  Finally, based on the results established in these sections, we prove Theorem \ref{mainthm0}.

\section{Preliminaries}
The first things in this section are to show that the conditions in Theorem \ref{mainthm0} are necessary.
The following proposition reveals that $P$ is a simple convex polytope is necessary.
\begin{prop}\label{prop1}
A necessary condition for the solvability of \eqref{intro1} and \eqref{intro1-1} with $0<h(x)\in C^\infty(\overline{P})$ is that $P$ is a simple convex polytope.
\end{prop}
\begin{proof}
Suppose not. Without loss of generality, we may assume that $0$ is a vertex of $P$ and that 
\begin{equation*}
l_1(0)=\cdots=l_k(0)=0,\quad l_{k+1}(0),\cdots,l_N(0)>0,\quad  k\ge n+1.
\end{equation*}
Let
\begin{equation}\label{decom-0}
u(x)=v(x)+\sum_{i=1}^N l_i(x)\ln  l_i(x),\quad v\in C^\infty(\overline P).
\end{equation}
\par Claim: $\lim\limits_{x\rightarrow 0}(\prod_{i=1}^N l_i(x)) \det D^2 u=0$.
\\ Let $l_i(x)={\bf n_i}\cdot x-c_i$ where ${\bf n_i}$ is a non-zero vector in $\mathbb R^n$. Hence it is easy to see that every term in $\det D^2 u$ has the following form
\begin{equation*}
\delta^{\alpha_1\cdots\alpha_n}_{\beta_1\cdots\beta_n}v_{\alpha_1\beta_1}\cdots v_{\alpha_l\beta_l} \frac{n^{\alpha_{l+1}}_{i_{l+1}} n^{\beta_{l+1}}_{i_{l+1}}}{l_{i_{l+1}}(x)}\cdots\frac{n^{\alpha_{n}}_{i_{n}}n^{\beta_{n}}_{i_{n}}}{l_{i_n}(x)}
\end{equation*}
where $0\le l\le n$, $(\alpha_1,\cdots,\alpha_n),(\beta_1,\cdots,\beta_n)$ are permutations of $(1,\cdots,n)$ and
\begin{equation*}
\delta^{\alpha_1\cdots\alpha_n}_{\beta_1\cdots\beta_n}=\begin{cases} 1,\quad (\beta_1,\cdots,\beta_n) \text{ is an even permutation of } (\alpha_1,\cdots,\alpha_n),\\
-1,\quad  (\beta_1,\cdots,\beta_n) \text{ is an odd permutation of } (\alpha_1,\cdots,\alpha_n).
\end{cases}
\end{equation*}
Now we discuss the terms such that  $i_{l+1},\cdots,i_n$  are not mutually distinct.
Suppose $i_{l+1}=i_{l+2}$. Then
\begin{equation*}
\begin{split}
&\delta^{\alpha_1\cdots\alpha_n}_{\beta_1\cdots\beta_n}v_{\alpha_1\beta_1}\cdots v_{\alpha_l\beta_l} \frac{n^{\alpha_{l+1}}_{i_{l+1}} n^{\beta_{l+1}}_{i_{l+1}}}{l_{i_{l+1}}(x)}\cdots\frac{n^{\alpha_{n}}_{i_{n}}n^{\beta_{n}}_{i_{n}}}{l_{i_n}(x)}\\
+&\delta^{\alpha_1\cdots\alpha_n}_{\beta_1\cdots\beta_{l+2}\beta_{l+1}\cdots\beta_n}v_{\alpha_1\beta_1}\cdots v_{\alpha_l\beta_l} \frac{n^{\alpha_{l+1}}_{i_{l+1}} n^{\beta_{l+2}}_{i_{l+1}}}{l_{i_{l+1}}(x)}\cdot \frac{n^{\alpha_{l+2}}_{i_{l+2}} n^{\beta_{l+1}}_{i_{l+2}}}{l_{i_{l+2}}(x)}\cdots\frac{n^{\alpha_{n}}_{i_{n}}n^{\beta_{n}}_{i_{n}}}{l_{i_n}(x)}=0
\end{split}
\end{equation*}
which follows from
 \begin{equation*}
 n^{\alpha_{l+1}}_{i_{l+1}}n^{\beta_{l+1}}_{i_{l+1}}n^{\alpha_{l+2}}_{i_{l+2}}n^{\beta_{l+2}}_{i_{l+2}}=n^{\alpha_{l+1}}_{i_{l+1}}n^{\beta_{l+2}}_{i_{l+1}}n^{\alpha_{l+2}}_{i_{l+2}}n^{\beta_{l+1}}_{i_{l+2}}, \quad
 \delta^{\alpha_1\cdots\alpha_n}_{\beta_1\cdots\beta_l\beta_{l+2}\beta_{l+1}\cdots\beta_n}=-\delta^{\alpha_1\cdots\alpha_n}_{\beta_1\cdots\beta_n}.
 \end{equation*}
The above argument implies $i_{l+1},\cdots,i_n$ are all distinct  if these terms do not cancel each other out. This implies the Claim which contradicts $h(0)>0$.
\end{proof}
Another necessary condition for the solvability is the relation between the values of $h$ and $l_i$, $i=1,\cdots,N$ at vertices $p_k$.
Due to the affine invariance of \eqref{intro1}, for simplicity in the statements of Lemma \ref{lemcompactible1} and Lemma \ref{lemcompactible2},  we may just assume that
$$l_1(x)=x_1,\cdots,l_n(x)=x_n,$$
and that $0$ is a vertex of $P$.
\begin{lemma}\label{lemcompactible1}
Suppose $u$ solves \eqref{intro1} and \eqref{intro1-1}. Then the following must hold
\begin{equation*}
h(0)=l_{n+1}(0)\cdots l_N(0).
\end{equation*}
\end{lemma}
\begin{proof}
It easily follows from \eqref{intro1} and \eqref{decom-0} that
\begin{equation*}
\begin{split}
h(0)&=\lim\limits_{x\rightarrow 0}\left(\prod_{i=1}^N l_i(x)\right) \det D^2 u\\
&=l_{n+1}(0)\cdots l_N(0).
\end{split}
\end{equation*}
\end{proof}

\begin{remark}
Lemma \ref{lemcompactible1} is exactly the condition \eqref{comp-c1} in Theorem \ref{mainthm0}.
\end{remark}
As pointed out in Remark \ref{remark0-1}, we can only prescribe boundary data on the vertices of $P$. In fact, we can generate boundary data on $\partial P$ by induction, as described in the following lemma.
\begin{lemma}\label{lemcompactible2} Suppose $u$ solves \eqref{intro1} and \eqref{intro1-1}. Let $x'=(x_1,\cdots,x_{k}),x''=(x_{k+1},\cdots,x_n)$ for some $k\in\{1,\cdots,n-1\}$.
Then $v(x'')=u(0',x'')$ solves
\begin{equation*}
\begin{split}
&\det D^2_{x''} v=\frac{h(0',x'')}{\displaystyle \prod_{i=k+1}^{N}l_i(0',x'')},\quad \text{in}\quad \text{int} (\partial P\cap \{x_{1}=\cdots=x_k=0\}),
\\ & v(x'')-\sum_{i=k+1}^N l_i(0',x'')\ln  l_i(0',x'')\in C^\infty(\partial P\cap \{x_{1}=\cdots=x_k=0\}).
\end{split}
\end{equation*}
\end{lemma}
\begin{proof}
It also follows from \eqref{intro1} and \eqref{decom-0} easily.
\end{proof}
\begin{remark}
Lemma \ref{lemcompactible2} indicates that the solvability of $n-$dimensional Guillemin boundary value problem depends on the solvability of $(n-1)$-dimensional Guillemin boundary value problem. This is because the restriction of $u$ on each $(n-1)$-face satisfies a Monge-Amp\`ere equation with the Guillemin boundary condition.
\end{remark}
The above discussion suggests that the Guillemin boundary value problem can be regarded as a Dirichlet boundary value problem with specified boundary data.  Consider
\begin{equation}\label{Dirichlet1}
\begin{cases}
\det D^2 \mathfrak u=\displaystyle \frac{\mathfrak h}{\prod_{i=1}^N l_i},&\quad \text{in}\quad P,\\
\mathfrak u=\varphi,&\quad \text{on}\quad \partial P.
\end{cases}
\end{equation}
\begin{theorem}\label{thmdiri1}
Let $\varphi$ be a  function such that $\varphi-\displaystyle \sum_{i=1}^N l_i\ln  l_i\in C^2(\overline P)$.
Suppose $\varphi$ is convex on each $(n-1)$-face of $P$ and $0<\frac 1{\Lambda}\le \mathfrak h\le \Lambda<+\infty$ for some positive constant $\Lambda>1$. Then there is a unique convex function $\mathfrak u\in C(\overline P)$ solves \eqref{Dirichlet1}.
\end{theorem}
\begin{proof}
Replace $\mathfrak h$ by $\mathfrak h_k$ in \eqref{Dirichlet1}, where $k=1,2,\cdots$, and
\begin{equation*}
\mathfrak h_k=\begin{cases} \mathfrak h,\quad \text{in}\quad P_k=\{x\in P| \text{dist}(x,\partial P)>\frac 1k\},\\
0,\quad \text{in}\quad P\backslash P_k.
\end{cases}
\end{equation*}
Then by Theorem \ref{appthm}, one knows that \eqref{Dirichlet1} admits a unique Alexandrov solution $\mathfrak u_k\in C(\overline P)$ for such $\mathfrak h_k$.
Let $\varphi_h$ be the harmonic extension of $\varphi$. Then it is easy to see that $\mathfrak u_k\le \varphi_h$.

Hence, the remaining task is to demonstrate the existence of a uniform lower barrier. We will just use the same barrier as in \cite{Rubin2015}. For the convenience of the reader, we include the proof here.
Let
\begin{equation*}
\mathcal H(x)=\left(\prod_{i=1}^N l_i(x)\right)^\alpha,\quad \alpha\in \left(0,1\right).
\end{equation*}
Then a direct computation yields that
\begin{equation}
\begin{split}
\frac{\mathcal H_k}{\mathcal H}=\sum_{i=1}^N \frac{\alpha n_i^k}{l_i},\quad \frac{\mathcal H_{km}}{\mathcal H}=\frac{\mathcal H_k\mathcal H_m}{\mathcal H^2}-\sum_{i=1}^N \frac{\alpha n_i^k n_i^m}{l_i^2}=-\sum_{i=1}^N\frac{\alpha n_i^k n_i^m}{l_i^2}+\alpha^2 \sum_{i,j=1}^N\frac{n_i^kn_j^m}{l_il_j}.
\end{split}
\end{equation}
Recall the following simple fact
\begin{equation*}
\mathbb A={\bf a}\otimes {\bf a}+{\bf b}\otimes {\bf b}-{\bf a}\otimes {\bf b}-{\bf b}\otimes {\bf a}\ge 0,
\end{equation*}
where ${\bf a}, {\bf b}$ are two column vectors in $\mathbb R^n$.
Hence for $\alpha>0$ small enough, one knows
\begin{equation}
-\frac{\mathcal H_{km}}{\mathcal H}\ge \frac{\alpha}2 \sum_{i=1}^N\frac{n_i^k n_i^m}{l_i^2}.
\end{equation}
Then it follows from above that
\begin{equation*}
\det (-D^2\mathcal H)\ge C\left(\prod_{i=1}^N l_i\right)^{n\alpha-2}\ge C\left(\prod_{i=1}^N l_i\right)^{-1}
\end{equation*}
provided $\alpha<1/n$.

Let $H(x)=\varphi-A\mathcal H$ for some positive $A$ large enough such that
\begin{equation*}
\det D^2  H\ge \det D^2 \mathfrak u\ge\det D^2  \mathfrak u_k.
\end{equation*}
By maximum principle, one obtains $\mathfrak u_k\ge H$. This implies the present theorem.
\end{proof}
For $n=2$, the strict convexity of $\mathfrak u$ directly follows  from the positivity of Monge-Amp\`ere measure in the interior. However, Pogorelov's counterexample shows that this is not true in general dimensions $n\ge 3$. Hence, we need the following interior regularity theorem.
\begin{theorem}\label{thm2.2}
Suppose all the assumptions in Theorem \ref{thmdiri1} are fulfilled.
Moreover, $\varphi$ is strictly convex on each $(n-1)$-face of $P$ and $\mathfrak h\in C^\infty(P)$. Then $\mathfrak u\in C^\infty(P)$.
\end{theorem}
\begin{proof}
It is enough to show $\mathfrak u$ is strictly convex in $P$.
\par Step 1.   $\mathfrak u$ is not Lipschitz up to the interior of an $(n-1)$-face. Let $p=(0,p'')$ be an interior point on $\Sigma=\{x_1=0\}\cap \partial P$, i.e., $B_r(p)\cap \partial P\subset\subset \Sigma$ for some small $r>0$. Here, $p''=(p_2,\cdots,p_n)$.
Without loss of generality, we may assume $\varphi(p)=0$ and $\varphi\ge 0$ on $\Sigma$. By the strict convexity of $\varphi$ on $\Sigma$ and the continuity of $\mathfrak u$(Theorem \ref{thmdiri1}), we may also assume
\begin{equation*}
2\varphi(0,x'')\ge \mathfrak u(x)\ge \varepsilon_0>0,\quad \text{on}\quad [0,\delta]\times (\partial B_r(p)\cap \Sigma)
\end{equation*}
for some small $\delta>0$.
Let
\begin{equation*}
\mathcal H(x)=4\varphi(0,x'')-\varepsilon x_1(-\ln  x_1)^{1/2}+\frac 1{\varepsilon} x_1.
\end{equation*}
Then it is easy to see
\begin{equation*}
\mathcal H(x)\ge \mathfrak u(x),\quad \text{on}\quad \partial([0,\delta] \times  (B_r(p)\cap \Sigma))
\end{equation*}
provided $\varepsilon>0$ is small enough.
A direct computation also yields
\begin{equation*}
\begin{split}
&\det D^2\mathcal H\\
=&\frac{4^{n-2}\varepsilon \det D^2_{x''}\varphi(0,x'')}{x_1(-\ln  x_1)^{\frac 12}} (2+(-\ln  x_1)^{-1})\\
\le& \det D^2 \mathfrak u,
\end{split}\quad \text{in}\quad (0,\delta)\times (B_r(p)\cap \Sigma)
\end{equation*}
for $\varepsilon>0$ small enough by noticing that
\begin{equation*}
\det D^2 \mathfrak u\approx \frac{1}{x_1},  \text{ near } p.
\end{equation*}
 Hence, by the standard maximum principle, one gets
 \begin{equation*}
 \mathcal H(x)\ge \mathfrak u(x), \text{ in }   [0,\delta]\times (B_r(p)\cap \Sigma).
 \end{equation*}
 This implies
 \begin{equation*}
 \mathfrak u(x_1,p'')\le \mathcal H(x_1,p'')=-\varepsilon x_1(-\ln  x_1)^{\frac 12}+\frac 1\varepsilon x_1,\quad 0<x_1<\delta.
 \end{equation*}
 Since  $(0,p'')$ is arbitrary in the interior of $\Sigma$, for $x_1$ small enough, one also gets
\begin{equation*}
u(x)\le u(0,x'')-Cx_1(-\ln  x_1)^{1/2},\quad \forall (0,x'')\in K\subset\subset \Sigma.
\end{equation*}
Here the constant $C$ depends only on the quantities in Theorem \ref{thmdiri1},  $\text{dist}(K,\partial\Sigma)$ and  the modulus of convexity of $\varphi$ on $\Sigma$.
\par Step 2. $\mathfrak u$ is strictly convex in $P$. Suppose not. By  [Theorem 1, \cite{Caffarelli1990-1}], one knows
$\{u=0\}$ will contain at least a line segment $L$ whose end  points $p,q$ are on $\partial P$. Without loss of generality, we may assume $p\in \Sigma=\{x_1=0\}\cap \partial P$ and the line segment $\overline{pq}=\{t\vec a+p| t\in [0,|\overline{pq}|]\}$ for some unit vector $\vec a=(a_1,\cdots,a_n)$. We have the following two cases.
\begin{itemize}
\item[1.] $p\in \mathring \Sigma.$  Let $p_t=t\vec a+p$.
Then for $t$ small, $(0,p_t'')\in $ $\mathring\Sigma$.
From Step 1, we know
\begin{equation*}
\begin{split}
0=&u(p_t)-u(p)=u(p_t)-u(0,p_t'')+u(0,p_t'')-u(p)\\
&\le -Ct(-\ln  t)^{\frac 12}+C_1 t<0
\end{split}
\end{equation*}
for $t$ small enough. This yields a contradiction.
\item[2.] $p\in \partial\Sigma$. Without loss of generality, we may assume $p=(\underbrace{0,\cdots,0}_k,p'')=(0',p'')$, $2\le k\le n$.
Then one knows
\begin{equation*}
\begin{split}
u(p_t)\le &  \frac{1}{k} \sum_{i=1}^k u[(0',p_t'')+kp_{t,i}{\bf e_i}] \\
\le & u(0',p_t'')+\sum_{i=1}^k\left(Cp_{t,i}+p_{t,i}\ln p_{t,i}\right)\\
\le & u(p)+C |p_t''-p''|+\sum_{i=1}^k\left(Cp_{t,i}+p_{t,i}\ln p_{t,i}\right)\\
\le& u(p)+C_1t\ln  t+Ct<u(p)
\end{split}
\end{equation*}
for $t$ small.
\end{itemize}
The above argument implies $u$ is strictly convex in $P$. It then follows from standard theory of Monge-Amp\`ere equation(Theorem 3.10, \cite{Figalli2017book}), $\mathfrak u\in C^\infty(P)$.
\end{proof}

\section{Regularity up to the $(n-1)$-face of $P$}
\subsection{Lipschitz regularity up to the $(n-1)$-face of $P$}
By  Section 2, we will restrict our discussion to the following type of equation:
\begin{equation}\label{s3-01}
\det D^2 u=\frac{h(x)}{x_1},\quad \text{in}\quad Q_3=(0,3)\times (-3,3)^{n-1}
\end{equation}
and $u(0,x'')$ is a smooth, uniformly convex function that solves
\begin{equation}\label{s3-02}
\det D^2_{x''} u(0,x'')=h(0,x''),\quad \text{in}\quad Q_3''=(-3,3)^{n-1}.
\end{equation}
Throughout this section, we denote $x''=(x_2,\cdots,x_n)$.

The main task in the present section is to show
\begin{equation*}
 v(x)\in C^\infty(\overline{Q_2})
\end{equation*}
where $v(x)$ is given by
\begin{equation}\label{s3-03}
v(x)=u(x)-x_1\ln  x_1.
\end{equation}

\begin{lemma}\label{lemlowb} Suppose $u,v\in C^\infty(Q_3)\cap C(\overline{Q_3})$ are given by \eqref{s3-01}-\eqref{s3-03}. Let $\Lambda$ be the positive constant such that
\begin{equation*}
\frac 1\Lambda \le h(x)\le \Lambda,\quad x\in Q_3,\quad \frac 1\Lambda  \mathbb I_{(n-1)\times(n-1)}\le D_{x''}^2 u(0,x'')\le \Lambda \mathbb I_{(n-1)\times(n-1)},\quad x''\in Q_3''.
\end{equation*}
Then
	\begin{equation*}
  |v(x)-v(0,x'')|\leq Cx_1,\quad \forall x\in Q_2,
\end{equation*}
   for some positive constant C depending only on
   \begin{equation}\label{lipq}
   \begin{split}
     \Lambda, \   M=\max(\|u\|_{C^0(\overline{Q_3})},\  \|h\|_{C^1(\overline{Q_3})}),\  \omega_{u(x)},
   \end{split}
   \end{equation}
   where $\omega_u(x)$ is the modulus of continuity of $u$ in $\overline Q_3$.
\end{lemma}
\begin{proof}
	The proof of Lemma \ref{lemlowb} will be divided into three steps.
	\par Step 1. The following estimate holds:
  \begin{equation}\label{1146}
  v(x)\ge v(0,x'')+A x_1\ln  x_1, \quad \forall x\in (0,1/2)\times (-5/2,5/2)^{n-1}.
\end{equation}
 Here, the positive constant $A$ depends only on the quantities in \eqref{lipq}.
Without loss of generality, it suffices to prove \eqref{1146} at $0$ with $u(0)=|D_{x'} u(0)|=0$.
Let
$$H(x)=\varepsilon_0 u(0,x'')+C_0\varepsilon_0^{-(n-1)}x_1\ln  x_1.$$
Then, by direct computation, we obtain
\begin{equation*}
\det D^2 H=\frac{C_0 h(0,x'')}{x_1}\ge \frac{h(x)}{x_1}
\end{equation*}
provided $C_0\ge \Lambda^2$.
\par By the strict convexity of $u(0,x'')$ on $x_1=0$, there holds
\begin{equation*}
	u(0,x'')\ge c_0\ge 4\sigma_0,\quad \forall x''\in\partial ([-3,3]^{n-1})
\end{equation*}
for some positive constant $\sigma_0$.
Hence, by the continuity of $u$, we can choose $\delta_0>0$ small enough such that
\begin{equation}\label{mod-con-b}
|u(x)-u(0,x'')|\le \sigma_0,\quad x\in [0,\delta_0]\times \partial ([-3,3]^{n-1}).
\end{equation}
\begin{remark}\label{mod-con-a}
This is the only place where we require the modulus of continuity of $u$ in $Q_3$. In fact, from the proof, the only requirement is the existence of a universal positive constant $\delta_0$ such that \eqref{mod-con-b} holds, which is a weaker condition than $u\in C(\overline{Q_3})$ with a uniform modulus of continuity.
\end{remark}
Fixing such $\delta_0$, we obtain
\begin{equation*}
	u(x)\ge H(x),\quad  x\in [0,\delta_0]\times \partial ([-3,3]^{n-1})
\end{equation*}
for $0<\varepsilon_0<\frac 34$.
Also,
\begin{equation*}
u(x)\ge H(x),\quad \text{on} \quad  \{0\}\times [-3,3]^{n-1}
\end{equation*}
follows from $0<\varepsilon_0<1$.
Moreover, we have
\begin{equation*}
u(x)\ge H(x),\quad \text{on}\quad \{\delta_0\}\times [-3,3]^{n-1}
\end{equation*}
for $0<\varepsilon_0\le \left(\frac{C_0\delta_0\ln (1/\delta_0)}{2M}\right)^{\frac 1{n-1}}$.
Then by the standard maximum principle, we prove Step 1.
\par Step 2. Now we improve the estimate in Step 1 to the desired result. Let
$$\mathcal H(x)=u(0,x'')+x_1\ln x_1+A |x''|^2 x_1\ln  x_1-B x_1(1-x_1^{\frac 13}),\quad\text{in}\quad (0,\delta)\times (-5/2,5/2)^{n-1}.$$
Here $0<\delta\le \delta_0$ is a  small positive constant to be determined later, and $A$ is the positive constant determined in Step 1. Throughout this section, we use the following convention. For any function $\mathcal F$, denote $\mathbb M_{\mathcal F}$ by
\begin{equation*}
\mathbb M_{\mathcal F}=  \begin{pmatrix}
  x_1D_{x_1}^2 \mathcal F   &    \sqrt{x_1} D_{x_1x''}\mathcal F \\[3pt]
  \sqrt{x_1} (D_{x_1x''}\mathcal F)^T &  D_{x''}^2 \mathcal F
 \end{pmatrix}.
\end{equation*}
A direct computation yields that
\begin{equation*}
\mathbb M_{\mathcal H}=\mathbb M_{u(0,x'')+x_1\ln x_1}+2A(x_1\ln  x_1) \mathbb E+\mathbb A+\mathbb B,
\end{equation*}
where
\begin{equation*}
\begin{split}
& \mathbb E= \sum_{i=2}^n {\bf e_i}\otimes {\bf e_i},\\
&\mathbb A=2A\sqrt{x_1}(1+\ln  x_1)\left[(0,x'')\otimes {\bf e_1}+{\bf e_1}\otimes (0,x'')\right],\\
& \mathbb B=\left(A|x''|^2+\frac{4}{9}Bx_1^{\frac 13}\right){\bf e_1}\otimes {\bf e_1}.
\end{split}
\end{equation*}
By the assumptions on $u(0,x'')$, we know
\begin{equation*}
\frac{1}{C_0}\mathbb I_{n\times n}\le \mathbb M_{u(0,x'')+x_1\ln x_1}\le C_0\mathbb I_{n\times n}
\end{equation*}
for some universal positive constant $C_0$. Hence, we can choose $\delta$ small enough such that
\begin{equation*}
\frac{1}{2C_0}\mathbb I_{n\times n}\le \mathbb M_{u(0,x'')+x_1\ln x_1}+2A(x_1\ln  x_1) \mathbb E+\mathbb A \le 2C_0\mathbb I_{n\times n}.
\end{equation*}
Then,
\begin{equation*}
\begin{split}
\det \mathbb M_{\mathcal H} & \ge \det (\mathbb M_{u(0,x'')+x_1\ln x_1}+2A(x_1\ln  x_1) \mathbb E+\mathbb A)+\frac{2\left(|x''|^2+ Bx_1^{\frac 13}\right)}{C_1}\\
&\ge \det \mathbb M_{u(0,x'')+x_1\ln x_1}+\frac{2Bx_1^{\frac 13}}{C_1}-C_2 x_1^{\frac 12}|\ln  x_1|\\
&\ge h(0,x'')+\frac{Bx_1^{\frac 13}}{C_1}
\\ &\ge h(x)+x_{1}^{\frac 13}\left(\frac{B}{C_1}-C_3 x_1^{\frac 23}\right)\ge \det \mathbb M_{u}.
\end{split}
\end{equation*}
In obtaining the above inequality, it suffices to choose
\begin{equation*}
B\ge C_1C_3\delta_0^{\frac 16}|\ln \delta_0|.
\end{equation*}
It remains to compare the boundary values of $\mathcal H$ and $u$.
\begin{itemize}
	\item[I.] On $\{0\}\times [-3,3]^{n-1}$, one has $u(0,x'')=\mathcal H(0,x'')$.
	\item[II.]
	On $\{\delta\}\times [-3,3]^{n-1}$, by \eqref{1146}, one has
	\begin{equation*}
  u(\delta,x'')-\mathcal H(\delta,x'')\ge B\delta(1-\delta^{\frac 13})-A\delta|\ln \delta|>0
		\end{equation*}
	provided
	\begin{equation}\label{B2}
		B\ge \frac{A|\ln \delta|}{1-\delta^{\frac 13}}.
		\end{equation}
	\item[III.] On $(0,\delta)\times \partial ([-3,3]^{n-1})$, it follows from Step 1 that $u\ge \mathcal H$.
\end{itemize}
 This proves
 \begin{equation*}
 v(x_1,0)\ge -Bx_1,\quad x_1\in (0,\delta).
 \end{equation*}
Hence also
\begin{equation*}
v(x_1,x'')\ge v(0,x'')-C x_1,\quad \forall x\in Q_2.
\end{equation*}
Step 3. Replace $H$ in Step 1 by
\begin{equation*}
H=\varepsilon_0^{-1}u(0,x'')+Ax_1+\frac{\varepsilon_0^{n-1}}{C_0}x_1\ln  x_1,\quad \text{in} \quad [0,\delta_0]\times [-3,3]^{n-1}
\end{equation*}
for $\varepsilon_0,\delta>0$ small enough and $A,C_0>0$ large enough.
 Then, replacing $\mathcal H$ in Step 2 by
 \begin{equation*}
 \mathcal H(x)=u(0,x'')+x_1\ln  x_1+\left(\frac{\varepsilon_0^{n-1}}{C_0}-1\right)\frac{|x''|^2}{9}x_1\ln  x_1+Bx_1(1-x_1^{\frac 13})
 \end{equation*}
 for some constant $B$ large enough, one gets
 \begin{equation*}
v(x_1,x'')\le v(0,x'')+C x_1,\quad \forall x\in Q_2.
\end{equation*}
This ends the proof of present lemma.
\end{proof}

\subsection{$C^{2,\alpha}$ regularity up to the $(n-1)$-face of $P$}

\begin{lemma}\label{lemc1-1}
Let $u,v$ be given as in Lemma \ref{lemlowb}, and assume all the conditions in Lemma \ref{lemlowb} are satisfied.
Then there holds
\begin{equation*}
x_1|v_{11}|+\sqrt {x_1}\sum_{i=2}^{n}|v_{1i}|+\sum_{i,j=2}^{n}|v_{ij}|\le C,\text{ in } (0,r_0)\times (-1,1)^{n-1},
\end{equation*}
where the constants $C$ and $r_0$ depend only on the quantities in Lemma \ref{lemlowb}, $\|u(0,x'')\|_{C^2(\overline{Q_3})}$ and $\|h\|_{C^2(\overline{Q_3})}$.
\end{lemma}
\begin{proof}
Without loss of generality, we may assume  $u(0)=|D_{x''}u(0)|=0$. Consider the point $(\lambda,0'')$, where $\lambda\in (0,1/2)$.
Denote
\begin{equation*}
u_\lambda(x)=\frac{u(\lambda x_1,\sqrt{\lambda} x'')-(\lambda \ln  \lambda)x_1}{\lambda},\quad 0<\lambda<1.
\end{equation*}
Then
\begin{equation*}
v_{\lambda}(x)\triangleq u_\lambda(x)-x_1\ln  x_1=\frac{v(\lambda x_1,\sqrt{\lambda} x'')}{\lambda}.
\end{equation*}
Also, $u_\lambda(x)$ solves
\begin{equation}\label{3-213}
\det D^2 u_{\lambda}=\frac{h(\lambda x_1,\sqrt \lambda x'')}{x_1}, \text{ in } (0,1/\lambda)\times \left(-\frac{2}{\sqrt \lambda},\frac{2}{\sqrt \lambda}\right)^{n-1}.
\end{equation}

By the discussion in Subsection 3.1, one knows $v(x)$ is Lipschitz up to the boundary. Hence,
\begin{equation*}
\begin{split}
|v_{\lambda}(x)|\le &\left|\frac{v(\lambda x_1,\sqrt\lambda x'')-v(0,\sqrt\lambda x'')}{\lambda}\right|+\left|\frac{v(0,\sqrt\lambda x'')}{\lambda}\right|\\
\le & C(x_1+|x'|^2),\quad \text{in}\quad 0<x_1<\frac 1{2\lambda},\quad |x''|\le \frac{1}{2\sqrt \lambda}
\end{split}
\end{equation*}
for some universal constant $C$ depending only on the quantities in Lemma \ref{lemlowb} and $\|u(0,x'')\|_{C^2(\overline{Q_3''})}$.  Let $l_\lambda(x)$ be the supporting plane of $u_\lambda$ at the point $(1,0'')$.
\par Claim: the strict convexity of $u_\lambda$ at $(1,0'')$ is independent of $\lambda$ for $\lambda$ small enough.
\\ Suppose not. Up to a subsequence, assume $l_\lambda\rightarrow l$, $u_\lambda\rightarrow \mathfrak u$ locally uniformly in $\mathbb R^n_+$, and $\Sigma=\{\mathfrak u=l\}$ is not a single point.
By  [Theorem 1, \cite{Caffarelli1990-1}], we distinguish with the following two cases.
\begin{itemize}
\item[(1).] $\Sigma$ contains a line in $\mathbb R^n_+$. This implies $\det D^2 \mathfrak u=0$, which is impossible.
\item[(2).] $\Sigma$ contains a ray $L$ with end point on $\partial\mathbb R^n_+$. Without loss of generality, let
\begin{equation*}
L=\{x|x_1=t,x_2=c_1t,\cdots,x_{n}=c_{n-1}t,\  t\ge 0\}
\end{equation*}
for some vector $\vec c\in \mathbb R^{n-1}$.
Then on $L$, there holds
\begin{equation}
\begin{split}
0=&|\mathfrak u(t,c_1t,\cdots,c_{n-1}t)-\mathfrak u(0)|\\
\ge & |\mathfrak u(t,c_1t,\cdots,c_{n-1}t)-\mathfrak u(0,c_1t,\cdots,c_{n-1}t)|-|\mathfrak u(0,c_1t,\cdots,c_{n-1}t)-u(0)|\\
 \gtrsim & O(t|\ln t|)-O(t^2)>0
\end{split}
\end{equation}
for $t$ small enough. This yields a contradiction.
\end{itemize}
This proves the present claim.
\begin{remark}
The above claim  holds for all $u$ as long as the quantities in Lemma \ref{lemlowb} and $\|u(0,x'')\|_{C^2(\overline{Q_3})}$ are uniformly bounded.
\end{remark}
\par Hence, there exists $\varepsilon_0>0$ such that
\begin{equation*}
S_{\lambda,4\varepsilon_0}=\{u_\lambda<l_\lambda+4\varepsilon_0\}\subset\subset \mathbb R^{n}_+.
\end{equation*}
Applying Pogorelov's estimate([Theorem 17.19,\cite{GilbargTrudinger2001}]) to $u_\lambda$ yields
\begin{equation*}
|D^2 u_\lambda|\le C_1,\quad \text{in}\quad S_{\lambda,2\varepsilon_0}
\end{equation*}
for some positive constant $C_1$ depending only on the quantities in Lemma \ref{lemlowb}, $\|u(0,x'')\|_{C^2(\overline{Q_3})}$, and $\|h\|_{C^2(\overline{Q_3})}$. Scaling back to $u$, one gets
\begin{equation}\label{3-319}
\lambda |u_{11}(\lambda,0'')|+\sqrt {\lambda }\sum_{i=2}^{n}|u_{i1}(\lambda,0'')|+\sum_{i,j=2}^{n}|u_{ij}(\lambda,0'')|\le C.
\end{equation}
This ends the proof of present lemma.
\end{proof}
For a bounded domain $\Omega \subset\mathbb R^n_+$, define the norm $\|\cdot\|_{\widetilde C^{2,\alpha}(\overline{\Omega})}$ as follows:
\begin{equation*}
\|f\|_{\widetilde C^{2,\alpha}(\overline{\Omega})}=\|f\|_{C^{1}(\overline{\Omega})}+  \|x_1f_{11}\|_{C^{\alpha}(\overline{\Omega})}+\sum_{i=2}^{n}\|\sqrt {x_1} f_{i1}\|_{C^{\alpha}(\overline{\Omega})}+\sum_{i,j=2}^{n}\|f_{ij}\|_{C^{\alpha}(\overline{\Omega})}.
\end{equation*}
\begin{theorem}\label{thm3-349}
Let $u,v$ be given as in Lemma \ref{lemlowb}, and assume all the conditions in Lemma \ref{lemlowb} are satisfied. Then, the following estimate holds:
\begin{equation*}
\|v\|_{\widetilde C^{2,\alpha}(\overline{(0,\frac {r_0}2)\times (-1/2,1/2)^{n-1}})}\le C,
\end{equation*}
where $\alpha\in (0,1/2)$ and $C>0$ depend only on the quantities in Lemma \ref{lemc1-1}.
\end{theorem}
\begin{proof}
By Lemma \ref{lemc1-1}, one knows
\begin{equation}\label{3-648}
D^2_{x''} u\ge c_0 \mathbb I_{(n-1)\times (n-1)},\quad \text{in}\quad (0,r_0)\times (-1,1)^{n-1}
\end{equation}
for some positive constant $c_0$. Consider the partial Legendre transformation of $u$:
\begin{equation*}
\begin{split}
& y_1=x_1,\quad y''=D_{x''} u,\\
&u^*(y)=x''\cdot D_{x''} u(x)-u(x).
\end{split}
\end{equation*}
Define
\begin{equation*}
\begin{split}
T:\ &(0,r_0)\times (-1,1)^{n-1} \rightarrow T((0,r_0)\times (-1,1)^{n-1})=\Omega,\\
&\quad\quad \quad\quad \quad \quad\quad \ x\mapsto y=T(x)=(x_1,D_{x''}u).
\end{split}
\end{equation*}
Then, $T,T^{-1}$ are both Lipschitz homeomorphisms between $(0,r_0)\times (-1,1)^{n-1}$ and $\Omega$. By direct computation, $u^*$ solves
\begin{equation}\label{3-752}
\begin{split}
& y_1 u^*_{11}+h(y_1,D_{y''}u^*)\det D^2_{y''} u^*=0,\quad \text{in}\quad \Omega\\
& \det D_{y''}^2 u^*(0,y'')=\frac{1}{h(0,D_{y''}u^*)},\quad \text{on}\quad \partial\Omega\cap \{y_1=0\}.
\end{split}
\end{equation}
Let $v^*(y)=u^*+y_1\ln  y_1$. Since $u^*_{y_1}=-u_{x_1}$, it follows that
\begin{equation*}
v^{*}_{y_1}=u^*_{y_1}+1+\ln  y_1=-u_{x_1}+1+\ln  x_1=-v_{x_1}\in L^\infty(\Omega).
\end{equation*}
Let $w^*=v_1^*$.
Differentiating \eqref{3-752} with respect to $y_1$ yields
\begin{equation*}
 \partial_1(y_1 w^*_{1})+\sum_{i,j=2}^{n}\partial_i(a_{ij}\partial_j w^*)=f,\quad \text{in}\quad \Omega,
\end{equation*}
where
\begin{equation*}
\begin{split}
a_{ij}=(U^*)^{ij} h,\quad f=-h_{1}\det D^2_{y'' } u^*\in L^\infty(\Omega),
\end{split}
\end{equation*}
and $(U^*)^{ij}$ is the co-factor matrix of $D_{y''}^2 u^*$.
By [Theorem 1.1,\cite{HongHuang2022}],
\begin{equation*}
v_1^*=w^*\in C^{\alpha}_{loc}(\Omega\cup\{y_1=0\})
\end{equation*}
 for some $\alpha\in(0,1)$. Also let
 $$x_1=t^2/4,\quad \mathfrak u(t,x'')=u(x).$$
 Then,
\begin{equation*}
\begin{split}
& |D_{x''}\mathfrak u(t,x'')-D_{x''}\mathfrak u(\tilde t,\tilde x'')|\\
\le &\int_0^1 |D_{t,x''}D_{x''} \mathfrak u(s(\tilde t,\tilde x'')+(1-s)(t,x'') )\cdot (t-\tilde t,x''-\tilde x'')| ds\\
\le & C(|x''-\tilde x''|+|t-\tilde t|).
\end{split}
\end{equation*}
In deriving the last inequality of the above, we used  Lemma \ref{lemc1-1} to obtain
\begin{equation*}
\mathfrak u_{ij}=u_{ij},\quad \mathfrak u_{1i}=\sqrt{x_1} u_{1i}\in L^\infty((0,r_0)\times (-1,1)^{n-1}),\quad i,j=2,\cdots,n.
\end{equation*}
Hence,
\begin{equation*}
\begin{split}
& |v_{x_1}(x)-v_{x_1}(\tilde x)|\\
= &|v^*_{y_1}(y)-v^*_{y_1}(\tilde y)| \\
\le & C\left(|x_1-\tilde x_1|^\alpha+|D_{x''}u(x)-D_{x''}u(\tilde x)|^\alpha\right)\\
\le & C\left(|x_1-\tilde x_1|^\alpha+|x''-\tilde x''|^\alpha+|\sqrt x_1-\sqrt{\tilde x_1}|^{\alpha}\right).
\end{split}
\end{equation*}
Let $\mathfrak v(z)=v(x)$, $x_1=\frac{z_1^2}{4},x''=z''$. Then $\mathfrak v$ solves

\begin{equation}\label{MAs2}
\mathcal F\left(D^2 \mathfrak v,\frac{\mathfrak v_{z_1}}{z_1},z\right)=\det \begin{pmatrix}
    \mathfrak v_{z_1z_1}-\frac{\mathfrak v_{z_1}}{z_1}+1&    \mathfrak v_{z_2z_1} & \cdots &  \mathfrak v_{z_nz_{1}}\\[3pt]
   \mathfrak v_{z_2z_1}& \mathfrak v_{z_2z_2} & \cdots&\mathfrak v_{z_2z_{n}}\\[3pt]
  \cdots &\cdots &\cdots &\cdots  \\[3pt]
  \mathfrak v_{z_nz_{1}} &\mathfrak v_{z_2z_{n}} & \cdots&\mathfrak v_{z_{n}z_{n}}
 \end{pmatrix}
    -h(z_1^2/4,z'')=0.
\end{equation}
Note that $\mathfrak v_{z_1}(z_1,z'')/z_1=\frac 12 v_{x_1}.$ Hence, we can make even extension of $\mathfrak v(z)$, i.e., $\mathfrak v(z)=\mathfrak v(-z_1,z'')$ such that
\begin{equation*}
\mathfrak v_{z_1}(z_1,z'')/z_1\in C^{\alpha}([-\sqrt{3r_0},\sqrt{3r_0}]\times [-3/4,3/4]^{n-1}).
\end{equation*}
Since $\mathcal F^{\frac 1n}$ is concave in $D^2 \mathfrak v$, by interior  $C^{2,\alpha}$ regularity \cite{Caffarelli1989-1}, one knows \begin{equation*}
\mathfrak v\in C^{2,\alpha}([-\sqrt{2r_0},\sqrt{2r_0}]\times [-1/2,1/2]^{n-1}).
\end{equation*}
This ends the proof of present theorem.
\end{proof}
\subsection{$C^\infty$ regularity up to the $(n-1)$-face of $P$}
Let $\mathfrak v(z)$ be as in the previous sub-section.
Following the ideas in \cite{JianWang2013} with suitable modifications, we can prove the following theorem.
\begin{theorem}\label{thm3-803}
Let $u,v$ be given as in Lemma \ref{lemlowb}, and assume that all the conditions in Lemma \ref{lemlowb} are satisfied.
If $h\in C^\infty(\overline{Q_3})$, then $v\in C^\infty(\overline{(0,\frac {r_0}4)\times (-1/4,1/4)^{n-1}}))$. Moreover, the following estimate holds
\begin{equation*}
\|v\|_{C^{k+2}(\overline{(0,\frac {r_0}4)\times (-1/4,1/4)^{n-1}}))}\le C_k
\end{equation*}
for some positive constant $C_k$ depending only on the quantities in Lemma \ref{lemlowb}, $\|u(0,x'')\|_{C^{k+4}(\overline{Q_3''})}$ and $\|h\|_{C^{k+2}(\overline{Q_3})}$.
\end{theorem}
\begin{proof}
Let $\mathfrak v$ be  given in Theorem \ref{thm3-349}, which solves \eqref{MAs2}.
Differentiating \eqref{MAs2} with respect to $z_2$ yields that $\mathfrak w=\mathfrak v_{z_2}$ solves
\begin{equation*}
a_{11}\left(\mathfrak w_{11}-\frac{\mathfrak w_1}{z_1}\right)+\sum_{i+j\ge 3}a_{ij}\mathfrak w_{ij}=h_{2},
\end{equation*}
where $a_{ij}$ is the cofactor matrix of the matrix in \eqref{MAs2}.
Define
\begin{equation*}
\begin{split}
&\mathcal L(\varphi)=a_{11}\left(\varphi_{11}-\frac{\varphi_1}{z_1}\right)+\sum_{i+j\ge 3}a_{ij}\varphi_{ij},\\
&\mathcal L_0(\varphi)(z)=a_{11}(0,z'')\left(\varphi_{11}-\frac{\varphi_1}{z_1}\right)+\sum_{i+j\ge 3}a_{ij}(0,z'')\varphi_{ij}.
\end{split}
\end{equation*}
\par {\bf Step 1.} The following estimate holds
\begin{equation*}
\left|\mathfrak w(z)-\mathfrak w(0,z'')\right|\le Az_1^2|\ln  z_1|,\quad \text{in}\quad (0,r_0/2)\times [-3/8,3/8]^{n-1}.
\end{equation*}
Denote
\begin{equation*}
H(z)=\mathfrak w(0,z'')-A z_1^2\ln  z_1+B|z''|^2.
\end{equation*}
By Theorem \ref{thm3-349}, $\mathfrak v\in C^{2,\alpha}([0,r_0/2]\times [-1/2,1/2]^{n-1})$ and $\mathfrak v_{z_1}(0,z'')=0$, one knows
$a_{i1}(0,z'')=0$, $i=2,\cdots,n$. It follows from \eqref{s3-02} that $\mathfrak w(0,z'')$ satisfies
\begin{equation*}
\mathcal L_0(\mathfrak w(0,z''))=h_2(0,z'').
\end{equation*}
Then
\begin{equation*}
\begin{split}
&\mathcal L(H)-h_{2}\\
= & (\mathcal L-\mathcal L_0) w(0,z'')-2Aa_{11}+2B\sum_{i=2}^{n}a_{ii}+h_{2}(0,z'')-h_2(z)\\
\le & -\frac{A}{C_0}+C_0B+C_0 z_1^\alpha\\
\le & 0.
\end{split}
\end{equation*}
In obtaining the above inequality, we used the fact that $a_{ij}(z)\in C^\alpha([0,\sqrt{2r_0}]\times [-1/2,1/2]^{n-1})$ and choose $A\ge C_0^2(B+1)$.
It can be seen directly that $H(0,z'')\ge \mathfrak w(0,z'')$.
We can then choose $B$ large enough such that
\begin{equation*}
H(z)\ge \mathfrak w(z),\quad z\in (0,\sqrt{2r_0})\times \partial([-1/2,1/2]^{n-1}).
\end{equation*}
Next, we choose $A$ large enough such that
\begin{equation*}
H(z)\ge \mathfrak w(z),\quad z\in \{\sqrt{2r_0}\}\times [-1/2,1/2]^{n-1}.
\end{equation*}
Thus, applying the standard maximum principle, we obtain
\begin{equation*}
\mathfrak w(z_1,0'')-\mathfrak w(0,0'')\le Az_1^2|\ln  z_1|.
\end{equation*}
By choosing a similar lower barrier and translating $0''$ to an arbitrary point in $[-3/8,3/8]^{n-1}$, we prove Step 1. The constant $A$ depends only on the quantities in Theorem \ref{thm3-349} and $\|u(0,x'')\|_{C^3(\overline{Q_3''})}$.

\par {\bf Step 2.} Let
\begin{equation*}
\mathcal H(z)=\mathfrak w(0,z'')+Mz_1^2(1-z_1^{\frac \alpha 2})-16 A|z''|^2 z_1^2 \ln  z_1,
\end{equation*}
where $A$ is the positive constant given by Step 1, and  $\alpha$ is the H\"older exponent of the coefficients $a_{ij}$.
Then, from Step 1, one obtains
\begin{equation*}
\mathcal H(z)\ge \mathfrak w(z),\quad  \text{on} \quad \partial([0,\sqrt{2r_0}]\times [-3/8,3/8]^{n-1})
\end{equation*}
 by choosing $M$ large enough. Also,
\begin{equation*}
\begin{split}
&\mathcal L(\mathcal H)-h_{2}\\
\le & Cz_1^\alpha-\frac{M\alpha}{C}(2+\alpha/2)z_1^{\frac \alpha 2}-\frac{A}{C}|z''|^2+CA z_1|\ln  z_1|<0
\end{split}
\end{equation*}
by choosing $M$ large enough.
Then, similar to Step 1, we have
\begin{equation*}
|\mathfrak w(z)-\mathfrak w(0,z'')|\le M z_1^2,\quad \text{in}\quad [0,\sqrt{2r_0}]\times [-5/16,5/16]^{n-1}.
\end{equation*}
The constant $M$ depends on the same quantities as $A$ in Step 1.
\par {\bf Step 3.}  For any point $(\lambda,p'')\in (0,3/8)\times [-9/32,9/32]^{n-1}$, let
\begin{equation*}
\mathfrak w_\lambda(z)=\frac{\mathfrak w(\lambda  r_0 z_1,\lambda z''+p'')}{\lambda^2}.
\end{equation*}
Here, we assume $\mathfrak w(0,p'')=|D_{z''}\mathfrak w(0,p'')|=0$. Also, $D_{z_1}\mathfrak w(0,p'')=0$ follows from Step 1.
Then, from Step 2, we have
\begin{equation*}
\begin{split}
|\mathfrak w_\lambda(z)|\le& \frac{|\mathfrak w(\lambda r_0 z_1,\lambda z''+p'')-\mathfrak w(0,\lambda z''+p'')|}{\lambda^2}+\frac{|\mathfrak w(0,\lambda z''+p'')|}{\lambda^2}\\
\le &C|z|^2.
\end{split}
\end{equation*}
Hence, standard $C^{2,\alpha}$-estimates yield that
\begin{equation*}
\|\mathfrak w_\lambda\|_{C^{2,\alpha}([5/6,7/6]\times [-3/8,3/8]^{n-1})}\le C'
\end{equation*}
for some positive constant $C'$ independent of $\lambda$.
Returning to $\mathfrak w$, this implies
\begin{equation*}
\|D_{z''}\mathfrak v\|_{C^{1,1}([0,3r_0/8]\times [-9/32,9/32]^{n-1})}\le C
\end{equation*}
for some positive constant $C$ depending on the same quantities as $A$ in Step 1.
\par {\bf Step 4.} Let $\mathfrak w^{(1)}=\partial_{z_2}^2 \mathfrak v$ and define the region $\Sigma_k$ as follows:
$$\Sigma_k=\left(0,\left(\frac 12-\sum\limits_{l=0}^k\frac{1}{2^{l+3}}\right)r_0\right)\times \left(-\left(\frac 12-\sum\limits_{l=0}^{k+3}\frac{1}{2^{l+3}}\right),\frac 12-\sum\limits_{l=0}^{k+3}\frac{1}{2^{l+3}}\right)^{n-1}.$$
Then, from Step 3, we have 
\begin{equation*}
\begin{split}
&a_{11}\left(\mathfrak w^{(1)}_{11}-\frac{\mathfrak w^{(1)}_1}{z_1}\right)+\sum_{i+j\ge 3}a_{ij}\mathfrak w^{(1)}_{ij}\\
=&h_{22}-\sum_{i+j\ge 3}^n \partial_{z_2}a_{ij}\mathfrak w_{ij}-\partial_{z_2}a_{11}\left(\mathfrak w_{11}-\frac{\mathfrak w_1}{z_1}\right) \in L^\infty(\Sigma_0).
\end{split}
\end{equation*}
Following a similar argument as Step 1, we obtain
\begin{equation*}
|\mathfrak w^{(1)}(z)-\mathfrak w^{(1)}(0,z'')|\le C_1 z_1^2|\ln  z_1|,\quad \text{in}\quad \Sigma_1.
\end{equation*}
As in Step 3, for any point $(\lambda r_0,p'')\in \Sigma_2$, define
\begin{equation*}
\mathfrak w^{(1)}_\lambda(z)=\frac{\mathfrak w^{(1)}(\lambda r_0 z_1,\lambda z''+p'')}{\lambda^{2-\varepsilon}}
\end{equation*}
for any small $\varepsilon\in (0,1)$.
Assuming $\mathfrak w^{(1)}(0,p'')=|D_{z''}\mathfrak w^{(1)}(0,p'')|=0$, we then have 
\begin{equation*}
|\mathfrak w^{(1)}_\lambda(z)|\le C_\varepsilon \lambda^{\frac{\varepsilon}{2}}.
\end{equation*}
Hence, we can apply $W^{2,p}$ estimate to $\mathfrak w^{(1)}$ to obtain
\begin{equation*}
\|\mathfrak w_\lambda^{(1)}\|_{W^{2,p}([1-\sigma_0,1+\sigma_0]\times [-1+\sigma_0,1+\sigma_0]^{n-1})}\le C_\varepsilon \lambda^{\frac{\varepsilon}{2}}
\end{equation*}
for some small $\sigma_0>0$.
Using Sobolev embedding theorem, and choosing $p=\frac{n}{\varepsilon}$, we have
\begin{equation*}
\|\mathfrak w_\lambda ^{(1)}\|_{C^{1,1-\varepsilon}([1-\sigma_0,1+\sigma_0]\times [-1+\sigma_0,1+\sigma_0]^{n-1})}\le C_\varepsilon \lambda^{\frac{\varepsilon}{2}}.
\end{equation*}
Scaling back to $\mathfrak w^{(1)}$, this implies
\begin{equation*}
\|\mathfrak w^{(1)}\|_{C^{1,1-\varepsilon}(\overline{\Sigma_3})}\le C_\varepsilon.
\end{equation*}
The positive constant $C_\varepsilon$ depends only on $\varepsilon$,  the quantities in Theorem \ref{thm3-803}, and $\|u(0,x'')\|_{C^4(\overline{Q_3''})}$.
\par {\bf Step 5.}  From Step 4, it is known that $D_{z''}^2D_z\mathfrak v\in C^\alpha(\overline{\Sigma_3})$, $\forall \alpha\in (0,1)$. From \eqref{MAs2}, one gets
\begin{equation*}
\mathfrak v_{z_1z_1}-\frac{\mathfrak v_{z_1}}{z_1}=\mathfrak F(z,D^2_z\mathfrak v),
\end{equation*}
where $\mathfrak F(z,D^2_z\mathfrak v)$ is a smooth function in $z, \mathfrak v_{z_iz_j}$, $i+j\ge 3$.
Hence, the coefficients $a_{ij}$ are smooth functions of $z$ and $\mathfrak v_{z_iz_j}$ for $i+j\ge  3$, and it follows that $\mathfrak w_{11}-\frac{\mathfrak w_1}{z_1}\in C^{\alpha}(\overline{\Sigma_3})$.
This further implies
\begin{equation*}
h_{22}-\sum_{i+j\ge 3}^n \partial_{z_2}a_{ij}\mathfrak w_{ij}-\partial_{z_2}a_{11}\left(\mathfrak w_{11}-\frac{\mathfrak w_1}{z_1}\right)\in C^\alpha(\overline{\Sigma_3}).
\end{equation*}
Then we can repeat the argument in Step 2 to get
\begin{equation*}
|\mathfrak w^{(1)}(z)-\mathfrak w^{(1)}(0,z'')|\le C_1 z_1^2.
\end{equation*}
Then, by Step 3, one gets
\begin{equation*}
\|D_{z''}^2\mathfrak w^{(1)}\|_{C^{1,1}(\overline{\Sigma_4})}\le C_4.
\end{equation*}

Repeating the arguments  step by step, one gets
\begin{equation}\label{bdy-z1-0}
D_{z_1}D_{z''}^k\mathfrak v(0,z'')=0
\end{equation}
and
$D_{z''}^k \mathfrak v\in C^{1,1}(\overline{\Sigma_{k+2}})$, $k=0,1,2,\cdots.$ Moreover, the following estimate holds:
\begin{equation}\label{eati-tangential-1}
\|D_{z''}^k \mathfrak v\|_{C^{1,1}(\overline{\Sigma_{k+2}})}\le C_k,
\end{equation}
where the positive constant $C_k$ depends on the quantities in Lemma \ref{lemc1-1}, $\|u(0,x'')\|_{C^{k+3}(\overline{Q_3''})}$ and $\|h\|_{C^{k+1}(\overline{Q_3})}$.
\par {\bf Step 6.}
Given the results from Steps 1–5, we have:
$$D_{x''}^k v,\sqrt{x_1}D_{x''}^kD_{x_1}v,\frac{1}{2}D_{x''}^kD_{x_1}v+x_1D_{x''}^kD^2_{x_1}v\in L^\infty(\widetilde \Sigma_{k+2}),\quad \forall k\ge 0$$
where
\begin{equation*}
\widetilde \Sigma_k=\{x|(2\sqrt{x_1},x'')\in \Sigma_k\}.
\end{equation*}
 Note that $D_{x''}^kD_{x_1}v(x)=2\frac{D_{z_1}D_{z''}^k\mathfrak v(z_1,z'')}{z_1}$. By
\eqref{bdy-z1-0} and \eqref{eati-tangential-1}, we have $D_{x''}^kD_{x_1}v\in L^\infty(\widetilde \Sigma_{k+2})$.
 Hence, by Lemma \ref{lemma-Growth-Regularity} and Lemma \ref{lemma-Growth-Regularity2}, one has
$$D_{x''}^{k} v, D_{x''}^{k+1} v, \sqrt{x_1}D_{x''}^kD_{x_1}v\in C^{\frac{1}{2}}(\widetilde \Sigma_{k+2}),\quad D_{x''}^kD_{x_1}v, x_1D_{x''}^kD^2_{x_1}v\in L^\infty(\widetilde \Sigma_{k+2}).$$

 Let $w=v_{x_1}$,   by Theorem \ref{thm3-349}, $w \in C (\overline{\widetilde \Sigma_1})$.  Differentiating \eqref{s3-01} with respect to $x_1$ yields that
\begin{equation}\label{20240402}
x_1w_{11}+w_1+\sum_{i=2}^{n}\tilde a_{i1}w_{i1}+\sum_{i,j\ge 2}^{n}\tilde a_{ij}w_{ij}=f,
\end{equation}
where $\tilde a_{i1}=x_1 \sum\limits_{j=2}^n  v_{1j}\tilde b_{ij}$
and $\tilde b_{ij},\tilde a_{ij}$ are all smooth functions in  $$x,v_{lm},x_1v_{l1}v_{1m},x_1v_{11},\quad  i,j,l,m\ge 2.$$
 Hence,
\begin{equation*}
\tilde a_{i1}=\sqrt{x_1}\hat a_{i1},\quad i=2,\cdots,n,
\end{equation*}
where $\hat a_{i1}$ is H\"older continuous and $\hat a_{i1}(0,x'')=0$.
 Hence,
 $$f\in L^p(\widetilde \Sigma_3),\quad \forall p>1,\quad \tilde a_{ij},\tilde a_{i1}\in C^\alpha\left(\overline{\widetilde \Sigma_3}\right),\quad i,j\ge 2.$$
  Denote $W(z)=w(x)$, $z''=x'', z_1=2\sqrt{x_1}$. Then $W$ solves
\begin{equation}\label{eq-W-1}
W_{z_1z_1}+\frac {W_{z_1}}{z_1}+2\sum_{i=2}^{n}\hat a_{i1}W_{z_iz_1}+\sum_{i,j=2}^{n}\tilde a_{ij}W_{z_iz_j}=f.
\end{equation}
Let $\widetilde W(y,z'')=W(z)$, $z_1=\sqrt{y_{1}^2+y_{2}^2}$.  Then, $\widetilde W$ solves
\begin{equation}\label{230711}
\begin{split}
\mathcal R(\widetilde W)=&\widetilde W_{y_1y_1}+\widetilde W_{y_2y_2}+2\sum_{i=2}^{n}\hat a_{i1}\left(\frac{y_1}{z_1}\widetilde W_{z_iy_1}+\frac{y_2} {z_1}\widetilde W_{z_i y_2}\right)\\
+&\sum_{i,j=2}^{n}a_{ij}\widetilde W_{z_i z_j}=f,\quad \text{in}\quad \widehat \Sigma_3\backslash\{y_1=y_2=0\},
\end{split}
\end{equation}
where  $\widehat \Sigma_k=\{(y,z')|(\sqrt{y_1^2+y_2^2},z'')\in \Sigma_k'\}$, and
$$\Sigma_k'=\left[0,\left(\frac 12-\sum\limits_{l=0}^k\frac{1}{2^{l+3}}\right)r_0\right)\times \left(-\left(\frac 12-\sum\limits_{l=0}^{k+3}\frac{1}{2^{l+3}}\right),\frac 12-\sum\limits_{l=0}^{k+3}\frac{1}{2^{l+3}}\right)^{n-1}.$$

Notice that the coefficients in \eqref{230711} are continuous in $\overline{\widehat \Sigma_3}$ and $f\in L^p(\overline{\widehat \Sigma_3})$. Consider the following Dirichlet problem
\begin{equation}\label{230711-1}
\begin{split}
&\mathcal R(\mathcal W)=f,\quad \text{in}\quad \widehat \Sigma_3,\\
& \mathcal W=\widetilde W,\quad \text{on}\quad \partial \widehat \Sigma_3.
\end{split}
\end{equation}
 By the standard $W^{2,p}$-estimates, \eqref{230711-1} admits a unique solution  $\mathcal W\in W_{loc}^{2,p}(\widehat \Sigma_3)\cap C(\overline{\widehat \Sigma_3})$ for any $p>n+1$. Since
 \begin{equation*}
 \mathcal R(\mathcal W-\widetilde W\pm \sigma \ln(y_1^2+y_2^2))=0,\quad \text{in}\quad \widehat \Sigma_3\backslash\{y_1=y_2=0\},
 \end{equation*}
 holds for any $\sigma>0$, one can apply the maximum principle to conclude that $\mathcal W=\widetilde W$.
 Hence, 
 \begin{equation*}
 \widetilde W\in W_{loc}^{2,p}(\widehat \Sigma_3)\Rightarrow D\widetilde W\in C^\alpha_{loc}(\widehat \Sigma_3),\quad \alpha=1-\frac{n+1}{p}.
 \end{equation*}
 This implies  $v_{x''x_1}\in C_{loc}^{\frac{\alpha}{2}}(\widetilde \Sigma_3')$, where
 $$\widetilde \Sigma_k'=\{x|(2\sqrt{x_1},x'')\in \Sigma_k'\},\quad k\geq 1.$$
From this, one knows
 \begin{equation*}
\frac{y_l \hat a_{il}}{z_1}=\frac 12y_l\sum_{j=2}^n v_{lj}\tilde b_{ij}\in C^{\frac{\alpha}{2}}_{loc}(\widehat \Sigma_3), \quad f\in C^{\frac{\alpha}{2}}_{loc}(\widehat \Sigma_3),\quad l=1,2.
 \end{equation*}
 Hence, all the coefficients in \eqref{230711} are in $C^{\frac{\alpha}{2}}_{loc}(\widehat \Sigma_3)$.
 Using Schauder estimates, we can then conclude that $W\in C^{2,\frac{\alpha}{2}}_{loc}(\widehat \Sigma_3)$, i.e.,
\begin{equation*}
D^N_x v,\quad D_{x''}^2D_x v,\quad \sqrt{x_1} D_{x''}D_{x_1}^2 v,\quad  x_1D_{x_1}^3 v    \in C^{\frac{\alpha}{4}}(\overline{\widetilde \Sigma_4}),\quad |N|\le 2.
\end{equation*}
Then, we can differentiate \eqref{20240402} with respect to $x''$  and repeat the above argument to get
\begin{equation*}
D_{x''}D_x^{N}v,\quad D_{x''}^3D_x v,\quad \sqrt{x_1}D_{x''}^2D_{x_1}^2 v,\quad x_1D_{x''}D_{x_1}^3v\in C^{\frac{\alpha}{4}}(\overline{\widetilde {\Sigma_5}}).
\end{equation*}
Differentiating \eqref{20240402} repeatedly with respect to $x''$ yields that
\begin{equation*}
D_{x''}^k D_{x_1}^2 v,\quad x_1D_{x''}^kD_{x_1}^3 v\in C^{\frac \alpha 4}(\overline{\widetilde {\Sigma_{k+4}  }}),\quad \forall k\ge 1.
\end{equation*}

\par
Differentiating \eqref{s3-01} with respect to $x_1$ twice yields that $\mathfrak w^{(2)}=v_{x_1x_1}$ solves
\begin{equation*}
\begin{split}
&x_1w^{(2)}_{11}+2w^{(2)}_1+\sum_{i=2}^{n}\tilde a_{i1}w^{(2)}_{i1}+\sum_{i,j\ge 2}^{n}\tilde a_{ij}w^{(2)}_{ij}+\sum_{i\ge 2}\left(\frac{\partial \tilde a_{i1}}{\partial x_1}-\frac{\partial f}{\partial v_{i1}}\right)w^{(2)}_{i}\\
=& \frac{\partial f}{\partial x_1}+\sum_{i,j=2}^{n}\left(\frac{\partial f}{\partial v_{ij}}-\frac{\partial a_{ij}}{\partial x_1}\right)v_{ij1}.
\end{split}
\end{equation*}
Repeating the argument and by induction, we can prove Theorem \ref{thm3-803}.
\end{proof}

\section{Smoothness up to $(n-\mathfrak k)$-face, $2\le \mathfrak k\le n$}
For any fixed $\mathfrak k$, consider the following type of equation:
\begin{equation}\label{locpro2-1}
\begin{split}
& \det D^2 u=\frac{h(x)}{\prod_{i=1}^{\mathfrak k} x_{i}},\quad \text{in}\quad Q_{3},\\
& \det D^2_{ij} u \left|_{\{x_l=0,l\in \Gamma,i,j\notin \Gamma\}}=\frac{h}{\prod_{i=1,i\notin \Gamma}^{\mathfrak k}x_i}\right|_{\{x_l=0,l\in \Gamma\}}.
\end{split}
\end{equation}
Here, $\Gamma$ can be any subset of $\{1,2,\cdots,\mathfrak k\}$, and $Q_{3}=(0,3)^{\mathfrak k}\times (-3,3)^{n-\mathfrak k}$. We make the following convention: for $\mathfrak k=n$, $\Gamma=\{1,\cdots,n\}$, the second equation in \eqref{locpro2-1} means
\begin{equation*}
h(0)=1.
\end{equation*}
 Let
\begin{equation}\label{locpro2-2}
v(x)=u(x)-\sum_{i=1}^\mathfrak k x_i\ln  x_i,\quad v(x)\in C^\infty(\overline{Q_{3}}\backslash \{x_1=\cdots=x_\mathfrak k=0\}).
\end{equation}
Additionally, assume
\begin{equation}\label{locpro2-3}
0<h(x)\in C^\infty(\overline{Q_3}),\quad u(x)\in C(\overline{Q_{3}}),\quad v|_{\partial_0 Q_3} \in C^\infty
\end{equation}
where
\begin{equation*}
\partial_0 Q_3=\{x\in \partial Q_3| x_i=0, \text{ for some } i=1,\cdots,\mathfrak k\}.
\end{equation*}
In the following discussion, we always denote
\begin{equation*}
x'=(x_1,\cdots,x_\mathfrak k),\quad x''=(x_{\mathfrak k+1},\cdots,x_n)
\end{equation*}
and
\begin{equation}\label{modulus of continuity}
\omega_{\mathfrak k, u}(t)=\sup_{\substack{x\in Q_3\\|x'|\le t}}|u(0',x'')-u(x)|
\end{equation}
 Moreover, on $x''=0$, one has
\begin{equation}\label{locpro2-4}
\frac 1\Lambda \mathbb I_{(n-\mathfrak k)\times (n-\mathfrak k)}\le D_{x''}^2 v(0,x'')\le \Lambda \mathbb I_{(n-\mathfrak k)\times (n-\mathfrak k)}
\end{equation}
for some positive constant $\Lambda$.

\par In the following, we will always use the following notations if no confusion occurs.
 We say two quantities $A,B$ satisfying $A\lesssim_l B$ if and only if there exists a positive constant $C$ such that $A\le CB$ where the positive constant $C$ depends only on $n$, $\mathfrak k$, $\Lambda$, $\|u\|_{L^\infty(Q_3)}$, $\|1/h\|_{L^\infty(Q_3)}$, $\|v\|_{C^{\sigma(l)}(\partial_0 Q_3)}$, $\|h\|_{C^{\widetilde \sigma (l)}(\overline{Q_3})}$ and $\omega_{\mathfrak k,u}$. Here, both $\{\sigma(l)\}_{l=0}^{+\infty}$ and $\{\widetilde\sigma(l)\}_{l=0}^{+\infty}$ are sequences of increasing positive integers that depend only on $\mathfrak k$ and $n$,
and as $l$ approaches infinity, both $\sigma(l)$ and $\widetilde\sigma(l)$ tend to infinity.
If $l=0$, we just use $A\lesssim B$. Also, $A\approx B$ means $A\lesssim B$ and $B\lesssim A$.

Then we can formulate our main result in the following theorem.
\begin{theorem}\label{thmlocal-k}
Let $u,v$ be given by \eqref{locpro2-1}-\eqref{locpro2-4}.
Then there exists a positive constant $r_{\mathfrak k}\approx 1$ small enough such that $v(x)\in C^\infty([0,r_\mathfrak k]^{\mathfrak k}\times [-1/2,1/2]^{n-\mathfrak k})$. Moreover, there  holds
\begin{equation*}
\|v\|_{C^{l+2}([0,r_\mathfrak k]^\mathfrak k\times [-1/2,1/2]^{n-\mathfrak k})}\lesssim_l 1,\quad l\ge 0.
\end{equation*}
\end{theorem}

\begin{remark}\label{ind-asum-k}
The proof for Theorem \ref{thmlocal-k} is based on the method of induction. Theorem \ref{thmlocal-k} is true for $\mathfrak k=1$, i.e., Theorem \ref{thm3-803}. In the following proof, we assume Theorem \ref{thmlocal-k} holds true for $1,\cdots,\mathfrak k-1$.
\end{remark}
\subsection{Quadratic polynomial growth up to $(n-\mathfrak k)$-face, $2\le \mathfrak k\le n$}

\begin{lemma}\label{lem-lip-k} Let $u,v$ be given by \eqref{locpro2-1}-\eqref{locpro2-4}.
Then there holds
\begin{equation*}
|v(x',x'')-v(0',x'')|\lesssim |x'|,\quad \text{in}\quad  Q_1.
\end{equation*}

\end{lemma}
\begin{proof}
We will restrict our discussion on $(0,\delta_0)^{\mathfrak k}\times (-3,3)^{n-\mathfrak k}$ for some $\delta_0>0$ small.
\par
Step 1. There exists a positive constant $C$ such that
\begin{equation*}
v(x',x'')-v(0',x'')\le C|x'|.
\end{equation*}
By the convexity of $u$, one has
\begin{equation*}
\begin{split}
u(x)\le & \frac 1{\mathfrak k}\sum_{a=1}^{\mathfrak k} u((0',x'')+\mathfrak k x_a {\bf e_a})\\
=&\frac 1{\mathfrak k}\sum_{a=1}^{\mathfrak k} v((0',x'')+\mathfrak k x_a {\bf e_a})+\sum_{a=1}^{\mathfrak k} (x_a\ln  x_a+(\ln  \mathfrak k)x_a)\\
\le & v((0',x''))+\sum_{a=1}^{\mathfrak k} x_a\ln  x_a+C_1|x'|.
\end{split}
\end{equation*}
The positive constant $C_1$ depends on $\|v\|_{C^1(\partial_0 Q_3)}$.
\par Step 2. There exists a positive constant $C$ such that
\begin{equation*}
v(x',x'')-v(0',x'')\ge -C|x'|.
\end{equation*}
\par Step 2.1. There exists a positive constant $A$ such that
\begin{equation*}
v(x',x'')-v(0',x'')\ge A\sum_{a=1}^\mathfrak k x_a\ln x_a.
\end{equation*}
The constant $A$ depends only on $\Lambda, \|u\|_{L^\infty(Q_3)}, \|v\|_{C^1(\partial_0 Q_3)}$, $\|h\|_{C(\overline{Q_3})}$ and $\omega_{\mathfrak k,u}$.
Without loss of generality, we may assume
\begin{equation*}
u(0)=v(0)=|D_{x''}v(0)|=|D_{x''}u(0)|=0.
\end{equation*}
Let
\begin{equation*}
\mathcal H(x)=\frac 12 v(0',x'')+A \sum_{a=1}^{\mathfrak k} x_a\ln  x_a.
\end{equation*}
A direct computation yields that
\begin{equation*}
\det D^2 \mathcal H=(1/2)^{n-\mathfrak k}A^{\mathfrak k}\frac{\det D_{x''}^2v(0',x'')}{\prod_{a=1}^\mathfrak k x_a}\ge \det D^2 u,\quad \text{in}\quad (0,\delta_0)^{\mathfrak k}\times (-3,3)^{n-\mathfrak k}
\end{equation*}
for $A$ large enough.
Now we compare the boundary values of $\mathcal H$ and $u$.
\\ Case I: On $\{x_a=0\}$, for some $a=1,\cdots,\mathfrak k$.  We just discuss $a=1$.  It follows that
\begin{equation*}
\begin{split}
\mathcal H(0,x_2,\cdots,x_n)\le & v(0',x'')+A\sum_{a=2}^{\mathfrak k} x_a\ln x_a \\
\le &  v(0,x_2,\cdots,x_n)+A \sum_{a=2}^{\mathfrak k} x_a\ln  x_a+C_1 \sum_{a=2}^{\mathfrak k} x_a\\
\le & u(0,x_2,\cdots,x_n)
\end{split}
\end{equation*}
for $A$ large enough depending on $\|v\|_{C^1(\partial_0 Q_{3})}$.
\\ Case II: On $\{x_a=\delta_0\}$, for some $a=1,\cdots,\mathfrak k$. By choosing $A$ large enough, we know $u\ge\mathcal H$.
\\ Case III: In the remaining case, one has $|x''|\ge 3$. By the uniform convexity of $v(0',x'')$, there exists $\varepsilon_0>0$ such that
\begin{equation*}
\begin{split}
v(0',x'') \ge 4\varepsilon_0,\quad \forall x''\in \partial ([-3,3]^{n-\mathfrak k}).
\end{split}
\end{equation*}
Let $\delta_0>0$ be small enough such that  $\omega_{\mathfrak k,u}(\delta_0)\le \varepsilon_0$,   where $\omega_{\mathfrak k,u}$ is defined in \eqref{modulus of continuity}.
Then, the following holds:
\begin{equation*}
\begin{split}
u(x)\ge u(0',x'')-\varepsilon_0\ge \frac 12 v(0',x'')+\sum_{a=1}^{\mathfrak k} x_a\ln  x_a\ge \mathcal H(x)
\end{split}
\end{equation*}
for $A>1$. The standard maximum principle then yields Step 2.1.
\\ Step 2.2.  Let
\begin{equation*}
H(x)=v(0',x'')+\sum_{a=1}^{\mathfrak k} \left (x_a\ln  x_a(1+A|x''|^2 )-Bx_a(1-x_a^{1/2})\right)
\end{equation*}
where $A$ is given by Step 2.1. Denote $\mathbb M_{\mathcal F}$ as follows:
\begin{equation}\label{hessian-def-k}
\mathbb M_{\mathcal F}= \left. \begin{pmatrix}
   \sqrt{x_ax_b} D_{x_ax_b}\mathcal F  &    \sqrt{x_a} D_{x''x_a}\mathcal F \\[3pt]
 \sqrt{x_a}(D_{x''x_a}\mathcal F)^T &  D_{x''}^2 \mathcal F
 \end{pmatrix}\right|_{a,b=1,\cdots,\mathfrak k}.
\end{equation}
A direct computation yields:
\begin{equation*}
\mathbb  M_{H}=\mathbb M_{v(0',x'')+\sum_{a=1}^{\mathfrak k} x_a\ln  x_a  }+\mathbb A+\mathbb B+\mathbb C
\end{equation*}
where
\begin{equation*}
\begin{split}
&\mathbb A=2A\left(\sum_{a=1}^\mathfrak k x_a\ln  x_a\right)\sum_{i=\mathfrak k+1}^{n}{\bf e_i}\otimes {\bf e_i},\\
&\mathbb B=2A \sum_{i=\mathfrak k+1}^{n}\sum_{a=1}^\mathfrak k \sqrt{x_a}(1+\ln  x_a)x_i\left(
{\bf e_i}\otimes {\bf e_a}+{\bf e_a}\otimes {\bf e_i}\right),\\
&\mathbb C=\sum_{a=1}^\mathfrak k \left(A|x''|^2+\frac{3B}{4}x_a^{1/2}\right){\bf e_a}\otimes {\bf e_a}.
\end{split}
\end{equation*}
In $[0,\delta_0]^{\mathfrak k}\times [-2,2]^{n-\mathfrak k}$,  one knows
\begin{equation*}
\begin{split}
\det \mathbb  M_{H}\ge & \det (\mathbb M_{v(0',x'')+\sum_{a=1}^{\mathfrak k} x_a\ln  x_a}+\mathbb A+\mathbb B)+\frac{1}{C_1}\left(|x''|^2+B|x'|^{1/2}\right),\quad \text{by } \eqref{locpro2-4}\\
\ge &\det \mathbb M_{v(0',x'')+\sum_{a=1}^{\mathfrak k} x_a\ln  x_a}+\frac{1}{C_1}\left(|x''|^2+B|x'|^{1/2}\right)\\
&+C_1\left(\sum_{a=1}^\mathfrak k x_a\ln  x_a\right)-C_1\sum_{a=1}^\mathfrak k |x''||x_a^{\frac 12}\ln  x_a|\\
\ge & h(0',x'')+\frac{B|x'|^{1/2}}{2C_1}
\ge  h(x)= \det \mathbb M_u,\quad \text{by } \eqref{locpro2-1} \text{ and } h\in C^1(\overline{Q_3})
\end{split}
\end{equation*}
for $B$ large enough.  Now we compare the boundary value of $H$ and $u$.
 \par Case I. On $\partial([0,\delta_0]^{\mathfrak k}\times [-2,2]^{n-\mathfrak k})\cap \{x_a=0\}$, for some $a=1,\cdots,\mathfrak k$,
without loss of generality, we assume $x_1=0$.
Then,
\begin{equation*}
\begin{split}
u(0,x_2,\cdots,x_n)=&v(0,x_2,\cdots,x_n)+\sum_{a=2}^{\mathfrak k}x_a\ln  x_a\\
\ge &v(0',x'')-C_3\sum_{a=2}^\mathfrak k x_a+\sum_{a=2}^{\mathfrak k}x_a\ln  x_a\\
\ge & H(0,x_2,\cdots,x_n)
\end{split}
\end{equation*}
provided $B\ge C_3/(1-\delta_0^{\frac 12})$.
\par Case II. On $\partial([0,\delta_0]^{\mathfrak k}\times [-2,2]^{n-\mathfrak k})\cap \{x_a=\delta_0\}$, for some $a=1,\cdots,\mathfrak k$,  it is sufficient to choose $B\ge C_4 \delta_0^{-1}$.
\par Case III. On  $[0,\delta_0]^{\mathfrak k}\times \partial( [-2,2]^{n-\mathfrak k})$. This follows from Step 2.1.
\\ From the standard maximum principle, one proves Step 2. This ends the proof of present lemma.
\end{proof}

Lemma \ref{lem-lip-k} is not enough for our blow up arguments. We also need the following type of estimates(Lemma \ref{lem-lip-k-refine}).
Let $F$ be a smooth extension of $v|_{\partial_0 Q_3}$, i.e.,
\begin{equation}\label{extension}
\begin{split}
&F(x)=v(x),\quad \text{on}\quad \partial_0 Q_3, \quad F\in C^\infty(\overline{Q_3}),\\
& \|F\|_{C^l(\overline{Q_3})}\le C(n)\|v\|_{C^l(\partial_0 Q_3)}
\end{split}
\end{equation}
for some constant $C(n)$ depending only on the dimension $n$. For example, for $n=3,\mathfrak k=2$, we can simply take
\begin{equation*}
F(x)=v(x_1,0,x_3)+v(0,x_2,x_3)-v(0,0,x_3).
\end{equation*}

\begin{lemma}\label{lem-lip-k-refine}
Let $u,v$ be given by \eqref{locpro2-1}-\eqref{locpro2-4}.
Then there holds
\begin{equation*}
|v(x)-F(x)|\lesssim (x_1\cdots x_{\mathfrak k})^\frac{1}{\mathfrak k}+o(|x'|),\quad \text{in}\quad  Q_1.
\end{equation*}
\end{lemma}
\begin{proof}
Step 1. We first show \begin{equation*}F(x)-v(x)\lesssim (x_1\cdots x_{\mathfrak k})^\frac{1}{\mathfrak k}+o(|x'|),\quad \text{in}\quad  Q_1.\end{equation*}

Since $F(0',x'')$ is uniformly convex(\eqref{locpro2-4}), one knows $F(x)-\frac{1}{8\Lambda}|x''|^2$ remains uniformly convex in $x''$ for $|x'|$ small enough. For simplicity, we assume
$$D^2_{x''}(F(x)-\frac{1}{8\Lambda}|x''|^2)\ge \frac{1}{8\Lambda}  \mathbb I_{(n-\mathfrak k)\times (n-\mathfrak k)},\quad \text{in}\quad Q_1. $$
Define
\begin{equation*}
\begin{split}
h_\delta(x)&=\sum_{a=1}^{\mathfrak k}x_a\ln x_a+F(x)-\frac{1}{8\Lambda}|x''|^2-\varepsilon \left(\sum_{a=1}^{\mathfrak k}x_a\right)\ln\left( \sum_{a=1}^{\mathfrak k}\frac{x_a}{\delta}+2\right)-B(x_1\cdots x_{\mathfrak k})^{\frac 1{\mathfrak k}}\\
&=H_\delta-B(x_1\cdots x_{\mathfrak k})^{\frac 1{\mathfrak k}}
\end{split}
\end{equation*}
in $[0,\sqrt\delta]^{\mathfrak k}\times [-1,1]^{n-\mathfrak k}$, where $\varepsilon=\frac{C}{\ln\left(\frac{1}{\sqrt \delta}+2\right)}$, and $C$ is the Lipschitz constant given by Lemma \ref{lem-lip-k}.
A simple calculation yields that
\begin{equation*}
\mathbb M_{H_\delta}\approx \mathbb I_{n\times n}.
\end{equation*}
Denote $G(x)=(x_1\cdots x_{\mathfrak k})^{\frac 1{\mathfrak k}}$. Since $G(x)$ is concave, we have
\begin{equation*}
\begin{split}
\det (\mathbb M_{H_\delta-BG})\ge & \det \mathbb M_{H_\delta}-\frac{B}{C_0} \text{Trace}(\mathbb M_{G})
\\ \ge & \frac{B}{C_1}\sum_{a=1}^{\mathfrak k}\frac{(x_1\cdots x_\mathfrak k)^{\frac{1}{\mathfrak k}}}{x_a}\\
\ge &\frac{B}{C_1}\ge \det \mathbb M_u
\end{split}
\end{equation*}
for $B$ large enough. We now compare the boundary value.
\begin{itemize}
	\item[(1).] On  $[0,\sqrt\delta]^{\mathfrak k}\times \partial([-1,1]^{n-\mathfrak k})$, we have
	\begin{equation*}
	\begin{split}
	u(x)-h_\delta\ge \frac{1}{8\Lambda}+v(x)-F(x)\ge \frac{1}{8\Lambda}-C\sum_{a=1}^{\mathfrak k} x_a\ge \frac{1}{16\Lambda}
	\end{split}
	\end{equation*}
	for $\delta\le \delta_0$ and $\delta_0$ small enough.
	\item[(2).] On $x_a=0$ for some $a\in \{1,\cdots,\mathfrak k\}$.  Since $v=F$ on $x_a=0$,  then  $u\ge h_\delta$ on $x_a=0$.
	\item[(3).] On $x_a=\sqrt{\delta}$ for some  $a\in \{1,\cdots,\mathfrak k\}$. Then, the following holds:
	\begin{equation*}
	\begin{split}
	u(x)-h_\delta\ge v(x)-F(x)+\varepsilon \left(\sum_{a=1}^{\mathfrak k}x_a\right)\ln\left( \sum_{a=1}^{\mathfrak k}\frac{x_a}{\delta}+2\right)\ge 0
	\end{split}
	\end{equation*}
by $\varepsilon$'s definition.
 \end{itemize}
 Overall, by maximum principle, we obtain $u\ge h_\delta$ in $[0,\sqrt\delta]^{\mathfrak k}\times [-1,1]^{n-\mathfrak k}$ for $\delta\in (0,\delta_0)$, i.e.,
 \begin{equation*}
 \begin{split}
 v(x',0'')-F(x',0'')&\ge -\varepsilon \left(\sum_{a=1}^{\mathfrak k}x_a\right)\ln\left( \sum_{a=1}^{\mathfrak k}\frac{x_a}{\delta}+2\right)-B(x_1\cdots x_{\mathfrak k})^{\frac 1{\mathfrak k}}\\
 &\ge -(n+2)\varepsilon \left(\sum_{a=1}^{\mathfrak k}x_a\right)-B(x_1\cdots x_{\mathfrak k})^{\frac 1{\mathfrak k}}
 \end{split}
 \end{equation*}
 for $|x'|\le \delta$.
 Also, the constant $B$ is independent of $\delta$, and $\varepsilon\rightarrow 0$ as $\delta\rightarrow 0$.

Step 2. Now we show \begin{equation*}v(x)-F(x)\lesssim (x_1\cdots x_k)^\frac{1}{k}+o(|x'|),\quad \text{in}\quad  Q_1.\end{equation*}

Define
\begin{equation*}
\begin{split}
h_\delta'(x)&=\sum_{a=1}^{\mathfrak k}x_a\ln x_a+F(x)+|x''|^2+\varepsilon \left(\sum_{a=1}^{\mathfrak k}x_a\right)\ln\left( \sum_{a=1}^{\mathfrak k}\frac{x_a}{\delta}+2\right)+B(x_1\cdots x_{\mathfrak k})^{\frac 1{\mathfrak k}}\\
&=H_\delta'+B(x_1\cdots x_{\mathfrak k})^{\frac 1{\mathfrak k}}
\end{split}
\end{equation*}
in $[0,\sqrt\delta]^{\mathfrak k}\times [-1,1]^{n-\mathfrak k}$, where $\varepsilon=\frac{C}{\ln\left(\frac{1}{\sqrt \delta}+2\right)}$,  with $C$ being the Lipschitz constant given by Lemma \ref{lem-lip-k}. For any $B>0$, arguing as in Step 1, we have
\begin{equation*}u(x)\leq h_\delta'(x)
 \quad \quad \text{ on   } \partial\big([0,\sqrt\delta]^{\mathfrak k}\times [-1,1]^{n-\mathfrak k} \big)\end{equation*}
for $\delta\le \delta_0$ and $\delta_0$ small enough.

We claim when $B\geq \|D^2 F\|_{L^{\infty}}+10C+10$, there holds
 \begin{equation}\label{v-geq-F-refine}u(x)\leq h_\delta'(x)
 \quad \quad \text{ in   }  [0,\sqrt\delta]^{\mathfrak k}\times [-1,1]^{n-\mathfrak k}. \end{equation}
Suppose \eqref{v-geq-F-refine} is not true. Then, we can assume $h_\delta'(x)-u(x)$ attains its negative minimum at $\bar x$ in the interior
of $[0,\sqrt\delta]^{\mathfrak k}\times [-1,1]^{n-\mathfrak k}$. At $\bar x$, we have 
\begin{equation*}
\mathbb M_{H_\delta'}(\bar x)- \mathbb M_{-BG}(\bar x) =\mathbb M_{h_\delta'(x)}(\bar x)\ge \mathbb M_{u}(\bar x).
\end{equation*}
However, the maximum eigenvalue of $\mathbb M_{-BG}(\bar x)$ is greater than or equal to $\frac{\text{Trace}(\mathbb M_{-BG})}{\mathfrak k} \geq B $, which leads to a contradiction.
Thus, we conclude that
\begin{equation*}
 v(x',0)-F(x',0)\le  (n+2)\varepsilon \left(\sum_{a=1}^{\mathfrak k}x_a\right)+B(x_1\cdots x_{\mathfrak k})^{\frac 1{\mathfrak k}}
 \end{equation*}
 for $|x'|\le \delta$.
\end{proof}

In the next result, we prove a Liouville-type theorem. This theorem will be used to establish the global $C^1$-regularity of $v$ through the blow-up argument.

\begin{theorem}\label{thm-liouk}
Let $\mathfrak u\in C(\overline{(\mathbb R^+)^{\mathfrak k}\times \mathbb R^{n-\mathfrak k}})$ be an Alexandrov solution of the following equation
\begin{equation}\label{liouk-1}
\begin{split}
& \det D^2 \mathfrak u=\frac{1}{\prod\limits_{a=1}^\mathfrak k x_a},\quad \text{in}\quad (\mathbb R^+)^{\mathfrak k}\times \mathbb R^{n-\mathfrak k},\\
& \mathfrak u|_{\partial((\mathbb R^+)^{\mathfrak k}\times \mathbb R^{n-\mathfrak k})}=\sum_{a=1}^\mathfrak k x_a\ln  x_a+\frac 12 |x''|^2,\\
& \left|\mathfrak u(x)-\sum_{a=1}^{\mathfrak k} x_a\ln  x_a-\frac 12 |x''|^2\right|\le C   (x_1\cdots x_{\mathfrak k})^\frac{1}{\mathfrak k}.
\end{split}
\end{equation}
Then,
\begin{equation*}
\mathfrak u(x)=\sum_{a=1}^{\mathfrak k} x_a\ln  x_a+\frac 12|x''|^2.
\end{equation*}
\end{theorem}
\begin{proof}
Following the  proof of Lemma \ref{lemc1-1}, one can deduce that $\mathfrak u$ is strictly convex in $(\mathbb R^+)^{\mathfrak k}\times \mathbb R^{n-\mathfrak k}$. Hence, $u\in C^\infty((\mathbb R^+)^{\mathfrak k}\times \mathbb R^{n-\mathfrak k})$.
Let
\begin{equation*}
\mathfrak v(x)=\mathfrak u(x)-\sum_{a=1}^{\mathfrak k} x_a\ln  x_a.
\end{equation*}
 For any point $p\in (\mathbb R^+)^{\mathfrak k}\times \mathbb R^{n-\mathfrak k}$ with $\lambda=|p'|$, let
\begin{equation*}
\mathfrak v_\lambda(x)=\frac{\mathfrak v(\lambda x',\sqrt \lambda x'')}{\lambda},\quad \mathfrak u_\lambda(x)=\mathfrak v_\lambda(x)+\sum_{a=1}^{\mathfrak k} x_a\ln  x_a.
\end{equation*}
Without loss of generality, we assume $p''=0$. From the last condition in \eqref{liouk-1}, one knows
\begin{equation*}
\left|\mathfrak v_\lambda(x)-\frac 12 |x''|^2\right|\le C\prod_{a=1}^\mathfrak k x_a^{\frac 1{\mathfrak k}}
\end{equation*}
and $\mathfrak u_\lambda$ also solves \eqref{liouk-1}. For any compact set $K\subset (\mathbb R^+)^{\mathfrak k}\times \mathbb R^{n-\mathfrak k}$, a similar argument as in Lemma \ref{lemc1-1} implies that  the strict convexity of $\mathfrak u_\lambda$ in $K$
is independent of $\lambda$.
By induction assumptions on $\mathfrak k$, i.e., Remark \ref{ind-asum-k}, we know that for any compact set $K\subset \overline{(\mathbb R^+)^{\mathfrak k}\times \mathbb R^{n-\mathfrak k}}\backslash\{x_1=\cdots=x_\mathfrak k=0\}$, the following holds:
\begin{equation*}
\|v_\lambda\|_{C^l(K)}\le C(l,K),\quad \forall l\ge 0
\end{equation*}
for some positive constant $C(l,K)$ independent of $\lambda$. Notice that
\begin{equation*}
D_{x'} \mathfrak v(p',0'')=D_{x'}\mathfrak v_\lambda\left(\frac{p'}{|p'|},0''\right).
\end{equation*}
We know $|D_{x'}\mathfrak v|$ is bounded in $(\mathbb R^+)^{\mathfrak k}\times \mathbb R^{n-\mathfrak k}$.
\par Claim: $D_{x'}\mathfrak v\equiv 0$.
\\ Suppose not. Without loss of generality, we may assume
\begin{equation*}
\sup_{(\mathbb R^+)^{\mathfrak k}\times \mathbb R^{n-\mathfrak k}} D_{x_1}\mathfrak v>0.
\end{equation*}
Consider a sequence of points $p_k$ such that
\begin{equation*}
D_{x_1}\mathfrak v(p_k)\rightarrow \sup_{(\mathbb R^+)^{\mathfrak k}\times \mathbb R^{n-\mathfrak k}} D_{x_1}\mathfrak v>0.
\end{equation*}
Again, we may assume $p_k''=0$. By the previous argument, up to a subsequence, we know
\begin{equation*}
\mathfrak v_{\lambda_k}\rightarrow \mathfrak V,\quad \text{locally uniformly in} \quad C^l(\overline{(\mathbb R^+)^{\mathfrak k}\times \mathbb R^{n-\mathfrak k}}\backslash\{x_1=\cdots=x_\mathfrak k=0\}),\quad \forall l\ge 1.
\end{equation*}
Hence, we have
\begin{equation*}
\mathfrak V_{x_1}(\bar x',0'')=\sup_{(\mathbb R^+)^\mathfrak k\times \mathbb R^{n-\mathfrak k}} \mathfrak V_{x_{1}}
\end{equation*}
for $\bar x'=\lim\limits_{k\rightarrow +\infty}\frac{p_k'}{|p_k'|}$. Also we know
\begin{equation*}
\mathfrak U=\mathfrak V+\sum_{a=1}^\mathfrak k x_a\ln  x_a
\end{equation*}
solves \eqref{liouk-1}.  Therefore, $\mathfrak W=\mathfrak V_{x_1}$
solves
\begin{equation*}
\begin{split}
&\sum_{b=1}^\mathfrak k x_b \tilde a_{bb}\mathfrak W_{bb}+\sum_{b\ne c}^\mathfrak k x_b x_c\tilde a_{bc}\mathfrak W_{bc}+\sum_{j=\mathfrak k+1}^n\sum_{b=1}^\mathfrak k x_b \tilde a_{bj}\mathfrak W_{bj}\\
+& \sum_{i,j=\mathfrak k+1}^n \tilde a_{ij}\mathfrak W_{ij}+\tilde a_{11}\mathfrak W_1+\sum_{b=2}^{\mathfrak k} x_b \tilde a_{b1} \mathfrak W_b=0,
\end{split},\quad \text{in}\quad (\mathbb R^+)^\mathfrak k\times \mathbb R^{n-\mathfrak k}.
\end{equation*}
 Here, $\tilde a_{ij}$, $i,j=1,\cdots,n$ are smooth functions in $x$ and $\mathfrak V_{pq}$ for $p,q=1,\cdots,n$. We will now discuss the following three cases.
\par Case 1. $(\bar x',0'')\in (\mathbb R^+)^\mathfrak k\times \mathbb R^{n-\mathfrak k}$. By the strong maximum principle, we know
\begin{equation*}
\mathfrak V_{x_{1}}\equiv \mathfrak V_{x_{1}}(\bar x',0)\ne 0
\end{equation*}
 which contradicts the boundary condition since $\mathfrak k\ge 2$.
\par Case 2. $(\bar x',0'')\in \{x_a=0\}$ for some $2\le a\le \mathfrak k$. From the boundary condition, it follows that $\mathfrak V_{x_{1}}(\bar x',0'')=0$, leading to a contradiction.
\par Case 3. $(\bar x',0'')\in \{x_{1}=0\}$ and $\bar x_a>0$, for  $a=2,\cdots,\mathfrak k$. By Remark \ref{ind-asum-k}, we know that
$\mathfrak V\in C^\infty(\overline{(\mathbb R^+)^{\mathfrak k}\times \mathbb R^{n-\mathfrak k}}\backslash\{x_1=\cdots=x_\mathfrak k=0\})$.  Let
\begin{equation*}
\widetilde {\mathfrak W}(y_1,y_2,x_2,\cdots,x_n)=\mathfrak W\left(\frac{y_1^2+y_2^2}{4},x_2,\cdots,x_n\right).
\end{equation*}
Then, $\widetilde {\mathfrak W}\in C^\infty(\overline{\mathbb R^2\times (\mathbb R^+)^{\mathfrak k-1}\times \mathbb R^{n-\mathfrak k}}\backslash \{x_2=\cdots=x_\mathfrak k=0\})$ and solves an elliptic PDE that is locally uniformly elliptic in $\mathbb R^2\times (\mathbb R^+)^{\mathfrak k-1}\times \mathbb R^{n-\mathfrak k}$. Hence, if $\widetilde {\mathfrak W}$ obtains its maximum in the interior, it implies $\widetilde {\mathfrak W}$ is a constant. This leads to a contradiction.
\par This proves $D_{x'}\mathfrak v=0$ and ends the proof of present theorem.
\end{proof}

\begin{lemma}\label{lem-C1-conti}
Let $u,v$ be given by Theorem \ref{thmlocal-k}.
Then $$D_{x'} v,D_{x''}^2 v,\sqrt{x_a}v_{ia},\sqrt{x_ax_b}v_{ab}\in C(\overline{Q_2}),\quad a,b=1,\cdots,\mathfrak k, i=\mathfrak k+1,\cdots, n,$$
and the modulus of continuity depends only on the quantities in Lemma \ref{lem-lip-k}.
\end{lemma}
\begin{proof}
It is enough to show $v_{x_1}\in C(\overline{Q_2})$. Suppose not. Without loss of generality, we may assume there exist  two sequences  $v^k,x^k$ such that
\begin{equation*}
x^k\rightarrow 0,\quad |D_{x_1}v^k(x^k)|\ge \varepsilon_0>0,\quad D v^k(0)=0,\quad D^2_{x''}v^k(0)=\mathbb I_{(n-\mathfrak k)\times (n-\mathfrak k)}.
\end{equation*}
Let $\lambda_k=|(x^k)'|$ and
\begin{equation*}
\tilde v^k(x)=\frac{v^k(\lambda_k x',\sqrt{\lambda_k}x''+(x^k)'')-v^k(0,(x^k)'')-Dv^k(0,(x^k)'')\cdot(\lambda_k x',\sqrt{\lambda_k}x'')}{\lambda_k}.
\end{equation*}
Then, by Remark \ref{ind-asum-k} and Lemma \ref{lem-lip-k-refine}, up to a subsequence, we have
\begin{equation*}
\tilde v^k\rightarrow \mathfrak v,\quad \text{in}\quad  C_{loc}^{\infty}(\overline{(\mathbb R^+)^{\mathfrak k}\times \mathbb R^{n-\mathfrak k}}\backslash\{x_1=\cdots=x_\mathfrak k=0\}),
\end{equation*}
 where $\mathfrak u=\mathfrak v+\sum\limits_{a=1}^\mathfrak k x_a\ln  x_a$ solves \eqref{liouk-1}. Also, we may assume $\left(\frac{(x^k)'}{\lambda_k},0''\right)\rightarrow (\bar x',0)$ for some $\bar x'\in  \mathbb S^{\mathfrak k-1}$.
Then $|D_{x'}\mathfrak v(\bar x',0'')|\ne 0$ which yields a contradiction.
\end{proof}
The following lemma shows the quadratic polynomial growth of $F(x)-v(x)$ up to the $(n-\mathfrak k)$-face when $\mathfrak k\geq2$, which is the basis for driving the sharp asymptotic estimate for  $F(x)-v(x)$ later.
\begin{lemma}\label{lem-k-quar-low}
Let $u,v$ be given by Theorem \ref{thmlocal-k} and $F$ be given by \eqref{extension}.  Then there holds
\begin{equation*}
|F(x)-v(x)|\lesssim \sum_{1\le a<b\le \mathfrak k}x_ax_b.
\end{equation*}

\end{lemma}

\begin{proof}
We first show $F(x)-v(x)\lesssim \sum_{1\le a<b\le \mathfrak k}x_ax_b.$
\par Step 1.  There holds the following estimate
\begin{equation*}
v(x)-F(x)\ge -C_0\sum_{1\le a<b\le \mathfrak k}(x_ax_b)^{3/4},\quad \text{in}\quad Q_2.
\end{equation*}
Let
\begin{equation*}
H(x)=F(x)+\sum_{a=1}^{\mathfrak k} x_a\ln  x_a-\varepsilon |x''|^2\sum_{a=1}^{\mathfrak k} x_a-\frac{\varepsilon}{\delta^{1/2}}\sum_{1\le a<b\le \mathfrak k}(x_ax_b)^{3/4}
\end{equation*}
in $[0,\delta]^{\mathfrak k}\times [-3,3]^{n-\mathfrak k}$ for some small positive constants $\varepsilon,\delta>0$ to be determined later. It follows from the definition of $F$ that
\begin{equation*}
D_{x_i}F(x)=D_{x_i}v(x),\quad \text{on}\quad x_a=0,\quad i=1,\cdots,n,\quad i\ne a.
\end{equation*}
Hence, by Lemma \ref{lem-C1-conti}, one knows
\begin{equation*}
F(x)-v(x)=o (x_a),\quad \forall a\in \{1,\cdots,\mathfrak k\}.
\end{equation*}
We first compare the boundary values of $H(x)$ and $u(x)$.
\begin{itemize}
\item[(1).] On $[0,\delta]^{\mathfrak k}\times \partial([-3,3]^{n-\mathfrak k})$, for any fixed $\varepsilon>0$, one knows
\begin{equation*}
\begin{split}
 H(x)-u(x)=&F(x)-v(x)-\varepsilon |x''|^2\sum_{a=1}^\mathfrak k x_a-\frac{\varepsilon}{\delta^{\frac 12}}\sum_{1\le a<b\le \mathfrak k}(x_a x_b)^{\frac 34}\\
 \le & o\left(\sum_{a=1}^{\mathfrak k} x_a\right)-\varepsilon \sum_{a=1}^\mathfrak k x_a<0
\end{split}
\end{equation*}
by choosing $\delta$ small enough.
\item[(2).] On $x_a=0$ for some $a\in \{1,\cdots,\mathfrak k\}$, it follows from the definition of $F(x)$ that $H(x)\le u(x)$.
\item[(3).]  On $x_a=\delta$ for some $a\in \{1,\cdots,\mathfrak k\}$. Without loss of generality, we assume $x_1=\delta$. Then there holds
\begin{equation*}
\begin{split}
H(x)-u(x)\le& F(\delta,x_2,\cdots,x_\mathfrak k,x'')-v(\delta,x_2,\cdots,x_\mathfrak k,x'')-\varepsilon\delta^{1/4}\sum_{a=2}^\mathfrak kx_a^{\frac 34}\\
\le & o\left(\sum_{a=2}^\mathfrak k x_a\right)-\varepsilon  \sum_{a=2}^\mathfrak k x_a<0
\end{split}
\end{equation*}
for $\delta>0$ small enough.
\end{itemize}
\par Now we need to calculate the Hessian of $H$ and its determinant.
A direct computation yields that
\begin{equation*}
\mathbb M_H=\mathbb M_{F+\sum_{a=1}^{\mathfrak k} x_a\ln  x_a}+\mathbb A+\mathbb B+\mathbb C+\mathbb D
\end{equation*}
where
\begin{equation*}
\begin{split}
&\mathbb A=-2\varepsilon\left(\sum_{a=1}^\mathfrak k x_a\right)\sum_{i=\mathfrak k+1}^{n}{\bf e_i}\otimes {\bf e_i},\\
&\mathbb B=-2\varepsilon\sum_{a=1}^\mathfrak k\sum_{i=\mathfrak k+1}^{n} x_i\sqrt{x_a}\left({\bf e_i}\otimes {\bf e_a}+{\bf e_a}\otimes {\bf e_i}\right)\\
& \mathbb C=-\frac{9\varepsilon}{16\delta^{\frac 12}}\sum_{a,b=1,a\ne b}^{\mathfrak k}(x_ax_b)^{\frac{1}{4}}{\bf e_a}\otimes {\bf e_b}\\
& \mathbb D=\frac{3\varepsilon}{16\delta^{\frac 12}}\sum_{a=1}^\mathfrak k\left(\sum_{b=1, b\ne a}^{\mathfrak k}x_b^{\frac 34}\right)x_a^{-\frac 14}{\bf e_a}\otimes {\bf e_a}
\end{split}
\end{equation*}
and $\mathbb M_H,\mathbb M_F$ are defined by \eqref{hessian-def-k}.
It follows directly that
\begin{equation}\label{0724-0}
 x_b^{\frac 34}x_a^{-\frac 14}+x_a^{\frac 34}x_b^{-\frac 14}\ge \sqrt{x_a}+\sqrt{x_b}\ge 2x_a^{\frac 14}x_b^{\frac 14}.
\end{equation}
By the definition of $F$, we know
\begin{equation*}
\frac 1C \mathbb I_{n\times n}\le \mathbb M_{_{F+\sum_{a=1}^{\mathfrak k} x_a\ln  x_a}}\le C \mathbb I_{n\times n},\quad \text{in}\quad [0,\delta]^\mathfrak k\times [-3,3]^{n-\mathfrak k}
\end{equation*}
for some positive constant $C>0$ provided that $\delta>0$ is small.
Let $\widetilde {\mathbb M}=\mathbb M_F+\mathbb C$. Then $\widetilde {\mathbb M}$ satisfies
\begin{equation*}
|(\text{cof }\widetilde {\mathbb M})_{ab}|\le \kappa(\varepsilon,\delta)\rightarrow 0,\quad \text{as}\quad \varepsilon,\delta\rightarrow 0
\end{equation*}
for $a\ne b$, $a,b=1,\cdots,\mathfrak k$.
This implies
\begin{equation*}
\begin{split}
\det \mathbb M_H\ge & \det (\mathbb M_{_{F+\sum_{a=1}^{\mathfrak k} x_a\ln  x_a}}+\mathbb A+\mathbb B+\mathbb C)+\frac{8\varepsilon }{C_1\delta^{\frac 12}}\sum_{a,b=1,a\ne b}^\mathfrak k x_b^{\frac 34}x_a^{-\frac 14}\\
\ge & \det (\mathbb M_{F+\sum_{a=1}^{\mathfrak k} x_a\ln  x_a}+\mathbb C)+\frac{8\varepsilon }{C_1\delta^{\frac 12}}\sum_{a,b=1,a\ne b}^\mathfrak k x_b^{\frac 34}x_a^{-\frac 14}-C_1\sum_{a=1}^\mathfrak k \sqrt{x_a}\\
\ge &\det (\mathbb M_{F+\sum_{a=1}^{\mathfrak k} x_a\ln  x_a})+\sum_{a,b=1,a\ne b}^\mathfrak k \left(\frac{4\varepsilon }{C_1\delta^{\frac 12}} x_b^{\frac 34}x_a^{-\frac 14}-\frac{C_1\varepsilon \kappa(\varepsilon,\delta)}{\delta^{\frac 12}} x_b^{\frac 14}x_a^{\frac 14}\right)\\
\ge &\det (\mathbb M_{F+\sum_{a=1}^{\mathfrak k} x_a\ln  x_a})+\frac{2\varepsilon }{C_1\delta^{\frac 12}}\sum_{a,b=1,a\ne b}^\mathfrak k  x_b^{\frac 34}x_a^{-\frac 14}
\end{split}
\end{equation*}
for $0<\delta<<\varepsilon<<1$ small enough.
By the definition of $F$ and Lemma \ref{lem-pre}, one knows
\begin{equation*}
|\det \mathbb M_{v+\sum_{a=1}^{\mathfrak k} x_a\ln  x_a}-\det \mathbb M_{F+\sum_{a=1}^{\mathfrak k} x_a\ln  x_a}|\le C\prod_{a=1}^\mathfrak k x_a.
\end{equation*}
Hence,
\begin{equation*}
\det \mathbb M_H\ge \det \mathbb M_{v+\sum_{a=1}^{\mathfrak k} x_a\ln  x_a}+\frac{2\varepsilon }{C_1\delta^{\frac 12}}\sum_{a,b=1,a\ne b}^\mathfrak k  x_b^{\frac 34}x_a^{-\frac 14}- C\prod_{a=1}^\mathfrak k x_a\ge \det \mathbb M_{u}
\end{equation*}
for $\delta$ small enough. It follows from standard maximum principle that $H\le u$, which completes the proof of Step 1. 
\par Step 2.  Let
\begin{equation*}
\begin{split}
\mathcal H(x)=F(x)+\sum_{a=1}^{\mathfrak k} x_a\ln  x_a-&C_0 |x''|^2\sum_{a,b=1,a\ne b}^\mathfrak k (x_ax_b)^{\frac 34}\\
-&\sum_{1\le a<b\le \mathfrak k}x_ax_b\left(\frac{\sigma}{r}-\frac{\sigma^{\frac 32}}{r^{\frac 32}}(x_ax_b)^{\frac 14}\right)
\end{split}
\end{equation*}
in $[0,r]^{\mathfrak k}\times [-3,3]^{n-\mathfrak k}$.
Here, $0<r<<\sigma<<1$   will be determined later, and $C_0$ is given by Step 1.
A direct computation yields
\begin{equation*}
\mathbb M_{\mathcal H}=\mathbb M_{F+\sum_{a=1}^{\mathfrak k} x_a\ln  x_a}+\mathbb A'+\mathbb B'+\mathbb C'+\mathbb D'+\mathbb E'
\end{equation*}
where
\begin{equation*}
\begin{split}
&\mathbb A'=-2C_0\left(\sum_{a,b=1,a\ne b}^\mathfrak k (x_ax_b)^{\frac 34}\right)\sum_{i=\mathfrak k+1}^{n}{\bf e_i}\otimes {\bf e_i},\\
&\mathbb B'=\sum_{a=1}^\mathfrak k\left(\frac{3C_0}{16}|x''|^2\left(\sum_{\substack{b=1\\b\ne a}}^\mathfrak k x_b^{\frac 34}\right)x_a^{-\frac 14}+\frac{5\sigma^{\frac 32}}{16r^{\frac 32}}\left(\sum_{\substack{b=1\\b\ne a}}^\mathfrak k x_b^{\frac 54}\right)x_a^{\frac 14}\right){\bf e_a}\otimes {\bf e_a},\\
&\mathbb C'=-  \frac{3}{2} C_0\sum_{i=\mathfrak k+1}^{n}\sum_{b=1}^\mathfrak k \left(x_ix_b^{\frac 14}\left(\sum_{\substack{a=1\\a\ne b}}^\mathfrak k x_a^{\frac 34}\right)\right)\left({\bf e_i}\otimes {\bf e_b}+{\bf e_b}\otimes {\bf e_i}\right)\\
&\mathbb D'=-\frac{9C_0}{16}|x''|^2\sum_{\substack{a,b=1\\b\ne a}}^\mathfrak k (x_ax_b)^{\frac 14}({\bf e_a}\otimes {\bf e_b}+{\bf e_b}\otimes {\bf e_a})\\
&\mathbb E'=-\sum_{\substack{a,b=1\\b\ne a}}^\mathfrak k \left(\frac{\sigma}{r}-\frac{25\sigma^{\frac 32}}{16r^{\frac 32}}(x_ax_b)^{\frac 14}\right)(x_ax_b)^{\frac 12}({\bf e_a}\otimes {\bf e_b}+{\bf e_b}\otimes {\bf e_a}).
\end{split}
\end{equation*}
Hence, one gets
\begin{equation}\label{0724-1}
\begin{split}
\det \mathbb M_{\mathcal H}\ge& \frac 1{C_1}\left(\sum_{\substack{a,b=1\\b\ne a}}^\mathfrak k   \left(|x''|^2x_b^{\frac 34}x_a^{-\frac 14}+\frac{\sigma^{\frac 32}}{r^{\frac 32}}x_b^{\frac 54}x_a^{\frac 14} \right) \right)\\ &-C_1\sum_{\substack{a,b=1\\b\ne a}}^\mathfrak k (x_ax_b)^{\frac 34}+\det (\mathbb M_{F+\sum_{a=1}^{\mathfrak k} x_a\ln  x_a}+\mathbb C'+\mathbb D'+\mathbb E').
\end{split}
\end{equation}
It is also easy to see that
\begin{equation*}
\begin{split}
&|\mathbb M_{F+\sum_{a=1}^{\mathfrak k} x_a\ln  x_a,ia}|\lesssim \sqrt{x_a},\quad |\mathbb C'_{ia}|\lesssim x_a^{\frac 14}\sum_{\substack{b=1\\b\ne a}}^\mathfrak k x_b^{\frac 34},\\
&|\mathbb M_{F+\sum_{a=1}^{\mathfrak k} x_a\ln  x_a,ab}|+|\mathbb D'_{ab}|+|\mathbb E'_{ab}|\lesssim |x''|^2(x_ax_b)^{\frac 14}+\frac{\sigma}{r}(x_ax_b)^{\frac 12}
\end{split}
\end{equation*}
for $i=\mathfrak k+1,\cdots,n$, $a,b=1,\cdots,\mathfrak k$ and $a\ne b$.
Hence
\begin{equation}\label{0724-2}
\begin{split}
&\det (\mathbb M_{F+\sum_{a=1}^{\mathfrak k} x_a\ln  x_a}+\mathbb C'+\mathbb D'+\mathbb E')\\
\ge& \det \mathbb M_{F+\sum_{a=1}^{\mathfrak k} x_a\ln  x_a}-C_1\sum_{a=1}^\mathfrak k\sqrt{x_a}\left(\sum_{\substack{b=1\\b\ne a}}^\mathfrak k x_a^{\frac 14} x_b^{\frac 34}\right) \\
-&C_1 \sum_{a=1}^{\mathfrak k}\big(x_a^{\frac 14}\sum_{\substack{b=1\\b\ne a}}^\mathfrak k x_b^{\frac 34}\big)^{2}- C_1\sum_{\substack{a,b=1\\a\ne b}}^\mathfrak k \left(|x''|^2(x_ax_b)^{\frac 14}+\frac{\sigma}{r}(x_ax_b)^{\frac 12}\right)^2\\
\ge & \det \mathbb M_{F+\sum_{a=1}^{\mathfrak k} x_a\ln  x_a}-C_1'\sum_{\substack{a,b=1\\a\ne b}}^\mathfrak k \left(x_b^{\frac 14} x_a^{\frac 54}+|x''|^2(x_ax_b)^{\frac 12}+\frac{\sigma^2}{r^2}x_ax_b\right).
\end{split}
\end{equation}

Notice that
\begin{equation*}
|x''|^2 (x_a x_b)^{\frac 12}\le r^{\frac 12}|x''|^2 (x_a x_b)^{\frac 14},\   \frac{\sigma^2}{r^2}x_ax_b\le \frac{\sigma^2}{r^{\frac 32}}(x_ax_b)^{\frac 34}.
\end{equation*}
Combining \eqref{0724-0}, \eqref{0724-1}, \eqref{0724-2} and choosing $0<r<< \sigma<<1$, by Lemma \ref{lem-pre}, one gets
\begin{equation*}
\det \mathbb M_{\mathcal H}\ge \det \mathbb M_u,\quad \text{in}\quad (0,r)^\mathfrak k\times (-3,3)^{n-\mathfrak k}.
\end{equation*}
\par Now we compare the boundary values of $\mathcal H$ and $u$.
\begin{itemize}
  \item[(1).] On $[0,r]^{\mathfrak k}\times \partial ([-3,3]^{n-\mathfrak k})$, it follows from the choice of $C_0$ and Step 1 such that $\mathcal H\le u$.
  \item[(2).] $x_a=0$ for some $a\in  \{1,\cdots,\mathfrak k\}$. It follows from the definition of $F(x)$ such that $\mathcal H\le u$.
  \item[(3).] $x_a=r$ for some $a\in  \{1,\cdots,\mathfrak k\}$. Similar argument as   Case (3) of Step 1, one gets $\mathcal H\le u$.
\end{itemize}
Hence by standard maximum principle, we get
$$F(x)-v(x)\lesssim \sum_{1\le a<b\le \mathfrak k}x_ax_b.$$
\par Step 3. It remains to show $v(x)-F(x)\lesssim \sum_{1\le a<b\le \mathfrak k}x_ax_b.$ Introduce the following barrier function
\begin{equation*}
H(x)=F(x)+\sum_{a=1}^{\mathfrak k} x_a\ln  x_a+\varepsilon |x''|^2\sum_{a=1}^{\mathfrak k} x_a+\frac{\varepsilon}{\delta^{1/2}}\sum_{1\le a<b\le \mathfrak k}(x_ax_b)^{3/4}
\end{equation*}
where $\varepsilon,\delta>0$ are small positive constants to be chosen. This function is similar to the one used in Step 1 and is defined in $[0,\delta]^\mathfrak k\times [-3,3]^{n-\mathfrak k}$.
The main issue here is that $H$ is not convex. However, we can still prove
\begin{equation}\label{4-01-24}
v(x)-F(x)\le C_0\sum_{1\le a<b\le \mathfrak k}(x_ax_b)^{\frac 34}
\end{equation}
by contradiction argument. Suppose \eqref{4-01-24} is not true. By appropriately choosing $\varepsilon$ and $\delta$,  we can assume $u(x)-H(x)$ attains its positive maximum at $\bar x$ in the interior. At $\bar x$, we have 
\begin{equation*}
D^2 H(\bar x)\ge D^2 u(\bar x),
\end{equation*}
i.e., $H(x)$ is convex near $\bar x$. This allows us to repeat the calculations from Step 1 and obtain:
\begin{equation*}
\det D^2 H(\bar x)< \det D^2 u(\bar x),
\end{equation*}
leading a contradiction. Once \eqref{4-01-24} is proved, similar arguments to those in Step 2 can be used to complete the proof of the lemma.
\end{proof}

\subsection{Smoothness up to $(n-2)$-face}
The computation for proving  \ref{thmlocal-k} for general $\mathfrak k$ is rather complex. In this sub-section, we will first illustrate our main strategies by proving Theorem \ref{thmlocal-k} for $\mathfrak k=2$.
Let $v,F,x',x''$  denote the same quantities as in the previous section for  $\mathfrak k=2$.
Set $I=v-F$. Our goal is to show that:
\begin{equation*}
I\in C^\infty([0,1]^2\times[-1,1]^{n-2}).
\end{equation*}
\smallskip
\smallskip
In the following, we give a lemma concerning the asymptotic estimate of the derivatives of $I$ with respect to $x'$.
\begin{lemma}\label{lemma-cd2-growth1}
There exists a constant $\tilde r_1\approx 1$  such that for any $\delta\in(0,1)$, the following estimate holds:
\begin{equation}\label{eq-cd2-growth1}
|D^\beta_{x'}D_{x''}^\gamma I(x)| \lesssim_{L(|\beta|+|\gamma|,\delta,n)}|x'|^{1+\delta-|\beta|-\frac{|\gamma|}{2}}
\end{equation}
in $\big([0,\tilde r_1]^2\setminus \{0\}^2\big)\times[-2,2]^{n-2}$, where $L(|\beta|+|\gamma|,\delta,n)$ is a positive integer depending only on $|\beta|+|\gamma|,\delta$ and $n$.
\end{lemma}

\begin{proof}
Fix any $\bar x=(\bar x',\bar x'')$ with $\bar x''\in[-2, 2]^{n-2}$.
Take $\lambda\in (0,\frac{1}{16}]$. For any $y\in  [0,2]^2\times[-2,2]^{n-2}$,
define
\begin{equation*}
u_\lambda(y)=\frac{u(\lambda y',\bar x''+\sqrt{\lambda} y'')-(\lambda \ln  \lambda)y_1 -(\lambda \ln  \lambda)y_2}{\lambda}
\end{equation*}
and
\begin{equation*}
v_{\lambda}(y)\triangleq u_\lambda(y)-y_1\ln  y_1-y_2\ln  y_2=\frac{v(\lambda y', \bar x''+\sqrt{\lambda} y'')}{\lambda},
\end{equation*}

\begin{equation*}
F_{\lambda}(y)\triangleq \frac{F(\lambda y',\bar x''+ \sqrt{\lambda} y'')}{\lambda},
\end{equation*}
and
\begin{equation*}
I_{\lambda}(y)\triangleq \frac{I(\lambda y',\bar x''+\sqrt{\lambda} y'')}{\lambda}.
\end{equation*}
Then, $u_\lambda(y)$ solves
\begin{equation}\label{3-cd2}
\det D_{y}^2 u_{\lambda}=\frac{h(\lambda y', \bar x''+\sqrt{\lambda}y'')}{y_1y_2}, \text{ in } [0,2]^2\times[-2,2]^{n-2}.
\end{equation}
By subtracting an affine function, we may assume
\begin{equation*}
F(0',\bar x'')=|D F(0',\bar x'')|=0.
\end{equation*}
Then, for any multi-index $\alpha$,
\begin{equation}\label{esti-1-alpha-F}
|D_{y}^{\alpha} F_{\lambda}|\leq C_{\alpha}',\quad  \text{in} \quad [0,2]^2\times [-2,2]^{n-2}.
\end{equation}
By Lemma \ref{lem-k-quar-low},
we have
$$|I(x)|\lesssim x_1x_2,\quad \text{in}\quad [0,\frac{1}{2}]^2\times[-\frac{5}{2},\frac{5}{2}]^{n-2}.$$
Hence, \begin{equation}\label{cd2-growth-lambda}
|I_{\lambda}(y)|\lesssim \lambda y_1y_2\quad \text{in}\quad [0,2]^2\times[-2,2]^{n-2}.
\end{equation}
By \eqref{3-cd2}, $u_\lambda(y) =[v_{\lambda}(y)+y_1\ln  y_1]+y_2\ln  y_2$ satisfies
\begin{equation*}
\det D^2 u_{\lambda}=\frac{h(\lambda y', \bar x''+\sqrt{\lambda} y'')/y_1}{y_2}=\frac{\widetilde{h}(y)}{y_2}, \text{ in } [1/2,2]\times [0,2] \times [-2,2]^{n-2},
\end{equation*}
with $u_{\lambda}(y_1,0,y'')=F_{\lambda}(y_1,0,y'')+y_1\ln  y_1\in C^\infty\big([1/2,2]\times [-2,2]^{n-2}\big)$.

On $[1/2,2]\times \{0\} \times [-2,2]^{n-2}$, $\big[\nabla^2_{y_1,y''}\big(F_{\lambda}(y_1,0,y'')+y_1\ln  y_1\big)\big]_{(n-1)\times(n-1)}$ equals to
\begin{equation*}
\begin{pmatrix}
     \frac{1}{y_1}+\lambda F_{11}(\lambda y_1,0,\bar x''+\sqrt{\lambda}y'')&    \sqrt{\lambda} F_{13}(\lambda y_1,0,\bar x''+\sqrt{\lambda}y'') & \cdots &  \sqrt{\lambda} F_{1n}(\lambda y_1,0,\bar x''+\sqrt{\lambda}y'') \\[3pt]
   \sqrt{\lambda} F_{31}(\lambda y_1,0,\bar x''+\sqrt{\lambda}y'')& F_{33}(\lambda y_1,0,\bar x''+\sqrt{\lambda}y'') & \cdots&   F_{3n}(\lambda y_1,0,\bar x''+\sqrt{\lambda}y'')\\[3pt]
  \cdots &\cdots &\cdots &\cdots  \\[3pt]
  \sqrt{\lambda}F_{n1}(\lambda y_1,0,\bar x''+\sqrt{\lambda}y'') &  F_{n3}(\lambda y_1,0,\bar x''+\sqrt{\lambda}y'') & \cdots& F_{nn}(\lambda y_1,0,\bar x''+\sqrt{\lambda}y'')
 \end{pmatrix}.
\end{equation*}
Then, from \eqref{locpro2-4}, the uniform convexity of $u_{\lambda}$ restricted on $[1/2,2]\times \{0\} \times [-2,2]^{n-2}$ is independent of $\lambda$ and $\bar x''\in[-2, 2]^{n-2}$ when $\lambda \leq \hat r_1$, where $\hat r_1$ is some positive constant depending only
on $F$.

Therefore, by Theorem \ref{thm3-803}, there exists $r_0>0$ such that
$$I_{\lambda}(y)=v_{\lambda}(y)-F_{\lambda}(y)\in C^\infty\big([\frac{3}{4},\frac{3}{2}]\times [0,r_0] \times [-\frac{3}{2},\frac{3}{2}]^{n-2}\big),$$
and for any multi-index $\alpha=\beta+\gamma$,
\begin{equation}\label{esti-1-alpha-I}
|D_{y}^{\alpha} I_{\lambda}|\lesssim_{|\alpha|} 1,\quad \text{in}\quad  [\frac{3}{4},\frac{3}{2}]\times [0,r_0] \times [-\frac{3}{2},\frac{3}{2}]^{n-2}
\end{equation}
 if $\lambda\leq \min\{\hat r_1,1/16\}$. A similar argument as in Lemma \ref{lemc1-1} implies that there exists a positive constant $\hat r_2>0$, such that the strict convexity of $u_{\lambda}(y)$ in $[\frac{3}{4},\frac{3}{2}]\times [r_0,\frac 32] \times [-\frac{3}{2},\frac{3}{2}]^{n-2}$
is independent of $\lambda$ and $\bar x''$ when $\lambda\in (0,\hat r_2]$.
Set $\tilde r_3=\min\{\hat r_1,\hat r_2,\frac{1}{16}\}$. In the following discussion, we further restrict $\lambda\in (0,\hat r_3]$. Hence, by interior estimate for $u_\lambda$, $\lambda\in (0,\tilde r_3]$, one knows
\begin{equation}\label{esti-1-alpha-I-b}
|D_{y}^{\alpha} I_{\lambda}|\lesssim_{|\alpha|} 1,\quad \text{in}\quad [\frac{3}{4},\frac{3}{2}]\times [r_0,3/2] \times [-\frac{3}{2},\frac{3}{2}]^{n-2}.
\end{equation}
Combining with \eqref{cd2-growth-lambda}, \eqref{esti-1-alpha-I} and \eqref{esti-1-alpha-I-b}, for any $\delta\in(0,1)$, one gets, by Lemma \ref{lemma-interpolation},
\begin{equation*}
|D_{y} I_{\lambda}|\lesssim_{L(|\alpha|,\delta,n)} \lambda^{\delta}\quad \text{in}\quad [\frac{3}{4},\frac{3}{2}]\times [0,\frac{3}{2}] \times [-\frac{3}{2},\frac{3}{2}]^{n-2}.
\end{equation*}
where $L(|\alpha|,\delta,n)$ is some positive integer depending only on $|\alpha|,\delta$ and $n$.
A similar argument also yields for any multi-index $\alpha$,
\begin{equation*}
|D_{y}^{\alpha} I_{\lambda}|\lesssim_{L(|\alpha|,\delta,n)} \lambda^{\delta}\quad\text{in}\quad [\frac{3}{4},\frac{3}{2}]\times [0,\frac{3}{2}] \times [-\frac{3}{2},\frac{3}{2}]^{n-2}.
\end{equation*}
Similarly, we have
\begin{equation*}\label{growth-interpo-ineq}
|D_{y}^{\alpha} I_{\lambda}|\lesssim_{L(|\alpha|,\delta,n)}  \lambda^{\delta},\quad \text{in}\quad [0,\frac{3}{2}] \times  [\frac{3}{4},\frac{3}{2}]\times \times [-\frac{3}{2},\frac{3}{2}]^{n-2}.
\end{equation*}
Therefore,
\begin{equation}\label{I-deriv-growth}
\lambda^{|\beta|+\frac{|\gamma|}{2}}|D_{y'}^{\beta}D_{y''}^{\gamma}I |(\lambda y', x''+\sqrt{\lambda}y'')\lesssim_{L(|\beta|+|\gamma| ,\delta,n)} \lambda^{1+\delta}
\end{equation}
in $\big([0,\frac{3}{2}]^2 \setminus [0,\frac{3}{4}]^2\big)\times [-\frac{3}{2},\frac{3}{2}]^{n-2}$.

If $x'\in[0,\hat r_3]^2\setminus\{0\}^2$, taking $\lambda=\max\{x_1,x_2\}$, $y'=\frac{x'}{\lambda}$ and $y''=0$ in \eqref{I-deriv-growth},
 we get the asymptotic estimate \eqref{eq-cd2-growth1}. 
\end{proof}

\begin{remark}
Note that in order to derive the estimate \eqref{growth-interpo-ineq}, we used the derivative estimate of $I_{\lambda}$, whose order is higher than $\alpha$ when using the interpolation inequality. This implies that our estimates is not optimal for the regularity of the dependent quantities. Alternative methods,  such as applying elliptic PDE estimates for the equation satisfied by $D_{y}^{\alpha} I_{\lambda}$, could provide better estimates.
Since our goal is to prove the smoothness of $I$ rather than obtaining the optimal estimate, we adopt this method mainly for convenience.
\end{remark}

With the help of Lemma \ref{lemma-Growth-Regularity} and Lemma \ref{lemma-Growth-Regularity2}, we have the following regularity result for $I$.
\begin{corollary}\label{cor4.1}
There  holds that
\begin{equation*}
\sum_{b=1}^2\sum_{a=1}^2\|\sqrt{x_a}I_{ab}\|_{C^{\frac 14}([0,\tilde r_1]^2\times [-2,2]^{n-2})}+\sum_{p=3}^n\sum_{i=1}^n \|I_{pi}\|_{C^{\frac 14}([0,\tilde r_1]^2\times [-2,2]^{n-2})}\lesssim 1.
\end{equation*}

\end{corollary}

The following lemma plays an important role in establishing the higher regularity result of $I$.

\begin{lemma}\label{lemma-I12}
There exists a positive constant $ r_{2*} \in (0,\tilde r_1/2)$ such that
\begin{equation}\label{eq-I12}
\|I_{12}\|_{C^{1/2}([0,2 r_{2*}]^2\times[-\frac{3}{2},\frac{3}{2}]^{n-2})}\lesssim 1.
\end{equation}
\end{lemma}

\begin{proof}
Note that $I=v-F$. We will prove $v_{12}\in C^{1/2}([0,2 r_{2*}]^2\times[-3/2,3/2]^{n-2})$ for some constant $r_{2*}\in(0,\tilde r_1/2)$.

For any $x\in([0,\tilde r_1]^2\setminus{0}^2)\times[-2,2]^{n-2}$,
denote the matrix
\begin{equation}\label{matrix-I12}
 \begin{pmatrix}
     x_1v_{11}+1&    \sqrt{x_1x_2}v_{12} & \cdots &  \sqrt{x_1}v_{1n}\\[3pt]
   \sqrt{x_1x_2}v_{21}& x_2 v_{22}+1 & \cdots&  \sqrt{x_2} v_{2n}\\[3pt]
  \cdots &\cdots &\cdots &\cdots  \\[3pt]
 \sqrt{x_1}v_{n1} & \sqrt{x_2} v_{n2} & \cdots& v_{nn}
 \end{pmatrix}
\end{equation}
by $(U_{ij})_{n\times n}$.
Then by Corollary \ref{cor4.1}, we have $U_{ij}\in C^{\frac 14}$, $i,j=1,\cdots,n$. Also, we have
\begin{equation}\label{MAI12}
\det U_{ij}=h(x).
\end{equation}
Differentiating \eqref{MAI12} with respect to $x_1,x_2$, we consider the terms $\delta^{1\cdots n}_{i_1\cdots\ i_n}U_{1i_1} \cdots U_{ni_n}$ in the expansion of $\det U_{ij}$ under three cases, where $(i_1,\cdots,i_n)$ are permutations of $(1,\cdots,n)$.

{\it Case 1.} $i_1=1$ and $i_2=2$.
Then, by Lemma \ref{lemma-cd2-growth1}, the terms containing $U_{ji_j,1}$ for $j\in\{1, \cdots ,n\}$ among $(U_{1i_1} \cdots U_{ni_n})_{12}$ are $\lesssim C_{\delta} |x'|^{2\delta-2}$, $\forall \delta\in (0,1)$.

{\it Case 2.} $i_1=1, i_2\neq2$ or $i_1\neq1, i_2=2$.

Without loss of generality,  assume $i_1=1$ and $i_2\neq2$. Then
$$U_{1i_1} \cdots U_{ni_n}=\delta^{1\cdots n}_{i_1\cdots\ i_n}x_2(x_1v_{11}+1)v_{2i_2} \cdots v_{ni_n}, \quad i_2\neq 2.$$
In this case, except for $\delta^{1\cdots n}_{i_1\cdots\ i_n}(x_1v_{11}+1)v_{2i_21} \cdots v_{ni_n}$, $\delta^{1\cdots n}_{i_1\cdots\ i_n}(x_1v_{11}+1)v_{2i_2} \cdots v_{j21} \cdots v_{ni_n}$ in $(U_{1i_1} \cdots U_{ni_n})_{12}$ for some $j\in\{3, \cdots ,n\}$, and the terms containing $v_{ji_j,12}$ for $j\in\{2, \cdots ,n\}$, all other terms are $\lesssim C_{\delta} |x'|^{2\delta-2}$, $\forall \delta\in (0,1)$. This again follows from Lemma \ref{lemma-cd2-growth1}.

{\it Case 3.} $i_1\neq1, i_2\neq2$. Then
$$U_{1i_1} \cdots U_{ni_n}=\delta^{1\cdots n}_{i_1\cdots\ i_n}x_1x_2v_{1i_1}v_{2i_2} \cdots v_{ni_n}, \quad i_1\neq 1,\quad i_2\neq 2.$$
In this case,  by Lemma \ref{lemma-cd2-growth1}, all terms among $(U_{1i_1} \cdots U_{ni_n})_{12}$ are $\lesssim C_{\delta}|x'|^{2\delta-2}$ since in this case, we have
$$|x_1x_2v_{ji_j12}|\lesssim  C_{\delta}|x'|^{\delta-1},\quad \forall j\in\{1,..,n\}.$$
 However, concerning the structure of the equation, we will still place the terms containing $v_{ji_j,12}$ for $j\in\{1, \cdots ,n\}$ on the left hand side of the following equation \eqref{eqV12-1}.

Therefore, $V=v_{12}$ solves an equation of the form:
\begin{equation}\label{eqV12-1}
A\big(x_1V_{11}+V_{1}+x_2 V_{22}+V_{2}\big)+\sum_{a=1}^{2}\sum_{p=3}^{n}x_a\frac{a_{ap}}{\sqrt{x_a}}V_{ap}+\sum_{p,q=3}^{n}a_{pq}V_{pq}+\sum_{p=3}^{n}B_pV_{p}=h_{12}+g:=G
\end{equation}
where
\begin{equation}\label{MI12}
A=\det \begin{pmatrix}
v_{33} & \cdots & v_{3n}\\[3pt]
  \cdots &\cdots &\cdots  \\[3pt]
 v_{n3} & \cdots& v_{nn}
 \end{pmatrix}
 =\det \begin{pmatrix}
F_{33}+I_{33} & \cdots & F_{3n}+I_{3n}\\[3pt]
  \cdots &\cdots &\cdots  \\[3pt]
 F_{n3}+I_{n3} & \cdots& F_{nn}+I_{nn}
 \end{pmatrix}
\end{equation}
$a_{ap}=(U^*)^{ap}$, $a_{pq}=(U^*)^{pq}$, and for fixed $x$, $B_p$ is some combination of polynomials in $v_{pq}$, $x_bv_{bb}+1$ and $v_{ap}$. The term $g$ is a combination of polynomials in $x_1x_2v_{ab12}$, $v_{ij}$ and $x_a v_{bij}$, where $a,b,c\in\{1,2\}$, $p,q\in\{3,..,n\}$, $i,j\in \{1,\cdots,n\}$.
 Hence,
 $$a_{pq}, \frac{a_{ap}}{\sqrt{x_a}}, A, B_p\in C^{\frac 14},\quad |g|\lesssim C_\delta |x'|^{2\delta-2},\quad \forall \delta\in(0,1).$$

Taking $\delta=1-\frac{1}{32(n+2)}$, therefore $ |g|\lesssim |x'|^{-\frac{1}{16(n+2)}}$.

Let $x_1=\frac{r_1^2}{4}$ and $x_2=\frac{r_{2}^2}{4}$. Then, \eqref{eqV12-1} becomes
\begin{equation}
\big(V_{r_1r_1}+\frac{V_{r_1}}{r_1}+ V_{r_2r_2}+\frac{V_{r_2}}{r_2}\big)+\sum_{a=1}^{2}\sum_{p=3}^{n}\frac{r_a}{2} b_{ap} V_{ap}+\sum_{p,q=3}^{n}\frac{a_{pq}}{A}V_{pq}+\sum_{p=3}^{n}\frac{B_p}{A}V_{p}=\widetilde{G}
\end{equation}
 where $b_{ap}=\frac{a_{ap}}{A\sqrt{x_a}}\in C^{\frac 14}$.
Let $$r_1=\sqrt{y_1^2+y_2^2},\quad r_2=\sqrt{y_3^2+y_4^2},\quad  y_{p}=x_{p-2},\quad p=5, \cdots ,n+2.$$
Let 
$$y=(y',y''),\quad y'=(y_1,y_2,y_3,y_4),\quad y''=(y_5, \cdots ,y_{n+2}).$$
 Then,
in  the domain, $$\bigg(\big(B_{2\sqrt{\tilde r_1}}^{2}(0)\times B_{2\sqrt{\tilde r_1}}^{2}(0)\big)\setminus\{0\}^4\bigg)\times [-2,2]^{n-2},$$
$\overline{V}(y)=v_{12}(x)$ solves
\begin{align}\label{eq-v12-y}\begin{split}
L\overline{V}:=&\sum_{i=1}^{4}\partial_{y_iy_i}\overline{V}+\sum_{a=1}^{2}\sum_{p=5}^{n+2}\frac{y_a}{2}\overline{b}_{1p}\overline{V}_{ap}+\sum_{a=3}^{4}\sum_{p=5}^{n+2}\frac{y_a}{2}\overline{b}_{2p}\overline{V}_{ap}\\
&+\sum_{p,q=5}^{n+2}\frac{\overline{a}_{pq}}{\overline{A}}\overline{V}_{pq}+\sum_{p=5}^{n+2}\frac{\overline{B}_p}{\overline{A}}\overline{V}_{p}\\
&=\overline{G},
\end{split}\end{align}
with
 $$| \overline{G}| \lesssim |y'|^{-\frac{1}{8(n+2)}},$$  
and $\overline{b}_{1p}(y)=b_{1x_{p-2}}(x)$, $\overline{b}_{2p}(y)$,
$\frac{\overline{a}_{pq}}{\overline{A}}$, $\frac{\overline{B}_p}{\overline{A}}$ and $\overline{G}$ are defined in the same way.
Let  $B_{\bar{r}}^{m}(\bar{x})$  be
 $$B_{\bar{r}}^{m}(\bar{x})=\{x\in \mathbb R^m| |x-\bar x|<\bar r\}.$$
Then,  there exists some constant $\tilde r'_1 \leq2\sqrt{\tilde r_1}$such that the above equation is uniformly elliptic in the domain
$$(B_{\tilde r'_1}^{4}(0)\setminus\{0\}^4)\times [-2,2]^{n-2}\subseteq\bigg(\big(B_{2\sqrt{\tilde r_1}}^{2}(0)\times B_{2\sqrt{\tilde r_1}}^{2}(0)\big)\setminus\{0\}^4\bigg)\times [-2,2]^{n-2},$$
and the coefficients
$\frac{\overline{a}_{pq}}{\overline{A}}, \frac{y_a}{2}\overline{b}_{\left[\frac {a+1}{2}\right]p} \in C^{\frac 14}(B_{\widetilde r_1'}^4(0)\times [-2,2]^{n-2})$.

Let $\eta(y)=\eta(|y''|)$ be the cut-off function satisfying $\eta=1$ when $|y''|\leq \frac{3\tilde r'_1}{4}$ and $\eta=0$ when $|y''|\geq \tilde r'_1$.
Then, $w=\eta\overline{V}$ solves
\begin{align}\label{eq-v12eta}\begin{split}
Lw=&\eta\overline{G}+\frac{\overline{V}}{\overline{A}}(\sum_{p,q=5}^{n+2}\overline{a}_{pq}\eta_{pq}+\sum_{p=5}^{n+2}\overline{B}_p\eta_{p})\\
&+\sum_{a=1}^{2}\sum_{p=5}^{n+2}\frac{y_a}{2}\overline{b}_{1p}\overline{V}_{a}\eta_p+\sum_{a=3}^{4}\sum_{p=5}^{n+2}\frac{y_a}{2}\overline{b}_{2p}\overline{V}_{a}\eta_p+2\sum_{p,q=5}^{n+2}\frac{\overline{a}_{pq}}{\overline A}\overline{V}_{p}\eta_q\\
&:=\widehat{G}
\end{split}\end{align}
in $B_{\tilde r'_1}^{n+2}(0)\setminus \left( \{0\}^4\times \overline{B_{\tilde r'_1}^{n-2}(0)}\right)$.
Then, $\widehat{G}\in L^{2}(B_{\tilde r'_1}^{n+2}(0))$ due to the fact
$$|D_y\overline{V}|\lesssim  |y'|^{-\frac{3}{2}}.$$
Moreover, $\widehat G\in L^{4(n+2)}(B_{\frac 34\widetilde r_1'}^{n+2}(0))$, which follows from $D\eta\equiv 0$ and $\widehat G=\overline{G}$ in $B_{\frac 34\widetilde r_1'}^{n+2}(0)$.

Let $W$ be the solution of the following Dirichlet problem:
 \begin{align}
\label{eq-W2}LW &= \widehat{G} \quad\text{in }B_{\tilde r'_1}^{n+2}(0),\\
\label{eqbdy-W2}W&=w \quad\text{on }\partial B_{\tilde r'_1}^{n+2}(0).
\end{align}
\eqref{eq-W2}-\eqref{eqbdy-W2} admits a unique solution
 $W\in W^{2,2}(B_{\tilde r'_1}^{n+2}(0))\bigcap W^{2,4(n+2)}_{loc}(B_{\frac 34\widetilde r_1'}^{n+2}(0))$. By Sobolev embedding theorem, we have $W\in C^{1,\frac{3}{4}}\left(\overline{B_{\frac{\tilde r'_1}{2}}^{n+2}(0)}\right)$. 

 For any $\epsilon>0$, by the maximum principle, we have
\begin{equation}
w+\epsilon|y'|^{-2}\geq W \geq w-\epsilon|y'|^{-2}\quad\text{in } B_{\tilde r'_1}^{n+2}(0)\setminus \left(\{0\}^4\times \overline{B_{\tilde r'_1}^{n-2}(0)}\right).
\end{equation}
Then, by taking $\varepsilon\rightarrow 0^+$,  we conclude that $w=W$. 
In particular, \begin{equation}\label{0726-01}
v_{12}=W\in C^{1,\frac{3}{4}}(\overline{B_{\frac{ \tilde r'_{1}}{2}}^{n+2}(0))}.
\end{equation}
 Back to the original coordinates, $$v_{12}\in C^{\frac{1}{2}}\left([0,\frac{\tilde r'^2_1}{8n}]^2\times [-\frac{\tilde r'_1}{\sqrt{2n}}, \frac{\tilde r'_1}{\sqrt{2n}}]^{n-2}\right).$$
A Similar argument shows that $$v_{12}\in C^{\frac{1}{2}}([0,\frac{\tilde r'^2_1}{8n}]^2\times \big(x''+[-\frac{\tilde r'_1}{\sqrt{2n}}, \frac{\tilde r'_1}{\sqrt{2n}}]^{n-2}\big),\quad x''\in [-\frac{3}{2},\frac{3}{2}]^{n-2}.$$

Hence, we have $v_{12}\in C^{\frac{1}{2}}([0,2r_{1*}]^2\times [-\frac{3}{2},\frac{3}{2}]^{n-2})$ for $r_{2*}=\frac{\tilde r'^2_1}{16n}$.
\end{proof}
\begin{remark}\label{I12-regul}
\eqref{0726-01} in Lemma \ref{lemma-I12} also tells us that
$$\sqrt{x_a}D_{x_a}I_{12},\	D_{x_p}I_{12}\in C^{\frac{1}{4}}([0,2r_{1*}]^2\times[-2,2]^{n-2}),\	a=1,2,\	 p=3, \cdots ,n.$$
\end{remark}

\begin{theorem}\label{thm-I-C2}
There holds that
\begin{equation}\label{IC2}
I\in C^{2,\frac{1}{8}}([0,\frac{3}{2}r_{2*}]^2\times[-\frac{5}{4},\frac{5}{4}]^{n-2}).
\end{equation}
Moreover, for any multi-index $\beta,\gamma$,
\begin{equation}\label{4.35}
D^{\beta}_{x'}D^{\gamma}_{x''}I, \sqrt{ x_a}D_{x_{p}}D^{\beta}_{x'}D^{\gamma}_{x''}I\in C^{\frac{1}{8}}([0,\frac{3}{2}r_{2*}]^2\times[-\frac{5}{4},\frac{5}{4}]^{n-2}), \quad\text{when } |\beta|+\frac{|\gamma|}{2}= 2,
\end{equation}
\begin{equation}\label{4.36}
 x_aD^{\beta}_{x'}D^{\gamma}_{x''}I,  \sqrt{ x_a}x_bD_{x_{p}}D^{\beta}_{x'}D^{\gamma}_{x''}I\in C^{\frac{1}{8}}([0,\frac{3}{2}r_{2*}]^2\times[-\frac{5}{4},\frac{5}{4}]^{n-2}), \quad\text{when }
2<|\beta|+\frac{|\gamma|}{2}\leq 3,
\end{equation}
 and
\begin{equation}\label{4.37}
 x_ax_{b}D^{\beta}_{x'}D^{\gamma}_{x''}I\in C^{\frac{1}{8}}([0,\frac{3}{2}r_{2*}]^2\times[-\frac{5}{4},\frac{5}{4}]^{n-2}) \quad\text{when }  3< |\beta|+\frac{|\gamma|}{2}\leq 4,
 \end{equation}
where $a,b=1,2$, $p=3, \cdots ,n$.
\end{theorem}

\begin{proof}
For $\beta+\gamma=(1,1,0, \cdots ,0)$, it is exactly Lemma \ref{lemma-I12}.
In order to show the remaining cases of \eqref{4.35}-\eqref{4.37}, it is enough to prove
\begin{equation}\label{I-growth1}
| D^{\beta}_{x'}D^{\gamma}_{x''}I(x)| \lesssim_{|\beta|+|\gamma|}  |x'|^{1/8+2-|\beta|-\frac{|\gamma|}{2}},\quad \text{in}\quad  \big([0,\frac{3}{2}r_{1*}]^2\setminus\{0\}^2\big)\times[-\frac{5}{4}, \frac{5}{4}]^{n-2}
 \end{equation}
for  $ |\beta|+\frac{|\gamma|}{2}\geq 2$ and $\beta+\gamma\neq (1,1,0, \cdots ,0)$.

Without loss of generality, we assume $x''=0$.
 Near the origin, Lemma \ref{lemma-I12} implies
\begin{align*}
I(y',y'')=& y_1y_2\int_0^1\int_0^1 I_{12}(s_1y_1,s_2y_2,y'')ds_1ds_2\\
=& y_1y_2\int_0^1\int_0^1\big[ I_{12}(0)+O( |y'|^{\frac{1}{2}}) +O(|y''|^{\frac{1}{2}})\big]ds_1ds_2\\
=& y_1y_2I_{12}(0)+O( |y'|^{2+\frac{1}{2}}) +O(|y'|^2|y''|^{\frac{1}{2}}).
\end{align*}
 Take $\lambda\in (0,\frac{3}{2}r_{2*}]$, and
set $I_{\lambda}(y)=\frac{I(\lambda y',\sqrt{\lambda} y'')}{\lambda}$, for any $y \in [0,1]^2\times[-1,1]^{n-2}$.
Then,
\begin{equation*}
 |I_{\lambda}(y)-\lambda I_{12}(0)y_1y_2|\lesssim \lambda^{1+\frac{1}{4}}.
\end{equation*}
Let $\widetilde{I}(y)=I_{\lambda}-\lambda I_{12}(0)y_1y_2$. Then by \eqref{esti-1-alpha-I}, for any multi-index $\alpha$,
there holds
\begin{equation*}
|D_{y}^{\alpha} \widetilde{I}|\lesssim_{|\alpha|} 1,\quad \text{in}\quad [1/2,1]\times [0,1] \times [-1,1]^{n-2}.
\end{equation*}
By interpolation inequality, we have
\begin{equation*}
|D_{y}^{\alpha} \widetilde{I}|\lesssim_{|\alpha|}  \lambda^{1+\frac{1}{4}-\frac{1}{8}},\quad \text{in}\quad  [1/2,1]\times [0,1] \times [-1,1]^{n-2}.
\end{equation*}
Similarly, we have
\begin{equation*}
|D^{\alpha}_y \widetilde{I}|\lesssim_{|\alpha|}   \lambda^{1+\frac{1}{4}-\frac{1}{8}},\quad \text{in}\quad [0,1]\times [1/2,1] \times [-1,1]^{n-2}.
\end{equation*}
Then,
\begin{equation*}
|D^{\alpha}_{y} \widetilde{I}|\lesssim_{|\alpha|}  \lambda^{1+\frac{1}{8}},\quad \text{in}\quad \big([0,1]^2\setminus[0,\frac{1}{2})^2 \big)\times [-1,1]^{n-2}.
\end{equation*}

Therefore,
\begin{equation*}
\lambda^{|\beta|+\frac{|\gamma|}{2}}|D^{\beta}_{y'}D^{\gamma}_{y''}\big[I(y)-I_{12}(0)y_1y_2\big]|\lesssim_{|\beta|+|\gamma|} \lambda^{2+\frac{1}{8}}
\end{equation*}
in $\big([0,\lambda]^2\setminus[0,\frac{1}{2}\lambda)^2 \big)\times [-\sqrt{\lambda}, \sqrt{\lambda}]^{n-2}$.

Taking $\lambda=\max\{x_1,x_2\}$, $y'=\frac{x'}{\lambda}$ and $y''=0$, this implies \eqref{I-growth1} for $ |\beta|+\frac{|\gamma|}{2}\geq 2$ and $\beta+\gamma\neq (1,1,0, \cdots ,0)$.
\end{proof}

We now prove Theorem \ref{thmlocal-k} for $\mathfrak k=2$.
\begin{theorem}\label{thm-I-2-Smooth}
Let $u,v$ be given by \eqref{locpro2-1}-\eqref{locpro2-4}, and let $I$ be given at the beginning of this sub-section. There holds
\begin{equation}\label{reg-I}
I\in C^{\infty}([0, r_{2*}]^2\times[-1,1]^{n-2}).
\end{equation}
\end{theorem}
\begin{proof}
Set
$$\delta_l=2^{-3l+3},\	 a_{l}=1+2^{2-2l},\	 a_{l}'=1+2^{1-2l}.$$
For any positive integer $l\geq 2$, we will prove
\begin{equation}\label{I12-l-1}
D^{\beta}_{x'}D^{\gamma}_{x''}I, \sqrt{x_a}D_{x_p}D^{\beta}_{x'}D^{\gamma}_{x''}I\in C^{\delta_l}([0,a_lr_{2*}]^2\times[-a_l,a_l]^{n-2})\quad\text{when } |\beta|+\frac{|\gamma|}{2}\leq l,
\end{equation}
\begin{equation}\label{I12-l-2}
x_a D^{\beta}_{x'}D^{\gamma}_{x''}I, \sqrt{x_a}x_bD_{x_p}D^{\beta}_{x'}D^{\gamma}_{x''}I\in C^{\delta_l}([0,a_lr_{2*}]^2\times[-a_l,a_l]^{n-2}), \quad\text{when } |\beta|+\frac{|\gamma|}{2} \leq l+1
\end{equation}
and
\begin{equation}\label{I12-l-3}
x_a x_bD^{\beta}_{x'}D^{\gamma}_{x''}I\in C^{\delta_l}([0,a_lr_{2*}]^2\times[-a_l,a_l]^{n-2}), \quad\text{when } |\beta|+\frac{|\gamma|}{2}\leq l+2,\end{equation}
where $a,b=1,2$, $p=3, \cdots ,n$.

We prove \eqref{I12-l-1}-\eqref{I12-l-3}  by induction on $l$. For $l=2$, the result is established by Theorem \ref{thm-I-C2}. 
Assume the conclusion holds for some integer $l\geq 2$. We will now show that the conclusion also holds for $l+1$.

{\it Step 1.}
For any two positive integers $\beta_1,\beta_2\geq1$ and $\beta_1+\beta_2=|\beta|$, $|\beta|+\frac{|\gamma|}{2}=l-\frac{1}{2}$, $l>2$, we set $\widetilde{V}=D^{\beta}_{x'}D^{\gamma}_{x''}v$.
Similar to the proof of Lemma  \ref{lemma-I12}, $\widetilde{V}$ solves an equation in the form:
\begin{equation}\label{eq-widetilde-I}
A\big(x_1\widetilde{V}_{11}+\beta_1\widetilde{V}_{1}+x_2 \widetilde{V}_{22}+\beta_2\widetilde{V}_{2}\big)+\sum_{a=1}^{2}\sum_{p=3}^{n}x_a\frac{a_{ap}}{\sqrt{x_a}}\widetilde{V}_{ap}+\sum_{p,q=3}^{n}a_{pq}\widetilde{V}_{pq}+\sum_{p=3}^{n}B_p\widetilde{V}_{p}=\widetilde{G}\end{equation}
where $A$, $a_{ap}$, $a_{pq}$ and $B_{p}$ are the same functions in the proof of Lemma \ref{lemma-I12}, and  $\widetilde{G}$ is smooth with respect to  $x$ and the terms in \eqref{I12-l-1}-\eqref{I12-l-3}.  Hence,
$\widetilde{G}\in C^{\delta_l} ([0,a_lr_{2*}]^2\times[-a_l,a_l]^{n-2})$. Similar to the proof of Lemma \ref{lemma-I12}, one can take the following coordinates transformations:
\begin{equation}
x_{1}=\frac{y_1^2+ \cdots +y_{2\beta_1}^2}{4},\quad\quad x_2=\frac{y_{2\beta_1+1}^2+ \cdots +y_{2\beta_1+2\beta_2}^2}{4},
\end{equation}and
\begin{equation}
x_{p}=y_{2\beta_1+2\beta_2+p-2}, \quad\quad p=3,\cdots,n.
\end{equation}
Then, $\overline{V}(y)=\widetilde{V}(x)$ solves
\begin{align}\label{eq-widetilde-I-y}\begin{split}
L\overline{V}:=&\sum_{i=1}^{2\beta_1+2\beta_2}\partial_{y_iy_i}\overline{V}+\sum_{a=1}^{2\beta_1}\sum_{p=2\beta_1+2\beta_2+1}^{2\beta_1+2\beta_2+n-2}\frac{y_a}{2}\frac{\overline{b}_{1p}}{\overline{A}}\overline{V}_{ap}+\sum_{a=2\beta_1+1}^{2\beta_1+2\beta_2}\sum_{p=2\beta_1+2\beta_2+1}^{2\beta_1+2\beta_2+n-2}\frac{y_a}{2}\frac{\overline{b}_{2p}}{\overline{A}}\overline{V}_{ap}\\
&+\sum_{p,q=2\beta_1+2\beta_2+1}^{2\beta_1+2\beta_2+n-2}\frac{\overline{a}_{pq}}{\overline{A}}\overline{V}_{pq}+\sum_{p=2\beta_1+2\beta_2+1}^{2\beta_1+2\beta_2+n-2}\frac{\overline{B}_p}{\overline{A}}\overline{V}_{p}\\
&=\overline{G},
\end{split}\end{align}
in   $\bigg(\big(B^{2\beta_1}_{2\sqrt{a_{l}r_{2*}}}(0)\times B^{2\beta_2}_{2\sqrt{a_{l}r_{2*}}}(0)\big)\setminus \{0\}^{2\beta_1+2\beta_2} \bigg)\times[-a_{l},a_{l}]^{n-2}$,
where $$b_{ap}(y)=\frac{a_{ap}}{\sqrt{x_a}}(x),\quad \overline{b}_{1p}(y)=b_{1x_{p-2\beta_1-2\beta_2}}(x),$$
and $\overline{b}_{2p}(y)$,
$\frac{\overline{a}_{pq}}{\overline{A}}$, $\frac{\overline{B}_p}{\overline{A}}$, $\overline{G}$ are defined in the same way.
This is a uniformly elliptic equation with $ C^{\delta_l}(B^{2\beta_1}_{2\sqrt{a_{l}r_{2*}}}(0)\times B^{2\beta_2}_{2\sqrt{a_{l}r_{2*}}}(0)\times[-a_{l},a_{l}]^{n-2})$ coefficients.
Since $$\overline{G}\in C^{\delta_l}(B^{2\beta_1}_{2\sqrt{a_{l}r_{2*}}}(0)\times B^{2\beta_2}_{2\sqrt{a_{l}r_{2*}}}(0)\times[-a_{l},a_{l}]^{n-2}),$$
by standard elliptic regularity results, we get
$$\overline{V}\in C^{2,\delta_l}(B^{2\beta_1}_{2\sqrt{a_{l}'r_{2*}}}(0)\times B^{2\beta_2}_{2\sqrt{a_{l}'r_{2*}}}(0)\times[-a_{l}',a_{l}']^{n-2}).$$
 Changing
back to $x$, one gets
$$D\widetilde{V}, \sqrt{x_ax_b}D_{x_ax_b}\widetilde{V},\sqrt{x_a}D_{x_ax_p}\widetilde{V},D_{p}^2\widetilde{V}\in C^{\frac{\delta_l}{2}}([0,a_l'r_{2*}]^2\times[-a_l',a_l']^{n-2}),$$ $a,b=1,2$, $p=3,\cdots,n$.

{\it Step 2.}
For any two positive integers $\beta_1,\beta_2\geq1$ satisfying $\beta_1+\beta_2=|\beta|$, $|\beta|+\frac{|\gamma|}{2}=l$,
set $\widehat{V}=D^{\beta}_{x'}D^{\gamma}_{x''}v$. Arguing  as in Step 1, one gets

$$D\widehat{V}, \sqrt{x_ax_b}D_{x_ax_b}\widehat{V},\sqrt{x_a}D_{x_ax_p}\widehat{V},D_{p}^2\widehat{V}\in C^{\frac{\delta_l}{2}}([0,a_{l+1}r_{2*}]^2\times[-a_{l+1},a_{l+1}]^{n-2}),$$  $a=1,2$, $p=3, \cdots ,n$.

Noting that $I=v-F$, we have
$$D^{\beta}_{x'}D^{\gamma}_{x''}I\in C^{\frac{\delta_l}{2}}([0,a_{l+1}r_{2*}]^2\times[-a_{l+1},a_{l+1}]^{n-2})$$
 for any multi-index $\beta,\gamma$ with $ |\beta|+\frac{|\gamma|}{2}\leq l+1$ and $(1,1,0, \cdots ,0)\leq \beta$.
In particular, there holds
$$D^{\beta}_{x'}D^{\gamma}_{x''}I_{12}\in C^{\frac{\delta_l}{2}}([0,a_{l+1}r_{2*}]^2\times[-a_{l+1},a_{l+1}]^{n-2}),\quad \forall |\beta|+\frac{|\gamma|}{2}\leq l-1.$$

\smallskip

{\it Step 3.}
We now prove the remaining cases. As in the proof of Theorem \ref{thm-I-C2}, we will estimate the growth of $D^{\beta}_{x'}D^{\gamma}_{x''}I(x)$ with respect to $x'$. We then use Lemma \ref{lemma-Growth-Regularity} and Lemma \ref{lemma-Growth-Regularity2} to obtain the desired conclusions.
Without loss of generality, we assume $x''=0$.

Let $P_{l+1}$ be a polynomial satisfying
\begin{equation}
D^{\beta}_{x'}D^{\gamma}_{x''}P_{l+1}(0)=0,\quad\text{when } (1,1,0, \cdots ,0)\nleq \beta,
\end{equation}
and
\begin{equation}
D_{x_1}D_{x_2}D^{\beta}_{x'}D^{\gamma}_{x''}P_{l+1}(0)=D^{\beta}_{x'}D^{\gamma}_{x''}I_{12}(0),\quad\text{for any } \beta, \gamma \text{ with } |\beta|+\frac{|\gamma|}{2}\leq l-1,
\end{equation}
\begin{equation}
D^{\beta}_{x'}D^{\gamma}_{x''}P_{l+1}(0)=0,\quad\text{when } |\beta|+\frac{|\gamma|}{2}>l+1.
\end{equation}
Let $\widetilde{I}=I-P_{l+1}$.
Near the origin, by Taylor's expansion, we have
\begin{align*}
&I(y',y'')\\
=& y_1y_2\int_0^1\int_0^1 I_{12}(s_1y_1,s_2y_2,y'')ds_1ds_2\\
=&y_1y_2\bigg(\sum_{|\tilde \beta|\leq l-1}\frac{D^{\tilde \beta}_{y'}I_{12}(0,0,y'')}{\tilde \beta!(\tilde \beta_1+1)(\tilde \beta_2+1)}y'^{\tilde \beta}+O( |y'|^{l-1+\frac{\delta_l}{2}})\bigg)\\
=&y_1y_2\bigg(\sum_{|\tilde \beta|\leq l-1}\sum_{\frac{|\tilde \gamma|}{2}\leq l-1-|\tilde \beta|}\frac{D^{\tilde \gamma}_{y''}D^{\tilde \beta}_{y'}I_{12}(0)}{\beta!(\beta_1+1)(\beta_2+1)\gamma!}y'^{\beta}y''^{\gamma}+
O(  |y'|^{ |\tilde \beta|}  |y''|^{2l-2-2 |\tilde \beta|+\delta_l} )+O( |y'|^{l-1+\frac{\delta_l}{2}})\bigg)\\
=&y_1y_2\bigg(\sum_{|\tilde \beta|\leq l-1}\sum_{\frac{|\tilde \gamma|}{2}\leq l-1-|\tilde \beta|}\frac{D^{\tilde \gamma}_{y''}D^{\tilde \beta}_{y'}P_{l+1,12}(0)}{\beta!(\beta_1+1)(\beta_2+1)\gamma!}y'^{\beta}y''^{\gamma}+
O( |y'|^{ |\tilde \beta|}  |y''|^{2l-2-2 |\tilde \beta|+\frac{\delta_l}{2}} )+O(  |y'|^{l-1+\frac{\tilde \delta_l}{2}})\bigg)\\
=&P_{l+1}(y)+\sum_{|\widetilde \beta|\le l-1}
O(  |y'|^{ |\tilde \beta|+2}  |y''|^{2l-2-2 |\tilde \beta|+\frac{\delta_l}{2}} )+O(  |y'|^{l+1+\frac{\delta_l}{2}}).
\end{align*}
By Lemma \ref{lemma-Growth-Regularity} and Lemma \ref{lemma-Growth-Regularity2}, it is enough to show that  
\begin{equation}\label{I-growth-l-1}
| D^{\beta}_{x'}D^{\gamma}_{x''}I| =| D^{\beta}_{x'}D^{\gamma}_{x''}[I-P_{l+1}(x)]| \lesssim_{L(|\beta|+|\gamma|,l,n)} |x'|^{\frac{\delta_l}{8}+l+1-|\beta|-\frac{|\gamma|}{2}}
 \end{equation}
 for $|\beta|+\frac{|\gamma|}{2}\geq l+\frac 32$ or
 $ |\beta|+\frac{|\gamma|}{2}=l+1$ and $\beta\ngeq (1,1,0, \cdots ,0)$,
 where $L(|\beta|+|\gamma|,l,n)$ is some positive integer depending only on $|\beta|+|\gamma|, l$ and $n$.

 Take $\lambda\in (0,a_{l+1}r_{2*}]$.
Set $I_{\lambda}(y',y'')=\frac{I(\lambda y',\sqrt{\lambda} y'')}{\lambda}$.
Then, for any $ y\in [0,1]^2\times[-1,1]^{n-2}$,
\begin{equation}
|I_{\lambda}(y',y'')- \frac{P_{l+1}(\lambda y',\sqrt{\lambda} y'')}{\lambda}|\lesssim_{l+1} \lambda^{l+\frac{\delta_l}{4}}.
\end{equation}

Let $$\widetilde{I_{\lambda}}(y)=I_{\lambda}(y',y'')-\frac{ P_{l+1}(\lambda y',\sqrt{\lambda} y'')}{\lambda}.$$

As in the proof in Theorem \ref{thm-I-C2}, we can obtain  
\begin{equation*}
|D^{\alpha}_{y} \widetilde{I}|\lesssim_{L(|\alpha|,l,n)} \lambda^{l+\frac{\delta_l}{4}-\frac{\delta_l}{8}}
\end{equation*}
in $\big([0,1]^2\setminus[0,\frac{1}{2})^2 \big)\times [-1,1]^{n-2}$.

Therefore,
\begin{equation*}
\lambda^{|\beta|+\frac{|\gamma|}{2}}|D^{\beta+\gamma}\big[I(y)-P_{l+1}(y)\big]|\lesssim_{L(|\beta|+|\gamma|,l,n)} \lambda^{l+1+\frac{\delta_l}{8}}
\end{equation*}
in $\big([0,\lambda]^2\setminus[0,\frac{1}{2}\lambda)^2 \big)\times [-\sqrt{\lambda},\sqrt{\lambda}]^{n-2}$.
Back to the original coordinates, this implies \eqref{I-growth-l-1}.

This shows that the inductive hypothesis holds for $l+1$ and hence we complete the proof.
\end{proof}

\begin{remark}
This sub-section is devoted to the proof that $v$ is smooth up to the co-dimension 2 boundary. The key issue is to derive the
corresponding estimates for $I$.
In fact, we prove if $u$ is the solution of \eqref{locpro2-1}-\eqref{locpro2-4}, then there holds,
\begin{equation}\label{4-01-24-1}
|D_{x'}^\beta D_{x''}^\gamma I|_{L^{\infty}([0, r_{2*}]^2\times[-1,1]^{n-2})}\lesssim_{|\beta|+|\gamma|} 1
\end{equation}
 for some $r_{2*}\approx 1$. Estimate \eqref{4-01-24-1} highly relies on Theorem \ref{thm3-803}. By rescaling, estimate \eqref{4-01-24-1} near the co-dimension 2 boundary can be reduced to the estimate near the co-dimension 1 boundary. Moreover, in this reduction, all the quantities involved in Theorem \ref{thm3-803} remain uniformly bounded. This implies the estimate \eqref{4-01-24-1}.
\end{remark}

\begin{remark}
The term $\displaystyle\sum_{p=3}^{n}b_p\widetilde{V}_{p}$ on the left-hand side of \eqref{eq-widetilde-I} and $\displaystyle \sum_{p=3}^{n}\overline{b}_p\overline{V}_{p}$ on the left-hand of \eqref{eq-widetilde-I-y} are only needed when $l=2$.  For $l>2$, these terms can be absorbed  into the right-hand side of the corresponding equation, as they satisfy the required estimates in this case.
\end{remark}

\subsection{Smoothness up to $(n-\mathfrak k)$-face}

Let $\mathfrak k$ be an integer with $3\leq \mathfrak k\leq n$.
 In this section, we use induction to prove the global regularity of $v$ up to codimension-$\mathfrak k$ boundary.
 Throughout the section, we will use superscripts $\alpha$, $\beta$, $\gamma$  to denote the multi-index.
 Specially, 
 \begin{itemize}
 	 	\item[$\bullet$] $\alpha$ will generally represent multi-indices without restriction to specific subsets of variables. 
 	\item[$\bullet$] $\beta$  will denote multi-indices with respect to the first $\mathfrak k$ variables $x_1, x_2, \cdots ,x_{\mathfrak k}$.
 	\item[$\bullet$] $\gamma$   will denote multi-indices with respect to the last $n-\mathfrak k$ variables $x_{\mathfrak k+1}, \cdots ,x_n$.
\end{itemize}
Furthermore, we will denote $x=(x',x'')$, where 
\begin{itemize}
	\item[$\bullet$] $x'=(x_1,\cdots,x_\mathfrak k)$ represents the first $\mathfrak k$ coordinate components of $x$,
\item[$\bullet$] $x''=(x_{\mathfrak k+1}, \cdots ,x_{n})$ represents the last $n-\mathfrak k$ coordinate components of $x$.
\end{itemize}

Set $I=v-F$.
The main goal of this section is to demonstrate that
\begin{equation*}
I\in C^\infty([0,1]^{\mathfrak k}\times[-1,1]^{n-\mathfrak k}).
\end{equation*}
We will accomplish this by induction on $k$.
To achieve this, we first need to improve Lemma \ref{lem-k-quar-low} to get a more refined growth estimate for $I$ near the codimension-$\mathfrak k$ boundary.
As a preliminary step, we will prove the following lemma.
\begin{lemma}\label{lemma-growth1}
Let $n\geq 3$ and $n\geq \mathfrak k\geq 3$.
Suppose there exist $C_{l}>0$ and $2\leq l< \mathfrak k$ such that
$$|I|\leq C_{l}\sum_{1\leq i_1< \cdots <i_l\leq k}x_{i_1} \cdots x_{i_l},$$
in $[0,r_{l}]^{ \mathfrak k}\times[-5/2,5/2]^{n- \mathfrak k}$ for some $r_l>0$.
 Then, there exists $r_*\approx 1$, such that for $r\in(0, r_*]$ and 
for any $\epsilon\in(0,1)$, the following estimate holds:
 \begin{equation}\label{refine-growth1-0}|I|\lesssim C_{\epsilon} \frac{r^{l-1-\epsilon}}{r^{ \mathfrak k-1}}x_{1} \cdots x_{ \mathfrak k}, \end{equation}
on $\big(\partial[0,r]^{ \mathfrak k}\big)\times[-9/4,9/4]^{n- \mathfrak k}$.
\end{lemma}

\begin{proof}
Note that \eqref{refine-growth1-0} obviously holds if $x_i=0$ for some $i\in\{1, \cdots  \mathfrak k\}$.
We consider the case $|\bar x'|>0$. Without loss of generality, we assume $0<\bar x_1\leq \bar x_2\leq \cdots \leq \bar x_{ \mathfrak k}=r$.

Set $r_1=\min\{ \frac{1}{64}, \frac{r_{l}}{2}\}$.
Take $\lambda\in (0, r_1]$. For any $\bar x''\in [-9/4, 9/4]^{n-k}$,
define
\begin{equation*}
u_\lambda(y)\triangleq\frac{u(\lambda y' ,\bar x''+\sqrt{\lambda}y'' )-\sum\limits_{i=1}^{ \mathfrak k}(\lambda \ln  \lambda)y_i}{\lambda},
\end{equation*}
\begin{equation*}
F_{\lambda}(y)\triangleq \frac{F(\lambda y' ,\bar x''+\sqrt{\lambda}y'' )}{\lambda},
\end{equation*} and
\begin{equation*}
I_{\lambda}(y)\triangleq \frac{I(\lambda y' ,\bar x''+\sqrt{\lambda}y'' )}{\lambda}.
\end{equation*}
By subtracting an affine function, we may assume
\begin{equation*}
F\big(0,\bar x''\big)=|D F\big(0,\bar x''\big)|=0.
\end{equation*}
Then,
$$|I_{\lambda}|\lesssim  \lambda^{l-1},\quad \text{in}\quad [0,2]^{k-1}\times[\frac{1}{2},2]\times[-2,2]^{n-k}.$$
By a similar argument as in Lemma \ref{lemma-cd2-growth1},  one knows  the uniform convexity of $u_{\lambda}$ on $\{0\}^{ \mathfrak k-1}\times[\frac{1}{2},2] \times [-2,2]^{n- \mathfrak k}$ is independent of $\lambda$ and $\bar x''\in[-2.5, 2.5]^{n-2}$ when $\lambda \leq r_2$,
where $r_2$ is some positive constant depending only on $F$. It can be checked directly that the assumptions of  Theorem \ref{thmlocal-k}
are fulfilled for $ \mathfrak k-1$. Hence by induction assumptions, there exists  $r_3\approx 1$ such that
\begin{equation}\label{840}
|D_y^\alpha I_\lambda|\lesssim_{|\alpha|} 1,\quad \text{in}\quad [0,r_3]^{ \mathfrak k-1}\times [\frac{1}{2},2] \times [-2,2]^{n- \mathfrak k}.
\end{equation}
Repeating the above argument yields that there exists a neighbourhood $N(\Gamma_0)$ of $\Gamma_0$,
$$\Gamma_0=\bigcup_{a=1}^{ \mathfrak k-1}([0,2]^{ \mathfrak k-1}\times[\frac{1}{2},2] \times [-2,2]^{n- \mathfrak k}\bigcap\{x|x_a=0\}\big)$$
 such that estimate \eqref{840} holds in  $N(\Gamma_0)$.

A similar argument as in Lemma \ref{lemc1-1} implies that there exists a positive constant $r_4\approx 1$ such that  the strict convexity of $u_{\lambda}(y)$ in $([0,2]^{ \mathfrak k-1}\times[\frac{1}{2},2]\times[-2,2]^{n- \mathfrak k})\backslash N(\Gamma_0)$
is independent of $\lambda$ and $x''$ for $\lambda\in (0,r_4]$. Hence
$$|D^{\alpha} I_{\lambda}|\lesssim_{|\alpha|}1 ,\quad \text{in}\quad  [0,\frac{3}{2}]^{ \mathfrak k-1}\times[\frac{3}{4},\frac{3}{2}]\times[-1,1]^{n- \mathfrak k}.$$

Then for any $\epsilon \in (0, 1)$, by interpolation inequality Lemma \ref{lemma-interpolation}, there holds   
$$|D^{\alpha} I_{\lambda}|\lesssim_{L(|\alpha|,l,n,\mathfrak k,\epsilon)}  \lambda^{l-1-\epsilon},$$
 where $L(|\alpha|,l,n, \mathfrak k,\epsilon)$ is some positive integer depending only on $|\alpha|, l, n, \mathfrak k$ and $\epsilon$. 

Note that \begin{equation}
 I_{\lambda}(y)=y_1 \cdots y_{ \mathfrak k-1}\int_0^1 \cdots \int_0^1I_{\lambda,1, \cdots  \mathfrak k-1}(s_1y_1, \cdots ,s_{ \mathfrak k-1}y_{ \mathfrak k-1},y_{ \mathfrak k}, \cdots ,y_n)ds_1 \cdots ds_{ \mathfrak k-1},
\end{equation}
in $[0,\frac{3}{2}]^{ \mathfrak k-1}\times[\frac{3}{4},\frac{3}{2}]\times[-\frac{3}{2},\frac{3}{2}]^{n- \mathfrak k}$.
Hence,
 \begin{equation}
 |I_{\lambda}(y)|\lesssim  \lambda^{l-1-\epsilon}y_1 \cdots y_{ \mathfrak k-1},\quad \text{in}\quad [0,\frac{3}{2}]^{ \mathfrak k-1}\times[\frac{3}{4},\frac{3}{2}]\times[-\frac{3}{2},\frac{3}{2}]^{n- \mathfrak k}.
\end{equation}

Take $\lambda=x_{ \mathfrak k}=r$, $y_i=\frac{x_i}{x_{ \mathfrak k}}=\frac{x_i}{r}$, $i=1, \cdots , \mathfrak k$.
Then,
\begin{align*}
|I(x)|=& \lambda |I_{\lambda}(y_1, \cdots ,y_{\mathfrak k-1},1,0,..,0)|\\
\lesssim &  \lambda^{l-1-\epsilon} \frac{x_1}{r} \cdots \frac{x_{ \mathfrak k-1}}{r} x_{ \mathfrak k} \\
\lesssim &  \frac{r^{l-1-\epsilon}}{r^{ \mathfrak k-1}}x_{1} \cdots x_{ \mathfrak k}.
\end{align*}

Let $r_*=\min(r_1,r_2,r_3,r_4)$, we have \eqref{refine-growth1-0} when $r\leq r_*$.
\end{proof}

We point out that the constant $C_{\epsilon}$ in the growth estimate \eqref{refine-growth1-0} is independent of $r$.
In the subsequent applications, we will choose $\epsilon=\frac{1}{100}$.

The following growth estimate plays an important role and is essentially needed in the proof of Lemma \ref{lemma-I12k} later on. Its proof is based on the maximum principle and it is tedious.

\begin{lemma}\label{lemma-growth2}
Let $n\geq 3$, $3\leq \mathfrak  k\leq n$. Suppose there exists a positive constant $C_2>0$  such that
$$|I|\leq C_{2}\sum_{1\leq i<j\leq  \mathfrak k}x_{i}x_{j},$$ in $[0,1/2]^{ \mathfrak k}\times[-5/2,5/2]^{n- \mathfrak k}$. Then, there exists $r_{ \mathfrak k}^*\approx 1$ such that
 \begin{equation}\label{refine-growth1}|I|\lesssim  x_{1} \cdots x_{ \mathfrak k}, \end{equation}
in $[0,r_{ \mathfrak k}^*]^{ \mathfrak k}\times[-2, 2]^{n- \mathfrak k}$.
\end{lemma}

\begin{proof}
 In the proof, 
 \begin{itemize}
 	\item[$\bullet$] subscripts $a, b, c$ denote taking derivatives only with respect to the first $ \mathfrak k$ variables $x_1, \cdots ,x_{ \mathfrak k}$,
 	\item[$\bullet$]subscripts $p, q$ denote taking derivatives only with respect to the last $n- \mathfrak k$ variables $x_{ \mathfrak k+1}, \cdots ,x_n$.
 	\end{itemize}
We will prove for any integer $l$ with $2\leq l\leq  \mathfrak k$, there exists $r_l\approx 1$,
\begin{equation}\label{refine-growth3}
|I|\le C_l \sum_{1\leq i_1< \cdots <i_l\leq  \mathfrak k} x_{i_1} \cdots x_{i_l},
\end{equation}
in $[0,r_{l}]^{ \mathfrak k}\times[-\frac{5}{2}+\frac{l}{2 \mathfrak k},\frac{5}{2}-\frac{l}{2 \mathfrak k}]^{n- \mathfrak k}$ for some $r_l\leq r_*$.

Suppose the result holds for some $l$,  $2\leq l\leq  \mathfrak k-1$.  We always assume that $r\leq r_*$.
By lemma \ref{lemma-growth1}, for any $\epsilon\in(0,\frac{1}{10})$, we have
\begin{equation}
|I|\lesssim A_{\epsilon} \frac{r^{1-\epsilon}}{r^{ \mathfrak k-1}}x_{1} \cdots x_{ \mathfrak k},
\end{equation}
on $\big(\partial[0,r]^{ \mathfrak k}\big)\times[-\frac{5}{2}+\frac{4l-4}{ 8\mathfrak k}, \frac{5}{2}-\frac{4l-4}{8 \mathfrak k}]^{n- \mathfrak k}$.
Hence, for any integer $l$ with $2\leq l\leq  \mathfrak k-1$,
\begin{equation}\label{bdy-comp1}
|I|\lesssim A_{\epsilon} \frac{r^{1-\epsilon}}{r^{ \mathfrak k-1}}\sum_{1\leq i_1< \cdots <i_{l+1}\leq  \mathfrak k}x_{i_1} \cdots x_{i_{l+1}}r^{ \mathfrak k-l-1}\lesssim A_{\epsilon} \frac{r^{1-\epsilon}}{r^{l}}\sum_{1\leq i_1< \cdots <i_{l+1}\leq  \mathfrak k}x_{i_1} \cdots x_{i_{l+1}},
\end{equation}
on $\big(\partial[0,r]^{ \mathfrak k}\big)\times[-\frac{5}{2}+\frac{4l-4}{ 8\mathfrak k}, \frac{5}{2}-\frac{4l-4}{8 \mathfrak k}]^{n- \mathfrak k}$.

Step 1: The following estimate holds:
\begin{equation}\label{refine-growth4}|I|\leq C_{l+1,1}\sum_{1\leq i_1< \cdots <i_{l+1}\leq\mathfrak  k}x_{i_1} \cdots x_{i_{l+1}}(-\sum_{j=1}^{l+1}\ln x_{i_{j}}),
\end{equation}
in $[0,r_{l}']^{\mathfrak k}\times[-5/2+\frac{2l-1}{4\mathfrak k},5/2-\frac{2l-1}{4\mathfrak k}]^{n-\mathfrak k}$.
After an affine transformation, it suffices to prove \eqref{refine-growth4} at $(x',0'')$.
Denote \begin{equation*}
m=\sum\limits_{i=1}^{\mathfrak k }x_i\ln  x_i+F.
\end{equation*}
 By Lemma \ref{lem-pre}, the following estimate holds 
\begin{equation}\label{mp-h}|x_1 \cdots x_{\mathfrak k}\det(m_{ij})_{n\times n}-h(x)|\leq C_{*}x_{1} \cdots x_{\mathfrak k}\end{equation}
 in $[0,1]^{\mathfrak k}\times[-3, 3]^{n-\mathfrak k}$.
Set \begin{align*} I^{+}=&A\frac{r^{1-\epsilon}}{r^{l}}\sum_{1\leq i_1< \cdots <i_{l+1}\leq \mathfrak k}x_{i_1} \cdots x_{i_{l+1}}(-\sum_{j=1}^{l+1}\ln x_{i_{j}})\\
+& C_{l}\sum_{p=\mathfrak k+1}^nx_{p}^2\sum_{1\leq i_1< \cdots <i_l\leq \mathfrak k}x_{i_1} \cdots x_{i_l},\end{align*}
where $C_l$ is given by \eqref{refine-growth3},  $A\geq \max\{A_{\epsilon}, 1\}$ is some constant to be determined, and $A$ is independent of $r$.

By  direct computation, we have

\begin{equation*}I^{+}_{aa}=-\frac{1}{x_a}A\frac{r^{1-\epsilon}}{r^{l}}\sum_{1\leq i_1< \cdots <i_{l}\leq  \mathfrak k, i_j\neq a}x_{i_1} \cdots x_{i_{l}}. \end{equation*}
\begin{align*} I^{+}_{ab}=&A\frac{r^{1-\epsilon}}{r^{l}}\sum_{1\leq i_1< \cdots <i_{l-1}\leq \mathfrak  k, i_j\neq a, b}x_{i_1} \cdots x_{i_{l-1}}(-\sum_{j=1}^{l-1}\ln x_{i_{j}})\\
+&A\frac{r^{1-\epsilon}}{r^{l}}\sum_{1\leq i_1< \cdots <i_{l-1}\leq  \mathfrak k, i_j \neq a, b}x_{i_1} \cdots x_{i_{l-1}}(-\ln x_a-\ln x_b -2)\\
+& C_{l}\sum_{p= \mathfrak k+1}^n x_{p}^2\sum_{1\leq i_1< \cdots <i_{l-2}\leq  \mathfrak k,i_j \neq a, b}x_{i_1} \cdots x_{i_{l-2}},\end{align*}
for $a\neq b$.
\begin{equation*}I^{+}_{pq}=2C_l\delta_{pq}\sum_{1\leq i_1< \cdots <i_{l}\leq \mathfrak  k}x_{i_1} \cdots x_{i_{l}}. \end{equation*}
\begin{equation*}I^{+}_{ap}=2C_lx_p\sum_{1\leq i_1< \cdots <i_{l-1}\leq \mathfrak  k,i_j\neq a}x_{i_1} \cdots x_{i_{l-1}}. \end{equation*}

We will prove $I\leq I^{+}$ in $[0,r]^{ \mathfrak k}\times[-1,1]^{n- \mathfrak k}$ by maximum principle.
Denote
\begin{equation*}
m^+=m+I^+.
\end{equation*}
It is easy to check that
$\sqrt{x_ax_b}I^{+}_{ab}, \sqrt{x_a}I^{+}_{ap}, I^{+}_{pq}$ tends to $0$ uniformly as $r$ tends to $0$.
Recall that we use the notation $\mathbb M_{ m^+}$ to denote the matrix
\begin{equation}\label{matrix-m+}
 \begin{pmatrix}
     x_1m^+_{11}&    \cdots &\sqrt{x_1x_{ \mathfrak k}}m^+_{1\mathfrak k} &\sqrt{x_1}m^+_{1( \mathfrak k+1)}&\cdots & \sqrt{x_1}m^+_{1n}\\[3pt]
      \cdots &\cdots &\cdots &\cdots &\cdots  &\cdots \\[3pt]
   \sqrt{x_1x_{ \mathfrak k}}m^+_{ \mathfrak k1}& \cdots&  x_{ \mathfrak k} m^+_{ \mathfrak k \mathfrak k} &\sqrt{x_{ \mathfrak k}} m^+_{ \mathfrak k( \mathfrak k+1)}  &\cdots & \sqrt{x_{ \mathfrak k}} m^+_{\mathfrak kn}\\[3pt]
  \sqrt{x_1}m^+_{( \mathfrak k+1)1} & \cdots& \sqrt{x_{ \mathfrak k}} m^+_{( \mathfrak k+1) \mathfrak k}& m^+_{( \mathfrak k+1)( \mathfrak k+1)} &\cdots & m^+_{( \mathfrak k+1)n}\\[3pt]
    \cdots &\cdots &\cdots &\cdots &\cdots  &\cdots \\[3pt]
 \sqrt{x_1}m^+_{n1}   & \cdots &\sqrt{x_{ \mathfrak k}} m^+_{n \mathfrak k}&m^+_{n( \mathfrak k+1)} &\cdots& m^+_{nn}
 \end{pmatrix}
\end{equation}
For convenience, we write $\mathbb M_{ m^+}=(\widetilde{m}_{ij}^+)_{n\times n}$ in the following of this subsection.
Similarly, we also write $\mathbb M_{I^+}=(\widetilde{I}_{ij}^+)_{n\times n}$, $\mathbb M_{m}=(\widetilde{m}_{ij})_{n\times n}$.
Note that
\begin{align*}
&x_1 \cdots x_{ \mathfrak k}m^+_{11} \cdots m^+_{ \mathfrak k \mathfrak k}-x_1 \cdots x_{ \mathfrak k}m_{11} \cdots m_{ \mathfrak k \mathfrak k}\\
=&-A\frac{r^{1-\epsilon}}{r^{l}}\sum_{1\leq i_1< \cdots <i_{l}\leq  \mathfrak k}x_{i_1} \cdots x_{i_{l}}+o\big(  A\frac{r^{1-\epsilon}}{r^{l}}     \sum_{1\leq i_1< \cdots <i_{l}\leq  \mathfrak k}x_{i_1} \cdots x_{i_{l}}\big)\\
\leq &-(A-\delta(r))\frac{r^{1-\epsilon}}{r^{l}}\sum_{1\leq i_1< \cdots <i_{l}\leq  \mathfrak k}x_{i_1} \cdots x_{i_{l}},
\end{align*}
and
\begin{align*}
&\det(m^+_{pq})_{(n- \mathfrak k)\times(n- \mathfrak k)}-\det(m_{pq})_{(n- \mathfrak k)\times(n- \mathfrak k)}\\
=&\det(m^+_{pq})_{(n- \mathfrak k)\times(n- \mathfrak k)}-\det(F_{pq})_{(n- \mathfrak k)\times(n- \mathfrak k)}\\
=&\sum_{p= \mathfrak k+1}^{n}A_{pp}2C_{l}\sum_{1\leq i_1< \cdots <i_{l}\leq  \mathfrak k}x_{i_1} \cdots x_{i_{l}}+o\big(\sum_{1\leq i_1< \cdots <i_{l}\leq \mathfrak k}x_{i_1} \cdots x_{i_{l}}\big)\\
\leq & \delta(r) \frac{r^{1-\epsilon}}{r^{l}}\sum_{1\leq i_1< \cdots <i_{l}\leq  \mathfrak k}x_{i_1} \cdots x_{i_{l}},
\end{align*}
where $A_{pp}$ is the cofactor of $F_{pp}$ for the matrix $(F_{pq})_{(n-\mathfrak k)\times(n-\mathfrak k)}$, and $0<\delta(r)\rightarrow0$ as $r\rightarrow0$.
By taking $A$ large, we get
\begin{align}\label{ineq-k-1}\begin{split}
&\widetilde{m}^+_{11} \cdots \widetilde{m}^+_{ \mathfrak k \mathfrak k}\det(\widetilde{m}^+_{pq})_{(n- \mathfrak k)\times(n- \mathfrak k)}-\widetilde{m}_{11} \cdots \widetilde{m}_{ \mathfrak k \mathfrak k}\det(\widetilde{m}_{pq})_{(n- \mathfrak k)\times(n- \mathfrak k)}\\
=&x_1 \cdots x_{ \mathfrak k}m^+_{11} \cdots m^+_{ \mathfrak k \mathfrak k}\det(m^+_{pq})_{(n- \mathfrak k)\times(n- \mathfrak k)}-x_1 \cdots x_{ \mathfrak k}m_{11} \cdots m_{ \mathfrak k \mathfrak k}\det(m_{pq})_{(n- \mathfrak k)\times(n- \mathfrak k)}\\
\leq &-\frac{A}{2}\det(F_{pq})_{(n- \mathfrak k)\times(n- \mathfrak k)}\frac{r^{1-\epsilon}}{r^{l}}\sum_{1\leq i_1< \cdots <i_{l}\leq  \mathfrak k}x_{i_1} \cdots x_{i_{l}},
\end{split}\end{align}
when $r$ is small.

Now we consider the terms in the expansion of $\det(\widetilde{m}_{ij}^+)_{n\times n}=x_1 \cdots x_{ \mathfrak k}\det(m^+_{ij})_{n\times n}$ that are  neither contained in the expansion of
$\widetilde{m}^+_{11} \cdots \widetilde{m}^+_{ \mathfrak k \mathfrak k}\det(m^+_{pq})_{(n- \mathfrak k)\times(n- \mathfrak k)}$ nor in the expansion of $x_1 \cdots x_{ \mathfrak k}\det(m_{ij})_{n\times n}$.
Then, each of these product term is a product of $\widetilde{I}^+_{ij}$ and $\widetilde{m}_{ij}$ and must contain $\widetilde{I}^+_{ij}$. We will prove these product terms are relatively small compared to $A\frac{r^{1-\epsilon}}{r^{l}}\sum\limits_{1\leq i_1< \cdots <i_{l}\leq  \mathfrak k}x_{i_1} \cdots x_{i_{l}}$ when $r$ is small.

We discuss four cases:

Case 1.1: $(i,j)=(a,a)$.

Without loss of generality, we assume $(i,j)=(1,1)$. This product term contains $$\widetilde{I}^+_{11}=x_1I^+_{11}=-A\frac{r^{1-\epsilon}}{r^{l}}\sum_{1\leq i_1< \cdots <i_{l}\leq  \mathfrak k, i_j\neq a}x_{i_1} \cdots x_{i_{l}}.$$Then, in this product term, there must be one of the following terms:
$$\widetilde{m}_{ab}, \widetilde{I}_{ab}^+$$
where $a\ne b$. The above terms tend to 0 as $r$ tends to 0.

Therefore, this product term is
$$o\big(A\frac{r^{1-\epsilon}}{r^{l}}\sum_{1\leq i_1< \cdots <i_{l}\leq  \mathfrak k}x_{i_1} \cdots x_{i_{l}}\big),\quad  \text{as}\quad r\rightarrow 0.$$
\smallskip

Case 1.2: $(i,j)=(a,b)$, $a\neq b$.
Without loss of generality, we assume $(a,b)=(1,2)$.

 Recall that
\begin{align*}\widetilde{ I}^{+}_{12}=\sqrt{x_1x_2}I^{+}_{12}=& \sqrt{x_1x_2}C_{l}\sum_{p= \mathfrak k+1}^n x_{p}^2\sum_{1\leq i_1< \cdots <i_{l-2}\leq \mathfrak k,i_j \neq 1, 2}x_{i_1} \cdots x_{i_{l-2}}\\
+&\sqrt{x_1x_2}A\frac{r^{1-\epsilon}}{r^{l}}\sum_{1\leq i_1< \cdots <i_{l-1}\leq  \mathfrak k, i_j\neq 1, 2}x_{i_1} \cdots x_{i_{l-1}}(-\sum_{j=1}^{l-1}\ln x_{i_{j}})\\
+&\sqrt{x_1x_2}A\frac{r^{1-\epsilon}}{r^{l}}\sum_{1\leq i_1< \cdots <i_{l-1}\leq \mathfrak  k, i_j \neq 1, 2}x_{i_1} \cdots x_{i_{l-1}}(-\ln x_1-\ln x_2 -2).
\end{align*}
We first deal with
$$\sqrt{x_1x_2}A\frac{r^{1-\epsilon}}{r^{l}}\sum\limits_{1\leq i_1< \cdots <i_{l-1}\leq  \mathfrak k, i_j\neq 1, 2}x_{i_1} \cdots x_{i_{l-1}}(-\sum\limits_{j=1}^{l-1}\ln x_{i_{j}}).$$
Note that in this case, there must be one of the following terms in the product term:
$$\widetilde{m}_{21}, \widetilde{I}_{21}^+,$$
or a product involving one term from $\widetilde{m}_{2i}, \widetilde{I}_{2i}^+$ and one term from $\widetilde{m}_{j1}, \widetilde{I}_{j1}^+$.
We denote this term by $Q$. Then,
$$Q=o\big(A^2r^{-3\epsilon}\sqrt{x_1x_2}(-\ln x_1-\ln x_2 )^2\big),\quad \text{as}\quad r\rightarrow 0.$$
Then, \begin{align*} & |Q\sqrt{x_1x_2}A\frac{r^{1-\epsilon}}{r^{l}}\sum_{1\leq i_1< \cdots <i_{l-1}\leq  \mathfrak k, i_j\neq 1, 2}x_{i_1} \cdots x_{i_{l-1}}(-\sum_{j=1}^{l-1}\ln x_{i_{j}})|\\
\leq & | A^2r^{-3\epsilon}\sqrt{x_1x_2}(-\ln x_1-\ln x_2 )^2\times \sqrt{x_1x_2}A\frac{r^{1-\epsilon}}{r^{l}}\sum_{1\leq i_1< \cdots <i_{l-1}\leq  \mathfrak k, i_j\neq 1, 2}x_{i_1} \cdots x_{i_{l-1}}(-\sum_{j=1}^{l-1}\ln x_{i_{j}}) |\\
\ll& |A\frac{r^{1-\epsilon}}{r^{l}}\sum_{1\leq i_1< \cdots <i_{l-1}\leq \mathfrak  k, i_j\neq 1, 2}x_{i_1} \cdots  \widehat{x_{i_j}}  \cdots x_{i_{l-1}} x_1^{\frac{2}{3}}x_2^{\frac{2}{3}}x_{i_j}^{\frac{2}{3}}|\\
\leq& |A\frac{r^{1-\epsilon}}{r^{l}}\sum_{1\leq i_1< \cdots <i_{l-1}\leq  \mathfrak k, i_j\neq 1, 2}x_{i_1} \cdots \widehat{x_{i_j}} \cdots x_{i_{l-1}} \frac{x_1x_2+x_1x_{i_j}+x_2x_{i_j}}{3}|\\
\leq & A\frac{r^{1-\epsilon}}{r^{l}}\sum_{1\leq i_1< \cdots <i_{l}\leq  \mathfrak k}x_{i_1} \cdots x_{i_{l}},
\end{align*}
when $r$ is sufficiently small, where
\begin{equation*}
x_{i_1} \cdots \widehat{x_{i_j}} \cdots x_{i_{l-1}}=
\begin{cases}
x_{i_2} \cdots x_{i_{l-1}},
& \text{if }  i_j=i_1,\\
x_{i_1} \cdots x_{i_{j-1}}x_{i_{j+1}} \cdots x_{i_{l-1}},
& \text{if } i_1<i_j<i_{l-1},\\
x_{i_1} \cdots x_{i_{l-2}},
& \text{if } i_j=i_{l-1}.\\
\end{cases}
\end{equation*}
The desired result follows since the remainder in the product term is bounded.

 Similarly, we can deal with
 $$\sqrt{x_1x_2}A\frac{r^{1-\epsilon}}{r^{l}}\sum\limits_{1\leq i_1< \cdots <i_{l-1}\leq  \mathfrak k, i_j \neq 1, 2}x_{i_1} \cdots x_{i_{l-1}}(-\ln x_1-\ln x_2 -2).$$

  It remains to deal with
  $$\sqrt{x_1x_2}C_{l}\displaystyle\sum_{p= \mathfrak k+1}^nx_{p}^2\sum_{1\leq i_1< \cdots <i_{l-2}\leq  \mathfrak k,i_j \neq 1, 2}x_{i_1} \cdots x_{i_{l-2}}.$$
   In this case, we don't need to consider the log term in $Q$; otherwise, we return to  the known situations in Case 1.2. Hence

\begin{equation*} |Q \sqrt{x_1x_2}C_{l}\sum_{p= \mathfrak k+1}^nx_{p}^2\sum_{1\leq i_1< \cdots <i_{l-2}\leq  \mathfrak k,i_j\ne 1,2}x_{i_1} \cdots x_{i_{l-2}} |
\ll A\frac{r^{1-\epsilon}}{r^{l}}\sum_{1\leq i_1< \cdots <i_{l}\leq  \mathfrak k}x_{i_1} \cdots x_{i_{l}},
\end{equation*}
when $r$ is small enough.
This completes the proof for Case 1.2.

\smallskip
Case 1.3: $(i,j)=(p,q)$.

Since the remainder in the product term is bounded,  the desired result follows when $r$ is small.

\smallskip
Case 1.4: $(i,j)=(a,p)$.

In this case,  the product term must contain one of the following terms:
\begin{equation*}\widetilde{m}_{i a}, \widetilde{I}_{ia}^+, \quad \quad i\neq a.\end{equation*}

We denote this term by $Q$. If $Q\neq\widetilde{I}_{b a}^+$ for any $b\in\{1, \cdots , \mathfrak k\}$, then
\begin{align*}|Q\widetilde{I}_{ap} |\leq & C \sqrt{x_a}C_{l}\sqrt{x_a}\sum_{1\leq i_1< \cdots <i_{l-1}\leq  \mathfrak k, i_j\neq a}x_{i_1} \cdots x_{i_{l-1}}\\
=& o\big(A\frac{r^{1-\epsilon}}{r^{l}}\sum_{1\leq i_1< \cdots <i_{l}\leq  \mathfrak k}x_{i_1} \cdots x_{i_{l}}\big).
\end{align*}
Then, the desired results follow since the remainder in the product term is bounded.

If $Q=\widetilde{I}_{b a}^+$ for some $b\in\{1, \cdots , \mathfrak k\}$, then this is the case handled in Case 1.2.

In summary, when $r$ is small, we have
\begin{align*}
&\det(\widetilde{m}^+_{ij})_{n\times n}- \det(\widetilde{m}_{ij})_{n\times n}\\
\leq &\widetilde{m}^+_{11} \cdots \widetilde{m}^+_{ \mathfrak k \mathfrak k}\det(\widetilde{m}^+_{pq})_{(n- \mathfrak k)\times(n- \mathfrak k)}-x_1 \cdots x_{ \mathfrak k}m_{11} \cdots m_{ \mathfrak k \mathfrak k}\det(m_{pq})_{(n- \mathfrak k)\times(n- \mathfrak k)}\\
&+\frac{A}{4}\det(F_{pq})_{(n- \mathfrak k)\times(n- \mathfrak k)}\frac{r^{1-\epsilon}}{r^{l}}\sum_{1\leq i_1< \cdots <i_{l}\leq  \mathfrak k}x_{i_1} \cdots x_{i_{l}}\\
\leq &-\frac{A}{4}\det(F_{pq})_{(n- \mathfrak k)\times(n- \mathfrak k)}\frac{r^{1-\epsilon}}{r^{l}}\sum_{1\leq i_1< \cdots <i_{l}\leq \mathfrak k}x_{i_1} \cdots x_{i_{l}}\\
\leq &-C_*x_1 \cdots x_{ \mathfrak k}\\
\leq & h- \det(\widetilde{m}_{ij})_{n\times n}
\end{align*}
in $[0,r]^{ \mathfrak k }\times[-1,1]^{n- \mathfrak k}$ , where we used \eqref{mp-h} and \eqref{ineq-k-1}.

We also have $I\leq I^+$ on $\partial\big([0,r]^{ \mathfrak k}\times[-1,1]^{n- \mathfrak k}\big)$.

Then by the standard maximum principle, $I\leq I^+$ in $[0,r]^{ \mathfrak  k}\times[-1,1]^{n- \mathfrak k}$. At the origin, this is \eqref{refine-growth4} for $C_{l+1,1}=A\frac{r^{1-\epsilon}}{r^{l}}$.

Set \begin{align*} I^{-}=&-A\frac{r^{1-\epsilon}}{r^{l}}\sum_{1\leq i_1< \cdots <i_{l+1}\leq \mathfrak  k}x_{i_1} \cdots x_{i_{l+1}}(-\sum_{j=1}^{l+1}\ln x_{i_{j}})\\
-& C_{l}\sum_{p= \mathfrak k+1}^nx_p^2\sum_{1\leq i_1< \cdots <i_l\leq  \mathfrak k}x_{i_1} \cdots x_{i_l},\end{align*}
We can derive the estimate $I\geq I^-$ in $[0,r]^{ \mathfrak k}\times[-1,1]^{n- \mathfrak k}$ similarly.

Therefore, we have proved Step 1.

\smallskip

Step 2: Then we derive the following estimate
\begin{equation}\label{refine-growth5}
|I|\leq C_{l+1}\sum_{1\leq i_1< \cdots <i_{l+1}\leq  \mathfrak k}x_{i_1} \cdots x_{i_{l+1}}
\end{equation}
in $[0,r_{l}']^{ \mathfrak k}\times[-5/2+\frac{l}{2 \mathfrak k},5/2-\frac{l}{2 \mathfrak k}]^{n- \mathfrak k}$.
By an affine transform, it suffices to prove \eqref{refine-growth5} at $(x',0'')$.

Set \begin{align*} I'^{+}=&C_{l+1,1}\sum_{p= \mathfrak k+1}^nx_{p}^4\sum_{1\leq i_1< \cdots <i_{l+1}\leq \mathfrak  k}x_{i_1} \cdots x_{i_{l+1}}\left(-\sum_{j=1}^{l+1}\ln x_{i_{j}}\right)\\
+& A\frac{r^{1-\epsilon}}{r^{l}}\sum_{1\leq i_1< \cdots <i_{l+1}\leq  \mathfrak k}x_{i_1} \cdots x_{i_{l+1}}\left(1-B\sum_{j=1}^{l+1} x_{i_{j}}\right),\end{align*}
where $A\geq \max\{2A_{\epsilon}, 1\}$ is some constant to be determined and $A$ is independent of $r$, $B=\frac{1}{100n^{2 \mathfrak k+2}r}$. Note $B$ is large when $r$ is small.

By a direct computation, we have

 \begin{align*} I'^{+}_{aa}=&-2AB\frac{r^{1-\epsilon}}{r^{l}}\sum_{1\leq i_1< \cdots <i_{l}\leq  \mathfrak k, i_j\neq a}x_{i_1} \cdots x_{i_{l}}\\
 &-\frac{1}{x_a}C_{l+1,1}\sum_{p= \mathfrak k+1}^nx_{p}^4\sum_{1\leq i_1< \cdots <i_{l}\leq  \mathfrak k, i_j\neq a}x_{i_1} \cdots x_{i_{l}},
 \end{align*}
\begin{align*} I'^{+}_{ab}=&C_{l+1,1}\sum_{p= \mathfrak k+1}^nx_{p}^4\sum_{1\leq i_1< \cdots <i_{l-1}\leq  \mathfrak k, i_j\neq a, b}x_{i_1} \cdots x_{i_{l-1}}(-\sum_{j=1}^{l-1}\ln x_{i_{j}})\\
+&C_{l+1,1}\sum_{p= \mathfrak k+1}^nx_{p}^4\sum_{1\leq i_1< \cdots <i_{l-1}\leq \mathfrak  k, i_j \neq a, b}x_{i_1} \cdots x_{i_{l-1}}(-\ln x_a-\ln x_b -2)\\
+&A\frac{r^{1-\epsilon}}{r^{l}}\sum_{1\leq i_1< \cdots <i_{l-1}\leq  \mathfrak k, i_j \neq a, b}x_{i_1} \cdots x_{i_{l-1}}(1-B\sum_{j=1}^{l-1} x_{i_{j}})\\
-& AB\frac{r^{1-\epsilon}}{r^{l}}\sum_{1\leq i_1< \cdots <i_{l-1}\leq  \mathfrak k,i_j \neq a, b}x_{i_1} \cdots x_{i_{l-1}}(2x_a+2x_b),\end{align*}
for $a\neq b$.
\begin{equation*}I'^{+}_{pq}=12C_{l+1,1}\delta_{pq}x_{p}^2\sum_{1\leq i_1< \cdots <i_{l+1}\leq  \mathfrak k}x_{i_1} \cdots x_{i_{l+1}}(-\sum_{j=1}^{l+1}\ln x_{i_{j}}). \end{equation*}

\begin{equation*}I'^{+}_{ap}=4C_{l+1,1}x_{p}^3\sum_{1\leq i_1< \cdots <i_{l}\leq  \mathfrak k,i_j\neq a}x_{i_1} \cdots x_{i_{l}}(-\sum_{j=1}^l\ln x_{i_j}-\ln x_a-1). \end{equation*}
Denote
\begin{equation*}
m'^+=m+I'^+,\quad(\widetilde{m}_{ij}'^+)_{n\times n}= \mathbb M_{m'^+},\quad (\widetilde{I}_{ij}'^+)_{n\times n}=\mathbb M_{I'^+}.
\end{equation*}
Arguing similarly as in Step 1, we can derive:
\begin{align}\label{ineq-k-2}\begin{split}
&\widetilde{m}'^+_{11} \cdots \widetilde{m}'^+_{ \mathfrak k \mathfrak k}\det(\widetilde{m}'^+_{pq})_{(n- \mathfrak k)\times(n- \mathfrak k)}-\widetilde{m}_{11} \cdots \widetilde{m}_{ \mathfrak k \mathfrak k}\det(\widetilde{m}_{pq})_{(n- \mathfrak k)\times(n- \mathfrak k)}\\
=&x_1 \cdots x_{ \mathfrak k}m'^+_{11} \cdots m'^+_{ \mathfrak k \mathfrak k}\det(m'^+_{pq})_{(n- \mathfrak k)\times(n- \mathfrak k)}-x_1 \cdots x_{ \mathfrak k}m_{11} \cdots m_{ \mathfrak k \mathfrak k}\det(m_{pq})_{(n- \mathfrak k)\times(n- \mathfrak k)}\\
\leq &-(1-\delta(r))\det(F_{pq})_{(n- \mathfrak k)\times(n- \mathfrak k)}\bigg[2AB\frac{r^{1-\epsilon}}{r^{l}}\sum_{1\leq i_1< \cdots <i_{l+1}\leq \mathfrak k}x_{i_1} \cdots x_{i_{l+1}}\\
 &+C_{l+1,1}\sum_{p= \mathfrak k+1}^nx_{p}^4\sum_{1\leq i_1< \cdots <i_{l}\leq \mathfrak  k}x_{i_1} \cdots x_{i_{l}}\bigg]\\
 &+12(1+\delta(r))C_{l+1,1}\sum_{p= \mathfrak k+1}^{n}A_{pp}x_{p}^2\sum_{1\leq i_1< \cdots <i_{l+1}\leq  \mathfrak k}x_{i_1} \cdots x_{i_{l+1}}(-\sum_{j=1}^{l+1}\ln x_{i_{j}})\\
\leq &-\frac{\det(F_{pq})_{(n- \mathfrak k)\times(n- \mathfrak k)}}{2}\bigg[AB\frac{r^{1-\epsilon}}{r^{l}}\sum_{1\leq i_1< \cdots <i_{l+1}\leq \mathfrak  k }x_{i_1} \cdots x_{i_{l+1}}\\
&+C_{l+1,1}\sum_{p= \mathfrak k+1}^nx_{p}^4\sum_{1\leq i_1< \cdots <i_{l}\leq  \mathfrak k}x_{i_1} \cdots x_{i_{l}}\bigg],
\end{split}\end{align}
when $r$ is small, where $A_{pp}$ is the cofactor of $F_{pp}$ of the matrix $(F_{pq})_{(n- \mathfrak k)\times(n- \mathfrak k)}$, and $\delta(r)\rightarrow0$ as $r\rightarrow0$.
Define $H$ as follows:
\begin{align*}
&-\frac{\det(F_{pq})_{(n- \mathfrak k)\times(n- \mathfrak k)}}{2}\bigg[AB\frac{r^{1-\epsilon}}{r^{l}}\sum_{1\leq i_1< \cdots <i_{l+1}\leq \mathfrak k}x_{i_1} \cdots x_{i_{l+1}}\\
&+C_{l+1,1}\sum_{p= \mathfrak k+1}^nx_{p}^4\sum_{1\leq i_1< \cdots <i_{l}\leq \mathfrak  k, i_j\neq a}x_{i_1} \cdots x_{i_{l}}\bigg].
\end{align*}
 Next, we consider the terms in the expansion of $\det(\widetilde{m}_{ij}'^+)_{n\times n}$ which are neither contained in the expansion of
$\widetilde{m}'^+_{11} \cdots \widetilde{m}'^+_{ \mathfrak k \mathfrak k}\det(m'^+_{pq})_{(n- \mathfrak k)\times(n- \mathfrak k)}$ nor  in the expansion of $x_1 \cdots x_{ \mathfrak k}\det(m_{ij})_{n\times n}$.
Then, each of these product terms is a product of $\widetilde{I}'^+_{ij}$ and $\widetilde{m}'_{ij}$, and must contain $\widetilde{I}'^+_{ij}$. We will prove these product terms are relatively small compared to $-H$ when $r$ is  sufficiently small.

As in Step 1, we discuss four cases:

Case 2.1: $(i,j)=(a,a)$.

Then, this product term is $o(-H)$ which follows from a similar argument as in Case 1.1.

\smallskip

Case 2.2: $(i,j)=(a,p)$.

\smallskip
Case 2.2.1: We first consider the term
$$P=4C_{l+1,1}x_p^3\sqrt{x_a}\sum_{1\le i_1<\cdots<i_l\le \mathfrak k,i_j\ne a}x_{i_1}\cdots x_{i_l}(-\ln x_a-1).$$
Then,
\begin{align*}|P|\leq &8C_{l+1,1}x_{p}^3\sqrt{x_a}|\ln x_a|\sum_{1\leq i_1< \cdots <i_{l}\leq  \mathfrak k,i_j\neq a}x_{i_1} \cdots x_{i_{l}}\\
\le & \delta(r)8C_{l+1,1}x_{p}^3x_a^{\frac{1}{4}}\sum_{1\leq i_1< \cdots <i_{l}\leq  \mathfrak k,i_j\neq a}x_{i_1} \cdots x_{i_{l}}\\
\leq &  \delta(r)C_{l+1,1}[x_{p}^{4}+x_a]\sum_{1\leq i_1< \cdots <i_{l}\leq  \mathfrak k,i_j\neq a}x_{i_1} \cdots x_{i_{l}}\\
= & o(-H).
 \end{align*}
Here, $\delta(r)=\sup_{t\in (0,r]}t^{\frac 14}|\ln t|$. Since the remainder in the product term are bounded, for small $r$, we complete the proof for Case 2.2.1.

\smallskip
Case 2.2.2: It remains to deal with the term
$$\widetilde P=4C_{l+1,1}x_p^3\sqrt{x_a}\sum_{1\le i_1<\cdots<i_l\le \mathfrak k,i_j\ne a}x_{i_1}\cdots x_{i_l}(-\ln x_{i_1}).$$
Denote the remaining term in the product by $\widetilde Q$.
Then the product $\widetilde Q$ must have $\widetilde I'^+_{ai}$ or $\widetilde m_{ai}$ as a factor.
\par If $\widetilde m_{ai}$ is in the product term, then  $|\widetilde Q|=O(\sqrt x_a)$ and
$$|\widetilde Q\cdot \widetilde P|=o( {\color{blue}-} H)$$
can be derived by the same argument as Case 2.2.1.
\par If $\widetilde I'^+_{aq}$ for some $q=\mathfrak k+1,\cdots,n$ is in the product term, then $\widetilde Q=o(x_q^3)$.
Hence,
\begin{equation*}
\begin{split}
|\widetilde Q\cdot \widetilde P|\le &\sum_{i\ne a} o(x_p^3x_q^3\sqrt{x_a x_i}\sum_{1\le i_1<\cdots i_{l-1}\le \mathfrak k,i_j\ne i,a}x_{i_1}\cdots x_{i_{l-1}})
\\ \le &o(\sum_{r\ge \mathfrak k+1} x_r^4\sum_{1\le i_1<\cdots i_{l}\le \mathfrak k,i_j\ne i}x_{i_1}\cdots x_{i_{l}})=o(-H).
\end{split}
\end{equation*}
\par  If $\widetilde I'^+_{ab}$ for some $b=1,\cdots,\mathfrak k$ is in the product term. By Case 2.4.1, we can just assume there is no log term in the product $\widetilde Q$. In this case, $|\widetilde Q|=o(\sqrt{x_a})$. Then
$$|\widetilde Q\cdot \widetilde P|=o(-H)$$
can be derived by the same argument as Case 2.2.1.
\smallskip

Case 2.3: $(i,j)=(p,q)$.

\begin{align*}|\widetilde{I}'^{+}_{pq}|\leq & 12C_{l+1,1}x_{p}^2\sum_{1\leq i_1< \cdots <i_{l+1}\leq  \mathfrak k  }x_{i_1} \cdots x_{i_{l+1}}(-\sum_{j=1}^{l+1}\ln x_{i_{j}})\\
= & \delta(r)8C_{l+1,1}x_{p}^2x_{i_j}^{\frac{1}{2}}\sum_{1\leq i_1< \cdots <i_{l+1}\leq  \mathfrak k}x_{i_1} \cdots \widehat{x_{i_j}} \cdots x_{i_{l+1}}\\
= & o(-H).
 \end{align*}
Hence, this product term is $o(-H)$ since the remainder in the product term is bounded.
\smallskip

Case 2.4: $(i,j)=(a,b)$, $a\neq b$.

\smallskip
Case 2.4.1:
If the product contains
\begin{align}\label{I-abpart3}\begin{split}
 &\sqrt{x_ax_b}C_{l+1,1}\sum_{p= \mathfrak k+1}^nx_{p}^4\sum_{1\leq i_1< \cdots <i_{l-1}\leq  \mathfrak k, i_j\neq a, b}x_{i_1} \cdots x_{i_{l-1}}(-\sum_{j=1}^{l-1}\ln x_{i_{j}})\\
+&\sqrt{x_ax_b}C_{l+1,1}\sum_{p= \mathfrak k+1}^nx_{p}^4\sum_{1\leq i_1< \cdots <i_{l-1}\leq  \mathfrak k, i_j \neq a, b}x_{i_1} \cdots x_{i_{l-1}}(-\ln x_a-\ln x_b -2), \end{split}\end{align}
then this product is $o(-H)$ by a similar argument as we deal with the log term in Step 1 Case 1.2.

\smallskip
Case 2.4.2:
We now handle the term
\begin{align}\label{I-abpart2}
&\sqrt{x_ax_b}A\frac{r^{1-\epsilon}}{r^{l}}\sum_{1\leq i_1< \cdots <i_{l-1}\leq  \mathfrak k, i_j \neq a, b}x_{i_1} \cdots x_{i_{l-1}}(1-B\sum_{j=1}^{l-1} x_{i_{j}})\\
+&\sqrt{x_ax_b} AB\frac{r^{1-\epsilon}}{r^{l}}\sum_{1\leq i_1< \cdots <i_{l-1}\leq  \mathfrak k,i_j \neq a, b}x_{i_1} \cdots x_{i_{l-1}}(2x_a+2x_b).\label{I-abpart1}
\end{align}
By the choice of $B=\frac{1}{100 n^{2 \mathfrak k+2}r}$, one knows $\eqref{I-abpart2}\gg |\eqref{I-abpart1}|$. Hence it is enough to consider \eqref{I-abpart2}.

As discussed in Step 1 Case 1.2, there must be one of the following terms in the product term:
$$\widetilde{m}_{ba}, \widetilde{I}'^+_{ba},$$
or a product involving a term from $\widetilde{m}_{b\alpha}, \widetilde{I}'^+_{b\alpha}$ and a term from $\widetilde{m}_{\beta a}, \widetilde{I}'^+_{\beta a}$.
We denote the above term by $Q$.

If there is no log term in $Q$, then $|Q|=O(\sqrt{x_ax_b}Ar^{-\epsilon})$. Hence, \eqref{I-abpart2} multiplied  by $Q$ is small compared to $$ |Q\cdot \eqref{I-abpart2}|\ll AB\frac{r^{1-\epsilon}}{r^{l}}\sum_{1\leq _1< \cdots <i_{l+1}\leq  \mathfrak k}x_{i_1} \cdots x_{i_{l+1}}.$$

If there is a log term in $Q$, then  it returns to the discussion about  \eqref{I-abpart3}, Case 2.2 or Case 2.3.

This ends the proof for Case 2.4.

In summary, when $r$ is small, we have
\begin{align*}
&\det(\widetilde{m}'^+_{ij})_{n\times n}-\det(\widetilde{m}_{ij})_{n\times n}\\
\leq &\widetilde{m}'^+_{11} \cdots \widetilde{m}'^+_{ \mathfrak k \mathfrak k}\det(\widetilde{m}'^+_{pq})_{(n- \mathfrak k)\times(n- \mathfrak k)}-\widetilde{m}_{11} \cdots \widetilde{m}_{ \mathfrak k \mathfrak k}\det(\widetilde{m}_{pq})_{(n- \mathfrak k)\times(n- \mathfrak k)}\\&+\frac{\det(F_{pq})_{(n- \mathfrak k)\times(n- \mathfrak k)}}{4}\bigg[AB\frac{r^{1-\epsilon}}{r^{l}}\sum_{1\leq l_1< \cdots <i_{l+1}\leq \mathfrak k}x_{i_1} \cdots x_{i_{l+1}}\\&+C_{l+1,1}\sum_{p= \mathfrak k+1}^nx_{p}^4\sum_{1\leq i_1< \cdots <i_{l}\leq  \mathfrak k, i_j\neq a}x_{i_1} \cdots x_{i_{l}}\bigg]\\
\leq &-C_*x_1 \cdots x_{ \mathfrak k}\\
\leq & h-\det(\widetilde{m}_{ij})_{n\times n}
\end{align*}
in $[0,r]^{ \mathfrak k}\times[-1,1]^{n- \mathfrak k}$, where we used \eqref{mp-h} and \eqref{ineq-k-2}.
We also have 
$$I\leq I'^+,\quad \text{on} \quad\partial\big([0,r]^{ \mathfrak k}\times[-1,1]^{n- \mathfrak k}\big).$$
Then, by the standard maximum principle, 
$$I\leq I'^+,\quad \text{in} \quad [0,r]^{ \mathfrak k}\times[-1,1]^{n- \mathfrak k}.$$
At $(x',0'')$, this is \eqref{refine-growth5} for $C_{l+1}=A\frac{r^{1-\epsilon}}{r^{l}}$.

Set \begin{align*} I'^{-}=&-C_{l+1,1}\sum_{p= \mathfrak k+1}^nx_{p}^4\sum_{1\leq i_1< \cdots <i_{l+1}\leq  \mathfrak k}x_{i_1} \cdots x_{i_{l+1}}(-\sum_{j=1}^{l+1}\ln x_{i_{j}})\\
& -A\frac{r^{1-\epsilon}}{r^{l}}\sum_{1\leq i_1< \cdots <i_{l+1}\leq  \mathfrak k}x_{i_1} \cdots x_{i_{l+1}}(1-B\sum_{j=1}^{l+1} x_{i_{j}}).
\end{align*}
We can derive the estimate $I\geq I'^-$ in $[0,r]^{ \mathfrak k}\times[-1,1]^{n- \mathfrak k}$ similarly.

This implies the desired result for $l+1$ and proves Lemma \ref{lemma-growth2}.
\end{proof}

\smallskip

Now let us return to the discussion about the regularity of $I$.

\begin{lemma}\label{lemma-I12k-1}
There  holds that
\begin{equation*}
\|D^\beta_{x'}D^\gamma_{x''}I\|_{C^{\frac 14}([0,\frac{\mathfrak k^*}{4}]^{ \mathfrak k}\times[-3/2,3/2]^{n- \mathfrak k})}\lesssim 1,\quad |\beta|+\frac{|\gamma|}{2}\le \mathfrak  k-\frac{1}{2}.
\end{equation*}
\end{lemma}

\begin{proof}

By Lemma \ref{lemma-growth2}, the following estimate holds:
 \begin{equation*}
 |I|\leq C x_{1} \cdots x_{ \mathfrak k}, \quad \text{in} \quad [0,  r_{\mathfrak k}^*]^{ \mathfrak k}\times[-2,2]^{n- \mathfrak k}.
 \end{equation*}
Take $\lambda\in (0,\frac{r_{ \mathfrak k}^*}{4}]$.
For any $x''\in [-3/2,3/2]^{n- \mathfrak k}$,
define
\begin{equation*}
u_\lambda(y) \triangleq \frac{u(\lambda y' ,\bar x''+\sqrt{\lambda}y'' )-\sum\limits_{i=1}^{ \mathfrak k}(\lambda \ln  \lambda)y_i}{\lambda},
\end{equation*}
\begin{equation*}
F_{\lambda}(y)\triangleq \frac{F(\lambda y' ,\bar x''+\sqrt{\lambda}y'' )}{\lambda},
\end{equation*} and
\begin{equation*}
I_{\lambda}(y)\triangleq \frac{I(\lambda y' ,\bar x''+\sqrt{\lambda}y'' )}{\lambda}.
\end{equation*}
After subtracting an affine function, we may assume
\begin{equation*}
F\big(0',\bar x''\big)=|D F\big(0',\bar x''\big)|=0.
\end{equation*}
Then,
$$|I_{\lambda}|\lesssim  \lambda^{ \mathfrak k-1}y_1y_2 \cdots y_{ \mathfrak k},\quad \text{in}\quad [0,2]^{ \mathfrak k}\times[-2,2]^{n- \mathfrak k}.$$
Similar argument as in Lemma \ref{lemma-growth1} yields
$$|D^{\beta}_{y'} D^{\gamma}_{y''}I_{\lambda}|\lesssim_{|\beta|+|\gamma|} 1,\quad \text{in}\quad [0,\frac{3}{2}]^{ \mathfrak k-1}\times[\frac{3}{4},\frac{3}{2}]\times[-1,1]^{n- \mathfrak k}.$$
Repeating the above argument, in fact, one gets
\begin{equation*}\label{estimate-k}
|D^{\beta}_{y'} D^{\gamma}_{y''}I_{\lambda}|\lesssim_{|\beta|+|\gamma|}  1,\quad \text{in}\quad ([0,\frac{3}{2}]^{ \mathfrak k}\setminus [0,\frac{3}{4}]^{ \mathfrak k})\times[-1,1]^{n- \mathfrak k}.
\end{equation*}
Then, by Lemma \ref{lemma-interpolation},  there holds
$$|D^{\beta}_{y'} D^{\gamma}_{y''} I_{\lambda}|\lesssim_{|\beta|+|\gamma|}  \lambda^{ \mathfrak k-\frac 54},\quad \text{in}\quad ([0,\frac{3}{2}]^{ \mathfrak k}\setminus [0,\frac{3}{4}]^{ \mathfrak k})\times[-1,1]^{n- \mathfrak k}.$$
In the original coordinates, this implies
\begin{equation}\label{estimate-k-growth}
|D^{\beta}_{x'} D^{\gamma}_{x''} I|\lesssim_{|\beta|+|\gamma|}   |x'|^{ \mathfrak k-\frac 14-|\beta|-\frac{|\gamma|}{2}}.
\end{equation}

Then, the desired result follows from Lemma \ref{lemma-Growth-Regularity}.

\end{proof}

\begin{remark}\label{I12k-regul}

$\forall |\beta|+\frac{|\gamma|}{2}<  \mathfrak k$, the asymptotic estimate {\color{blue}\eqref{estimate-k-growth}} implies that
 $$\sqrt{x_a}D_{x_p}D^{\beta}_{x'}D^{\gamma}_{x''}I,\sqrt{x_a}x_bD_{x_p}D_xD^{\beta}_{x'}D^{\gamma}_{x''}I \in C^{\frac{1}{4}}([0,\frac{r_{ \mathfrak k}^*}{4}]^{ \mathfrak k }\times[-2,2]^{n- \mathfrak k}),$$
  and $$x_{a}D_xD^{\beta}_{x'}D^{\gamma}_{x''}I, x_aD_{x_p}D_{x_q}D^{\beta}_{x'}D^{\gamma}_{x''}I, x_{a}x_bD^2_xD^{\beta}_{x'}D^{\gamma}_{x''}I \in C^{\frac{1}{4}}([0,\frac{r_{ \mathfrak k}^*}{4}]^{ \mathfrak k}\times[-2,2]^{n- \mathfrak k}),$$
   where $a,b=1, \cdots , \mathfrak k$, $p,q= \mathfrak k+1, \cdots ,n$.
\end{remark}

\begin{lemma}\label{lemma-I12k}
There exists $r_{ \mathfrak k*}\approx 1$ such that
\begin{equation*}\label{eq-I12k}
\|I_{12 \cdots  \mathfrak k}\|_{C^{1/2}([0,2r_{ \mathfrak k_*}]^{ \mathfrak k}\times[-5/4,5/4]^{n- \mathfrak k})}\lesssim 1.
\end{equation*}
\end{lemma}

\begin{proof}
The proof proceeds similarly to that in Lemma \ref{lemma-I12}. We will prove
$$\|v_{12 \cdots  \mathfrak k}\|_{C^{1/2}([0,2r_{ \mathfrak k*}]^{ \mathfrak k}\times[-3/2,3/2]^{n- \mathfrak k})}\lesssim 1 $$ for some constant $r_{ \mathfrak k*}\in(0,\frac{r_{ \mathfrak k}^*}{8})$.

Set $r_{**}=\frac{r_{ \mathfrak k}^*}{4}$.
For any $x\in([0,r_{**}]^{ \mathfrak k}\setminus\{0\}^{ \mathfrak k})\times[-3/2,3/2]^{n- \mathfrak k}$,
define the matrix $U_{ij}$ as follows:
\begin{equation*}\label{matrix-I12k}
  \begin{pmatrix}
      1+x_1v_{11}&    \cdots &\sqrt{x_1x_{ \mathfrak k}}v_{1 \mathfrak k} &\sqrt{x_1}v_{1( \mathfrak k+1)}&\cdots & \sqrt{x_1}v_{1n}\\[3pt]
      \cdots &\cdots &\cdots &\cdots &\cdots  &\cdots \\[3pt]
   \sqrt{x_1x_{ \mathfrak k}}v_{ \mathfrak k1}& \cdots&  1+x_{ \mathfrak k} v_{ \mathfrak k \mathfrak k} &\sqrt{x_{ \mathfrak k}} v_{ \mathfrak k( \mathfrak k+1)}  &\cdots & \sqrt{x_{ \mathfrak k}} v_{( \mathfrak k+1)n}\\[3pt]
  \sqrt{x_1}v_{( \mathfrak k+1)1} & \cdots& \sqrt{x_{ \mathfrak k}} v_{( \mathfrak k+1) \mathfrak k}& v_{( \mathfrak k+1)( \mathfrak k+1)} &\cdots & v_{( \mathfrak k+1)n}\\[3pt]
    \cdots &\cdots &\cdots &\cdots &\cdots  &\cdots \\[3pt]
 \sqrt{x_1}v_{n1}   & \cdots &\sqrt{x_{ \mathfrak k}} v_{n \mathfrak k}& v_{n( \mathfrak k+1)} &\cdots& v_{nn}.
 \end{pmatrix}
\end{equation*}
Then, $U_{ij}$ solves
\begin{equation*}\label{MAI12k}
\det U_{ij}=h(x).
\end{equation*}
By Lemma \ref{lemma-I12k-1} and Remark \ref{I12k-regul}, there holds
\begin{equation*}
\|U_{ij}\|_{C^{\frac 14}([0,r_{ **}]^{ \mathfrak k}\times[-3/2,3/2]^{n- \mathfrak k})}\lesssim 1,\quad i,j=1,\cdots,n.
\end{equation*}
Similar to $v_{12}$ in  Lemma \ref{lemma-I12}, $V=v_{12 \cdots  \mathfrak k}$ solves an equation of the form:
\begin{equation*}\label{eqV12k-1}
A\sum_{a=1}^{ \mathfrak k}(x_aV_{aa}+V_{a})+\sum_{a=1}^{ \mathfrak k}\sum_{p= \mathfrak k+1}^{n}x_a\frac{a_{ap}}{\sqrt{x_a}}V_{ap}+\sum_{p,q= \mathfrak k+1}^{n}a_{pq}V_{pq}+\sum_{p= \mathfrak k+1}^{n}B_pV_{p}=h_{12 \cdots  \mathfrak k}+g:=G.
\end{equation*}
where
\begin{equation}\label{MI12k}
\begin{split}
A=&\det \begin{pmatrix}
v_{( \mathfrak k+1)( \mathfrak k+1)} & \cdots & v_{( \mathfrak k+1)n}\\[3pt]
  \cdots &\cdots &\cdots  \\[3pt]
 v_{n( \mathfrak k+1)} & \cdots& v_{nn}
 \end{pmatrix}\\
 =&\det \begin{pmatrix}
F_{( \mathfrak k+1)( \mathfrak k+1)}+I_{( \mathfrak k+1)( \mathfrak k+1)} & \cdots & F_{( \mathfrak k+1)n}+I_{( \mathfrak k+1)n}\\[3pt]
  \cdots &\cdots &\cdots  \\[3pt]
 F_{n( \mathfrak k+1)}+I_{n( \mathfrak k+1)} & \cdots& F_{nn}+I_{nn},
 \end{pmatrix}
 \end{split}
\end{equation}
and
 $$a_{ap}=(U^*)^{ap}=\sqrt {x_a} b_{ap},\quad a_{pq}=(U^*)^{pq},\quad a=1, \cdots , \mathfrak k,\quad p,q= \mathfrak k+1, \cdots ,n.$$
For fixed $x$, $B_p$ is a combination of polynomials in $v_{pq}$, $x_bv_{bb}+1$, and$v_{ap}$. The function $g$ is a combination of polynomials in $x_ax_bD^{\alpha}v_{ab}$, $x_aD^{\alpha}_xv_{b}$, and $D^{\alpha}_xv$, where $a,b\in\{1,2, \cdots , \mathfrak k\}$, $p,q\in\{ \mathfrak k+1,..,n\}$,  and $\alpha$ is a multi-index with $|\alpha|\leq  \mathfrak k$.
Also, we have
\begin{equation*}
\begin{split}
&\|a_{pq}\|_{C^{\frac 14}([0,r_{ **}]^{ \mathfrak k}\times[-3/2,3/2]^{n- \mathfrak k})}+\|\frac{a_{ap}}{\sqrt{x_a}}\|_{C^{\frac 14}([0,r_{ **}]^{ \mathfrak k}\times[-3/2,3/2]^{n- \mathfrak k})}\\
 +& \|A\|_{C^{\frac 14}([0,r_{**}]^{ \mathfrak k}\times[-3/2,3/2]^{n- \mathfrak k})}+\|B_p\|_{C^{\frac 14}([0,r_{**}]^{ \mathfrak k}\times[-3/2,3/2]^{n- \mathfrak k})}\lesssim 1,
\end{split}
\end{equation*}
and  $|g|\lesssim  |x'|^{-\frac{1}{16(n+ \mathfrak k)}} $.

Let $$x_a=\frac{y_{2a-1}^2+y_{2a}^2}{4},\quad a=1, \cdots , \mathfrak k,\quad y_{p}=x_{p- \mathfrak k}, \quad p=2 \mathfrak k+1, \cdots ,n+ \mathfrak k.$$
Write
$$y=(y',y''),\quad y'=(y_1, \cdots ,y_{2 \mathfrak k}),\quad y''=(y_{2 \mathfrak k+1}, \cdots ,y_{n+ \mathfrak k}).$$
Similarly to the proof of Lemma \ref{lemma-I12}, $\overline{V}(y)=v_{12 \cdots  \mathfrak k}(x)$ solves
\begin{align*}\label{eq-v12k-y}\begin{split}
L\overline{V}:=&\sum_{a=1}^{2 \mathfrak k}\partial_{y_ay_a}\overline{V}+\sum_{a=1}^{ \mathfrak k}\sum_{p=2 \mathfrak k+1}^{n+ \mathfrak k}\frac{\overline{b}_{ap}}{\overline{A}}\bigg(\frac{y_{2a-1}}{2}\overline{V}_{(2a-1)p}+\frac{y_{2a}}{2}\overline{V}_{(2a)p}\bigg)\\
&+\sum_{p,q=2 \mathfrak k+1}^{n+ \mathfrak k}\frac{\overline{a}_{pq}}{\overline{A}}\overline{V}_{pq}+\sum_{p=2 \mathfrak k+1}^{n+ \mathfrak k}\frac{\overline{B}_p}{\overline{A}}\overline{V}_{p}\\
&=\overline{G},
\end{split}\end{align*}
in $\bigg(\big(\underbrace{B_{2\sqrt{r_{**}}}^{2}(0)\times  \cdots \times B_{2\sqrt{r_{**}}}^{2}(0)}_{ \mathfrak k}\big)\setminus\{0\}^{2 \mathfrak k}\bigg)\times [-3/2,3/2]^{n- \mathfrak k}$
with $|\overline G| \lesssim  |y'|^{-\frac{1}{8(n+ \mathfrak k)}}$. Then,
there exists some constant $r_*' \leq2\sqrt{r_{**}}$, the above equation is a uniformly elliptic equation in
$(B_{r_*'}^{2 \mathfrak k}(0)\setminus\{0\}^{2 \mathfrak k})\times [-3/2, 3/2]^{n- \mathfrak k}\subseteq \bigg(\big(\underbrace{B_{2\sqrt{r_{**}}}^{2}(0)\times  \cdots \times B_{2\sqrt{r_{**}}}^{2}(0)}_{ \mathfrak k}\big)\setminus\{0\}^{2 \mathfrak k}\bigg)\times [-3/2,3/2]^{n- \mathfrak k}$, and
$\|\frac{\overline{a}_{pq}}{\overline{A}}\|_{C^{\frac 14}},$ $\|\frac{y_a}{2}\overline{b}_{\left[\frac{a+1}{2}\right]p}\|_{C^\frac{1}{4}}\lesssim 1$.

Let $\eta(y)=\eta(|y''|)$ be the cut-off function such that
\begin{equation*}
\begin{cases}
\eta\equiv 1,\quad |y''|\le \frac 34r_*'\\
\eta\equiv 0,\quad |y''|\ge r_*'
\end{cases},\quad \mathfrak k<n;\quad \eta\equiv 1,\quad \mathfrak k=n.
\end{equation*}
Then,
 $w=\eta\overline{v}$ solves
\begin{align*}\label{eq-v12keta}\begin{split}
Lw=&\eta\overline{G}+\overline{V}(\sum_{p,q=2 \mathfrak k+1}^{n+ \mathfrak k}\frac{\overline{a}_{pq}}{\overline{A}}\eta_{pq}+\sum_{p=2 \mathfrak k+1}^{n+ \mathfrak k}\frac{\overline{B}_p}{\overline{A}}\eta_{p})\\
&+\sum_{a=1}^{ \mathfrak k}\sum_{p=2 \mathfrak k+1}^{n+ \mathfrak k}\frac{\overline{b}_{ap}}{\overline{A}}(\frac{y_{2a-1}}{2}\overline{V}_{2a-1}\eta_p+\frac{y_{2a}}{2}\overline{V}_{2a}\eta_p)+2\sum_{p,q=2 \mathfrak k+1}^{n+ \mathfrak k}\frac{\overline{a}_{pq}}{\overline{A}}\overline{V}_{p}\eta_q\\
&:=\widehat{G},
\end{split}\end{align*}
in $B_{r_*'}^{n+ \mathfrak k}(0)\setminus \left(\{0\}^{ \mathfrak k}\times \overline{B_{r_*'}^{n- \mathfrak k}(0)}\right)$.
Then, noting that 
\begin{equation*}
|\overline{V}_{i}|\lesssim |y'|^{-\frac 32},\quad i=1,\cdots, n+ \mathfrak k,\quad D\eta|_{B_{\frac 34r_*'}^{n+ \mathfrak k}(0)}=0,
\end{equation*}
we conclude that
\begin{equation*}
\|\widehat{G}\|_{L^\mathfrak k(B_{r_*'}^{n+ \mathfrak k}(0)))}+\|\widehat{G}\|_{L^{4(n+\mathfrak k)}(B_{\frac 34r_*'}^{n+ \mathfrak k}(0)))}\lesssim 1.
\end{equation*}
Let $W$ be the solution of the following Dirichlet problem.
 \begin{align}
\label{eq-Wk}LW &= \widehat{G} \quad\text{in }B_{r_*'}^{n+ \mathfrak k}(0),\\
\label{eqbdy-Wk}W&=w \quad\text{on }\partial B_{r_*'}^{n+ \mathfrak k}(0).
\end{align}
\eqref{eq-Wk}-\eqref{eqbdy-Wk} admits a unique solution $W\in W^{2, \mathfrak k}(B_{r_*'}^{n+ \mathfrak k}(0))\bigcap W^{2,4(n+\mathfrak k)}_{loc}(B_{\frac 34r_*'}^{n+ \mathfrak k}(0))$. By Sobolev embedding theorem,  we have $$\|W\|_{C^{1,\frac 34}(\overline{B_{\frac{r_*'}{2}}^{n+ \mathfrak k}(0))}}\lesssim 1.$$

For any $\epsilon>0$, by the maximum principle, we have
\begin{equation}
w+\epsilon|y'|^{2-2 \mathfrak k}\geq W \geq w-\epsilon|y'|^{2-2 \mathfrak k}\quad\text{in } B_{r_*'}^{n+ \mathfrak k}(0).
\end{equation}
 One gets $w=W$ by taking $\varepsilon\rightarrow 0$.
In particular,
 \begin{equation}\label{v12kW}
v_{12 \cdots  \mathfrak k}=W \in C^{1,\frac{3}{4}} (\overline{B_{\frac{r_*'}{2}}^{n+ \mathfrak k}(0)}). \end{equation}
  In the original coordinates, it implies
\begin{equation*}
\|v_{12 \cdots  \mathfrak k}\|_{C^{\frac{1}{2}}([0,\frac{r_*'^{2}}{8n}]^{ \mathfrak  k}\times [-\frac{r_*'}{\sqrt{2n}}, \frac{r_*'}{\sqrt{2n}}]^{n- \mathfrak k}}\lesssim 1.
\end{equation*}

Let $r_{ \mathfrak k*}=\frac{r_*'^{2}}{16n}$. By translating the coordinates, we conclude the proof of the present lemma.
\end{proof}

\begin{remark}\label{remark-I12k-regul}
 \eqref{v12kW} in Lemma \ref{lemma-I12k} also implies that
 $$\sqrt{x_a}D_{x_a}I_{12 \cdots  \mathfrak k},\		D_{x_p}I_{12 \cdots  \mathfrak k}\in C^{\frac{1}{4}}([0,2r_{k_*}]^{ \mathfrak k}\times[-5/4,5/4]^{n- \mathfrak k}),\quad a=1, \cdots , \mathfrak k,\	p= \mathfrak k+1, \cdots ,n.$$
\end{remark}

\begin{theorem}\label{thm-I-Ck}
It holds that
\begin{equation*}\label{IkC2}
I\in C^{ \mathfrak k,\frac{1}{8}}([0, \frac{3}{2}r_{ \mathfrak k*}]^{ \mathfrak k}\times[-\frac{9}{8}, \frac{9}{8}]^{n- \mathfrak k}),
\end{equation*}
Moreover, for any multi-index $\beta,\gamma$,
\begin{equation*}
D_{x'}^{\beta}D_{x''}^{\gamma}I, \sqrt{x_a}D_{x_p}D_{x'}^{\beta}D_{x''}^{\gamma}I\in C^{\frac{1}{8}}([0,\frac{3}{2} r_{ \mathfrak k*}]^{ \mathfrak k}\times[-\frac{9}{8}, \frac{9}{8}]^{n- \mathfrak k}),\quad |\beta|+\frac{|\gamma|}{2}=  \mathfrak k,
\end{equation*}
\begin{equation*}
 x_aD_{x'}^{\beta}D_{x''}^{\gamma}I,\sqrt{x_a}x_bD_{x_p}D_{x'}^\beta D_{x''}^\gamma I\in C^{\frac{1}{8}}([0, \frac{3}{2} r_{ \mathfrak k*}]^{ \mathfrak k}\times[-\frac{9}{8}, \frac{9}{8}]^{n- \mathfrak k}),  \quad     \mathfrak k< |\beta|+\frac{|\gamma|}{2}\leq  \mathfrak k+1,\end{equation*}
 \begin{equation*}
 x_ax_bD_{x'}^{\beta}D_{x''}^{\gamma}I\in C^{\frac{1}{8}}([0, \frac{3}{2} r_{ \mathfrak k*}]^{ \mathfrak k}\times[-\frac{9}{8}, \frac{9}{8}]^{n- \mathfrak k}), \quad    \mathfrak k+1<|\beta|+\frac{|\gamma|}{2}\leq  \mathfrak k+2,\end{equation*}
 where $a,b=1, \cdots , \mathfrak k$, $p= \mathfrak k+1, \cdots ,n$.
\end{theorem}

\begin{proof}
The case  $\beta+\gamma=(\underbrace{1,1, \cdots ,1}_{ \mathfrak k},0, \cdots ,0)$ is already obtained in Lemma \ref{lemma-I12k}
Similar to the proof of Theorem \ref{thm-I-C2}, it is sufficient to show that  $\forall |\beta|+\frac{|\gamma|}{2}\geq  \mathfrak k$ and $\beta+\gamma\neq (\underbrace{1,1, \cdots ,1}_{ \mathfrak k},0, \cdots ,0)$, there holds
\begin{equation}\label{I-growth-k}
| D_{x'}^{\beta}D_{x''}^{\gamma}I(x)| \lesssim_{|\beta|+|\gamma|} |x'|^{\frac{1}{8}+ \mathfrak k-|\beta|-\frac{|\gamma|}{2}},
 \end{equation}
 in $\big([0,\frac{3}{2}r_{ \mathfrak k*}]^{ \mathfrak k}\setminus\{0\}^{ \mathfrak k}\big)\times[-\frac{9}{8}, \frac{9}{8}]^{n- \mathfrak k}$.
Then, by Lemma \ref{lemma-Growth-Regularity} and Lemma \ref{lemma-Growth-Regularity2}, we have the desired results.

Without loss of generality, assume $\bar x''=0$.
 Lemma \ref{lemma-I12k} implies
\begin{align*}
I(y',y'')=& y_1y_2 \cdots y_{ \mathfrak k}\int_0^1 \cdots \int_0^1 I_{12 \cdots  \mathfrak k}(s_1y_1,s_2y_2, \cdots ,s_{ \mathfrak k}y_{ \mathfrak k},y'')ds_1ds_2 \cdots ds_{ \mathfrak k}\\
=& y_1y_2 \cdots y_{ \mathfrak k}\int_0^1 \cdots \int_0^1\big[ I_{12 \cdots  \mathfrak k}(0)+O(|y'|^{\frac 12}) +O(|y''|^{\frac 12})\big]ds_1ds_2 \cdots ds_{ \mathfrak k}\\
=& y_1y_2 \cdots y_{ \mathfrak k}I_{12 \cdots  \mathfrak k}(0)+O( |y'|^{ \mathfrak k+\frac{1}{2}}) +O(|y'|^{ \mathfrak k}|y''|^{\frac{1}{2}}).
\end{align*}
 Take $\lambda\in (0,   \frac{3}{2}r_{ \mathfrak k*} ]$.
Set
$$I_{\lambda}(y',y'')=\frac{I(\lambda y',\sqrt{\lambda} y'')}{\lambda},\quad \forall y \in  [0,1]^{ \mathfrak k}\times [-1,1]^{n- \mathfrak k}.$$
By Taylor's expansion, we have
\begin{equation*}
 |I_{\lambda}-\lambda^{ \mathfrak k-1} I_{12 \cdots  \mathfrak k}(0)y_1y_2 \cdots y_{ \mathfrak k}|\leq C \lambda^{ \mathfrak k-1+\frac{1}{4}}.
\end{equation*}
Let $$\widetilde{I}(y)=I_{\lambda}(y',y'')-\lambda^{ \mathfrak k-1} I_{12 \cdots  \mathfrak k}(0)y_1y_2 \cdots y_{ \mathfrak k}.$$
Then by \eqref{estimate-k}, there holds
\begin{equation*}
|D^{\alpha}_{y} \widetilde{I_{\lambda}}|\lesssim_{|\alpha|} 1,\quad \text{in}\quad [1/2,1]\times [0,1]^{ \mathfrak k-1} \times [-1,1]^{n- \mathfrak k}.
\end{equation*}
By Lemma \ref{lemma-interpolation},  we have
\begin{equation*}
|D^{\alpha}_{y} \widetilde{I_{\lambda}}|\lesssim_{|\alpha|} \lambda^{ \mathfrak k-1+\frac{1}{8}},\quad \text{in}\quad  [1/2,1]\times [0,1]^{ \mathfrak k-1} \times [-1,1]^{n- \mathfrak k}.
\end{equation*}
Similarly, we have
\begin{equation*}
|D^{\alpha}_{y} \widetilde{I_{\lambda}}|\lesssim_{|\alpha|} \lambda^{ \mathfrak k-1+\frac{1}{8}},\quad \text{in}\quad \big([0,1]^{ \mathfrak k}\setminus[0,\frac{1}{2})^{ \mathfrak k} \big)\times [-1,1]^{n- \mathfrak k}.
\end{equation*}
Therefore,
\begin{equation*}
\lambda^{|\beta|+\frac{|\gamma|}{2}}|D_{x'}^{\beta}D_{x''}^\gamma\big[I(y)-I_{12 \cdots  \mathfrak k}(0)y_1y_2 \cdots y_{ \mathfrak k}\big]|\lesssim_{|\beta|+|\gamma|} \lambda^{ \mathfrak k+\frac{1}{8}}
\end{equation*}
in $\big([0,\lambda]^{ \mathfrak k}\setminus[0,\frac{1}{2}\lambda)^{ \mathfrak k} \big)\times [-\sqrt{\lambda}, \sqrt{\lambda}]^{n- \mathfrak k}$.
In the original coordinates, this implies \eqref{I-growth-k}
\end{proof}
We now prove the main theorem in this section.
\begin{theorem}\label{theorem-I-k-smooth}
There holds
\begin{equation}\label{I-k smooth}
\|I\|_{C^l([0, r_{ \mathfrak k*}]^{ \mathfrak k}\times[-1,1]^{n- \mathfrak k})}\lesssim_{l} 1,\quad l=0,1,\cdots.
\end{equation}
\end{theorem}

\begin{proof}
Set $$\delta_l=2^{-3l+3 \mathfrak k-3},\	a_{l}=1+2^{2 \mathfrak k-3-2l},\	a_{l}'=1+2^{2 \mathfrak k-4-2l}.$$
 Similar to the proof of Theorem \ref{thm-I-2-Smooth},
for any positive integer $l\geq  \mathfrak k$, we will prove
\begin{equation}\label{I12k-l-1}
D_{x'}^{\beta}D_{x''}^{\gamma}I, \sqrt{x_a}D_{x_p}D_{x'}^{\beta}D_{x''}^{\gamma}I\in C^{\delta_l}([0,a_lr_{ \mathfrak k*}]^{ \mathfrak k}\times[-a_l,a_l]^{n- \mathfrak k})\quad\text{when } |\beta|+\frac{|\gamma|}{2}\leq l,
\end{equation}
\begin{equation}\label{I12k-l-2}
x_a D_{x'}^{\beta}D_{x''}^{\gamma}I, \sqrt{x_a}x_bD_{x_p}D_{x'}^{\beta}D_{x''}^{\gamma}I\in C^{\delta_l}([0,a_lr_{ \mathfrak k*}]^{ \mathfrak k}\times[-a_l,a_l]^{n- \mathfrak k}) , \quad\text{when } |\beta|+\frac{|\gamma|}{2}\leq l+1.\end{equation}
and
\begin{equation}\label{I12k-l-3}
x_a x_bD_{x'}^{\beta}D_{x''}^{\gamma}I\in C^{\delta_l}([0,a_lr_{ \mathfrak k*}]^{ \mathfrak k}\times[-a_l,a_l]^{n- \mathfrak k}), \quad\text{when } |\beta|+\frac{|\gamma|}{2}\leq l+2.\end{equation}
where $a,b=1,2, \cdots , \mathfrak k$, $p= \mathfrak k+1, \cdots ,n$.

 We prove Theorem \eqref{I12k-l-1}-\eqref{I12k-l-3} by induction on $l$. For $l=\mathfrak k$, the result is established by \ref{thm-I-Ck}. Assume the  conclusion holds for some integer $l\geq  \mathfrak k$. We will prove the conclusion also holds for $l+1$.

{\it Step 1.}
 Consider  positive integers $\beta_1, \cdots ,\beta_{ \mathfrak k}\geq1$ with $\beta_1+ \cdots +\beta_{ \mathfrak k}=|\beta|$, and $|\beta|+\frac{|\gamma|}{2}=l-\frac{1}{2}$. When $l= \mathfrak k$ or $\mathfrak k=n$, we do not need this step.
Set $\widetilde{V}=D_{x'}^{\beta}D_{x''}^{\gamma}v$.
Similarly to Theorem \ref{lemma-I12}, $\widetilde{V}$ solves an equation in the form:
\begin{equation}\label{eq-widetilde-Ik}
A\sum_{a=1}^{ \mathfrak k}(x_a\widetilde{V}_{aa}+\widetilde{V}_{a})+\sum_{a=1}^{ \mathfrak k}\sum_{p= \mathfrak k+1}^{n}x_a\frac{a_{ap}}{\sqrt{x_a}}\widetilde{V}_{ap}+\sum_{p,q= \mathfrak k+1}^{n}a_{pq}\widetilde{V}_{pq}+\sum_{p= \mathfrak k+1}^{n}B_p\widetilde{V}_{p}=\widetilde{G}\end{equation}
where $A$, $a_{ap}$, $a_{pq}$ and $B_{p}$ are the same function as in the proof of Lemma \ref{lemma-I12k}, and
$\widetilde{G}$ is smooth with respect to  $x$ and the terms in \eqref{I12k-l-1}-\eqref{I12k-l-3}.
Hence,
$$\|\widetilde{G}\|_{C^{\delta_l}([0,a_lr_{ \mathfrak k*}]^{ \mathfrak k}\times[-a_l,a_l]^{n- \mathfrak k})}\lesssim_{l} 1.$$
 Similarly to Lemma \ref{lemma-I12k}, one can take the following coordinates transformations. Let
\begin{equation*}
x_{1}=\frac{y_1^2+ \cdots +y_{2\beta_1}^2}{4},
\end{equation*}
\begin{equation*}
x_{a}=\frac{y_{2\beta_1+ \cdots 2\beta_{a-1}+1}^2+ \cdots +y_{2\beta_1+ \cdots +2\beta_a}^2}{4},\quad \quad a=2, \cdots , \mathfrak k,
\end{equation*}
and
\begin{equation*}
x_{p}=y_{2\beta_1+ \cdots +2\beta_{ \mathfrak k}+p-k},\quad \quad p= \mathfrak k+1, \cdots ,n.
\end{equation*}
Then, $\overline{V}(y)=\widetilde{V}(x)$ solves
\begin{align}\label{eq-widetilde-Ik-y}\begin{split}
L\overline{V}:=&\sum_{a=1}^{2\beta_1+ \cdots +2\beta_{ \mathfrak k}}\partial_{y_ay_a}\overline{V}\\
&+\sum_{a=1}^{2\beta_1}\sum_{p=2\beta_1+ \cdots +2\beta_{ \mathfrak k}+1}^{2\beta_1+ \cdots +2\beta_{ \mathfrak k}+n- \mathfrak k}\frac{y_a}{2}\frac{\overline{b}_{1p}}{\overline{A}}\overline{V}_{ap}
+\sum_{i=2}^{ \mathfrak k}\sum_{a=2\beta_1+ \cdots +2\beta_{i-1}+1}^{2\beta_1+ \cdots +2\beta_i}\sum_{p=2\beta_1+ \cdots +2\beta_{ \mathfrak k}+1}^{2\beta_1+ \cdots +2\beta_{ \mathfrak k}+n- \mathfrak k}\frac{y_a}{2}\frac{\overline{b}_{ip}}{\overline{A}}\overline{V}_{ap}\\
&+\sum_{p,q=2\beta_1+ \cdots +2\beta_{ \mathfrak k}+1}^{2\beta_1+ \cdots +2\beta_{ \mathfrak k }+n- \mathfrak k}\frac{\overline{a}_{pq}}{\overline{A}}\overline{V}_{pq}+\sum_{p=2\beta_1+ \cdots +2\beta_{ \mathfrak k}+1}^{2\beta_1+ \cdots +2\beta_{ \mathfrak k}+n- \mathfrak k}\frac{\overline{B}_p}{\overline{A}}\overline{V}_{p}\\
&=\overline{G},
\end{split}\end{align}
in   $\bigg( \big(B^{2\beta_1}_{2\sqrt{a_{l}r_{ \mathfrak k*}}}(0)\times  \cdots \times B^{2\beta_{ \mathfrak k}}_{2\sqrt{a_{l}r_{ \mathfrak k*}}}(0)\big)\setminus \{0\}^{2\beta_1+ \cdots +2\beta_{ \mathfrak k}}\bigg)\times[-a_{l},a_{l}]^{n- \mathfrak k}$,
where
$$b_{ap}(y)=\frac{r_a}{2}\frac{a_{ap}}{a\sqrt{x_a}}(x),\quad \widehat{G}\in C^{\delta_l}(B^{2\beta_1}_{2\sqrt{a_{l}r_{ \mathfrak k*}}}(0)\times  \cdots \times B^{2\beta_{ \mathfrak k}}_{2\sqrt{a_{l}r_{ \mathfrak k*}}}(0)\times[-a_{l},a_{l}]^{n- \mathfrak k}).$$
This is a uniformly elliptic equation with $ C^{\delta_l}(B^{2\beta_1}_{2\sqrt{a_{l}r_{ \mathfrak k*}}}(0)\times  \cdots \times B^{2\beta^{ \mathfrak k}}_{2\sqrt{a_{l}r_{ \mathfrak k*}}}(0)\times[-a_{l},a_{l}]^{n- \mathfrak k})$ coefficients.
By standard elliptic regularity results, we get
$$\widetilde{V}\in C^{2,\delta_l}(B^{2\beta_1}_{2\sqrt{a_{l}'r_{ \mathfrak k*}}}(0)\times  \cdots \times B^{2\beta_{ \mathfrak k}}_{2\sqrt{a_{l}'r_{ \mathfrak k*}}}(0)\times[-a_{l}',a_{l}']^{n- \mathfrak k}).$$
 Changing
back to $x$, we get
$$D_{x}\widetilde{V}, \sqrt{x_ax_b}D_{x_a} D_{x_b}\widetilde{V},\sqrt{x_a}D_{x_ax_p}\widetilde{V},D_{x_px_q}\widetilde{V}\in C^{\frac{\delta_l}{2}}([0,a_l'r_{ \mathfrak k*}]^{ \mathfrak k}\times[-a_l',a_l']^{n- \mathfrak k}),$$ where $a,b=1, \cdots , \mathfrak k$ and $p,q= \mathfrak k+1, \cdots ,n$.

{\it Step 2.}
For any  positive integers $\beta_1, \cdots ,\beta_{ \mathfrak k}\geq1$ and $\beta_1+ \cdots +\beta_{ \mathfrak k}=|\beta|$, $|\beta|+\frac{|\gamma|}{2}=l$, let $\widehat{V}=D_{x'}^{\beta}D_{x''}^{\gamma}v$. Repeating the arguments in Step 1, one gets

$$D_{x}\widehat{V}, \sqrt{x_ax_b}D_{x_a} D_{x_b}\widehat{V},\sqrt{x_a}D_{x_ax_p}\widehat{V},D_{x_px_q}\widehat{V}\in C^{\frac{\delta_l}{2}}([0,a_{l+1}r_{ \mathfrak k*}]^{ \mathfrak k}\times[-a_{l+1},a_{l+1}]^{n- \mathfrak k}),$$ where $a,b=1, \cdots , \mathfrak k$ and $p,q= \mathfrak k+1, \cdots ,n$. This implies
$$\|D_{x'}^{\beta}D_{x''}^{\gamma}I_{12 \cdots  \mathfrak k}\|_{C^{\frac{\delta_l}{2}}([0,a_{l+1}r_{ \mathfrak k*}]^{ \mathfrak k}\times[-a_{l+1},a_{l+1}]^{n- \mathfrak k})}\lesssim_{|\beta|+|\gamma|} 1,\quad |\beta|+|\gamma|\leq l+1- \mathfrak k.$$
\smallskip

{\it Step 3.}
We now prove the remaining cases. Without loss of generality, assume $\bar x''=0$.
Let $P_{l+1}$ be a polynomial satisfying
\begin{equation*}
D_{x'}^{\beta}D_{x''}^{\gamma}P_{l+1}(0)=0,\quad\text{when } (\underbrace{1, \cdots ,1}_{ \mathfrak k},\underbrace{0 \cdots ,0}_{n- \mathfrak k})\nleq \beta,
\end{equation*}
\begin{equation*}
D_{x_1}D_{x_2} \cdots D_{x_{ \mathfrak k}}D_{x'}^{\beta}D_{x''}^{\gamma}P_{l+1}(0)=D_{x'}^{\beta}D_{x''}^{\gamma}I_{12 \cdots  \mathfrak k}(0),\quad\text{for any } \beta, \gamma \text{ with }  |\beta|+\frac{|\gamma|}{2}\leq l+1- \mathfrak k,
\end{equation*}
and
\begin{equation*}
D_{x'}^{\beta}D_{x''}^{\gamma}P_{l+1}(0)=0,\quad\text{when } |\beta|+\frac{|\gamma|}{2}>l+1.
\end{equation*}
Near the origin, by Taylor's expansion, we have
\begin{align*}
&I(y',y'')\\
=& y_1y_2 \cdots y_{ \mathfrak k}\int_0^1 \cdots \int_0^1 I_{12 \cdots  \mathfrak k}(s_1y_1, \cdots ,s_{ \mathfrak k}y_{ \mathfrak k},y'')ds_1 \cdots ds_{ \mathfrak k}\\
=&y_1y_2 \cdots y_{ \mathfrak k}\bigg(\sum_{|\beta|\leq l- \mathfrak k+1}\frac{D^{\beta}I_{12 \cdots  \mathfrak k}(0',y'')}{\beta!(\beta_1+1) \cdots (\beta_{ \mathfrak k}+1)}y'^{\beta}+O( |y'|^{l- \mathfrak k+1+\frac{\delta_l}{2}})\bigg)\\
=&y_1y_2 \cdots y_{ \mathfrak k}\bigg(\sum_{|\beta|\leq l- \mathfrak k+1}\sum_{\frac{|\gamma|}{2}\leq l- \mathfrak k+1-|\beta|}\frac{D^{\gamma}D^{\beta}I_{12 \cdots  \mathfrak k}(0)}{\beta!(\beta_1+1) \cdots (\beta_{ \mathfrak k}+1)\gamma!}y'^{\beta}y''^{\gamma}\\
&+O( |y'|^{ |\beta|} |y''|^{2l+2-2 \mathfrak k-2 |\beta|+\frac{\delta_l}{2}} )+O( |y'|^{l- \mathfrak k+1+\frac{\delta_l}{2}})\bigg)\\
=&y_1y_2 \cdots y_{ \mathfrak k}\bigg(\sum_{|\beta|\leq l- \mathfrak k+1}\sum_{\frac{|\gamma|}{2}\leq l- \mathfrak k+1-|\beta|}\frac{D^{\gamma}D^{\beta}P_{l+1,12 \cdots  \mathfrak k}(0)}{\beta!(\beta_1+1) \cdots (\beta_{ \mathfrak k}+1)\gamma!}y'^{\beta}y''^{\gamma}\\
&+O(|y'|^{ |\beta|} |y''|^{2l+2-2 \mathfrak k-2 |\beta|+\frac{\delta_l}{2}} )+O(|y'|^{l- \mathfrak k+1+\frac{\delta_l}{2}})\bigg)\\
=&P_{l+1}(y',y'')+
O( |y'|^{ |\beta|+ \mathfrak k} |y''|^{2l+2-2 \mathfrak k-2 |\beta|+\frac{\delta_l}{2}} )+O( |y'|^{l+1+\frac{\delta_l}{2}}).
\end{align*}
For any multi-indices $\beta,\gamma$ satisfying $ |\beta|+\frac{|\gamma|}{2}\geq l+1$ and
$(\underbrace{1, \cdots ,1}_{ \mathfrak k},\underbrace{0 \cdots ,0}_{n- \mathfrak k})\nleq \beta$ if $ |\beta|+\frac{|\gamma|}{2}=l+1$, we will prove
\begin{equation}\label{I-growth-l-k}
| D_{x'}^{\beta}D_{x''}^{\gamma}I| =| D_{x'}^{\beta}D_{x''}^{\gamma}[I-P_{l+1}(x)]| \lesssim_{L(|\beta|+|\gamma|,l,\mathfrak k, n)} |x'|^{\frac{\delta_l}{8}+l+1-|\beta|-\frac{|\gamma|}{2}},
 \end{equation}
  where $L(|\beta|+|\gamma|,l,\mathfrak k, n)$ is a positive integer depending only on $|\beta|+|\gamma|, l, \mathfrak k$ and $n$.
Then, the desired result can be obtained directly by Lemma \ref{lemma-Growth-Regularity} and Lemma \ref{lemma-Growth-Regularity2}.

 Take  $\lambda\in (0,a_{l+1}r_{ \mathfrak k*}]$.
For any $y \in [0,1]^{ \mathfrak k}\times [-1,1]^{n- \mathfrak k}$,
set $I_{\lambda}(y)=\frac{I(\lambda y',\sqrt{\lambda} y'')}{\lambda}$.
Then, \begin{equation}
|I_{\lambda}(y)- \frac{P_{l+1}(\lambda y',\sqrt{\lambda} y'')}{\lambda}|\leq C \lambda^{l+\frac{\delta_l}{4}}.
\end{equation}
Let $$\widetilde{I_{\lambda}}(y)=I_{\lambda}(y',y'')-\frac{ P_{l+1}(\lambda y',\sqrt{\lambda} y'')}{\lambda}.$$
Following the proof in Theorem \ref{thm-I-Ck}, and using Lemma \ref{lemma-interpolation} along with \eqref{estimate-k}, 
we can obtain   
\begin{equation*}
|D^{\alpha}_{y} \widetilde{I}|\lesssim_{L(|\alpha|,l,\mathfrak k, n)} \lambda^{l+\frac{\delta_l}{8}},\quad \text{in}\quad \big([0,1]^{ \mathfrak k}\setminus[0,\frac{1}{2})^{ \mathfrak k}\big)\times [-1,1]^{n- \mathfrak k}.
\end{equation*}  
In the original coordinates, this implies \eqref{I-growth-l-k}.
This demonstrates that the inductive hypothesis holds for $l+1$, thus completing the proof.
\end{proof}

Now we are ready to prove Theorem \ref{mainthm0}.  
\smallskip \smallskip \\
{\textbf{The proof of Theorem \ref{mainthm0}}:}
We prove it by the method of induction.
When $n=2$, the conclusion can be derived from \cite{Rubin2015} and \cite{Huang2023}.
\par Suppose that the conclusion
holds for some integer $n=m\geq 2$. Let $P$ be a simple convex polytope in $\mathbb R^{m+1}$.  By definition, any $m$-face of $P$ is a simple convex polytope. Let $\Sigma$ be an $m$-face of $P$ such that 
$$\Sigma=\overline{P}\bigcap\{x|l_1(x)=0\} = \overline{P}\bigcap\{x_1=0\}.$$
The vertices of $\Sigma$ are as follows:
 $$\{p_{i_1},\cdots,p_{i_{N'}}\}=\{p_{1},\cdots,p_{N}\}\bigcap\{x|l_1(x)=0\}.$$
By inductive hypothesis, \eqref{intro1}-\eqref{intro1-1} admits a unique convex solution $u'$ with
$u'(p_{i_j})=\alpha_{i_j}$, $j=1,...,N'$. Therefore, the boundary condition of $u$ can be determined by solving the Guillemin boundary value problem on each $m-$face of $P$. Denote the boundary value of $u$ by $\varphi$. Then, $\varphi$ satisfies the assumption in  Theorem 2.1. This implies the existence of an Alexandrov solution $u$ of \eqref{intro1} with boundary value $\varphi$. By Theorem \ref{thm2.2}, Theorem \ref{thm3-803} and Theorem \ref{thmlocal-k}, $u$ satisfies \eqref{intro1}-\eqref{intro1-1}.

This shows the conclusion holds for $n=m+1$. $\hfill\square$

\section{Appendix A-Some calculus lemmas}
In this Appendix, we prove some simple calculus lemmas which are used repeatedly in the present paper.
\begin{lemma}\label{lem-pre}
Suppose $\varphi, \phi\in C^{\mathfrak k}(\overline{(\mathbb R^+)^{\mathfrak k}\times \mathbb R^{n-\mathfrak k}})$
and $\varphi=\phi$ on $\partial((\mathbb R^+)^{\mathfrak k}\times \mathbb R^{n-\mathfrak k})$. Then there holds
\begin{equation}
|\varphi(x)-\phi(x)|\le C\prod_{a=1}^\mathfrak k x_a,\quad \text{in}\quad [0,1]^{\mathfrak k}\times [-1,1]^{n-\mathfrak k}.
\end{equation}
\end{lemma}
\begin{proof}
For $\mathfrak k=1$, it is obvious. We only need to consider $\mathfrak k\ge 2$.
\begin{equation}
\begin{split}
&(\varphi-\phi)(x)\\
=&x_1\int_0^1 (\varphi-\phi)_{x_1}(\lambda_1 x_1,x_2,\cdots,x_n)d\lambda_1\\
=& x_{1}x_2\int_0^1\int_0^1 (\varphi-\phi)_{x_{1}x_2}(\lambda_1 x_1,\lambda_2 x_2,x_3,\cdots,x_n)d\lambda_{1}d\lambda_2\\
=&\prod_{a=1}^\mathfrak k  x_a \int_{[0,1]^\mathfrak k}(\varphi-\phi)_{x_1\cdots x_{\mathfrak k}}(\lambda_1 x_1,\cdots,\lambda_\mathfrak kx_{\mathfrak k},x_{\mathfrak k+1},\cdots, x_n)d\lambda_{1}\cdots d\lambda_\mathfrak k.
\end{split}
\end{equation}
\end{proof}

\begin{lemma}\label{lemma-Growth-Regularity}
Let $f\in C\big([0,1]^{ \mathfrak k}\times[-1,1]^{n- \mathfrak k}\big)\bigcap C^{1}\big([0,1]^{ \mathfrak k}\times[-1,1]^{n- \mathfrak k}\setminus \{0\}^{ \mathfrak k}\times[-1,1]^{n- \mathfrak k}\big)$  and
$x'=(x_1,\cdots,x_{ \mathfrak k})$. Suppose
there exist constants $M >0$ and $\delta\in(0,1]$ such that, for any $x\in[0,1]^{ \mathfrak k}\times[-1,1]^{n- \mathfrak k}$,
$$|f(x)|\leq M |x'|^{\delta},$$ and for any $x\in[0,1]^{ \mathfrak k}\times[-1,1]^{n- \mathfrak k}\setminus \{0\}^{ \mathfrak k}\times[-1,1]^{n- \mathfrak k}$,
$$|\nabla f(x)|\leq M |x'|^{\delta-1}.$$
Then
 $$\|f\|_{C^{\delta}\big([0,1]^{ \mathfrak k}\times[-1,1]^{n- \mathfrak k}\big)}\leq C(n)M.$$
\end{lemma}
\begin{proof}
  \begin{itemize}
    \item[(1).] $x''$ is fixed. Consider two points $x=(x',x'')$ and $y=(y',x'')$. Without loss of generality, assume $|y'|\leq |x'|$. Consider the following two cases.
    \\ (1.1) $|x'-y'|<\frac{1}{2}|x'| $. Then, we have
    \begin{equation*}
    \begin{split}
    |f(x)-f(y)|&\le |\nabla f\big(x'+\theta(y'-x'),x''\big)||x'-y'|\\
    &\leq 2^{1-\delta}M |x'|^{\delta-1}|x'-y'|\\
    &\leq M|x-y|^{\delta},
    \end{split}
    \end{equation*}
    where $\theta\in (0,1).$
    \\ (1.2) $|x'-y'|\geq\frac{1}{2}|x'| $. Then, we have
 \begin{equation*}
    |f(x)-f(y)|\le  |f(x)|+|f(y)| \leq 2M|x'|^{\delta}\leq 2^{1+\delta}M|x-y|^{\delta}.
    \end{equation*}

\item[(2).]  $x'$ is fixed, i.e., $x=(x',x'')$ and $y=(x',y'')$. Consider the following two cases.
\\ (2.1) $|x''-y''|<|x'| $. Then, we have
    \begin{equation*}
    \begin{split}
    |f(x)-f(y)|&\le |\nabla f\big(x',x''+\theta(y''-x'')\big)||x''-y''|\\
    &\leq M |x'|^{\delta-1}|x''-y''|\\
    &\leq M|x-y|^{\delta},
    \end{split}
    \end{equation*}
    where $\theta\in (0,1).$
\\ (2.2)  $|x''-y''|\geq|x'| $. Then, we have
 \begin{equation*}
    |f(x)-f(y)|\le  |f(x)|+|f(y)| \leq 2M|x'|^{\delta}\leq 2M|x-y|^{\delta}.
    \end{equation*}
\item[(3).] For general $x,y$, one gets
\begin{equation*}
|f(x)-f(y)|\le  2^{2+\delta}M|x-y|^{\delta}\le  8M|x-y|^{\delta}.
\end{equation*}
  \end{itemize}
\end{proof}

\begin{lemma}\label{lemma-Growth-Regularity2}
Let $f\in C\big([0,1]^{ \mathfrak k}\times[-1,1]^{n- \mathfrak k}\big)\bigcap C^{1}\big([0,1]^{ \mathfrak k}\times[-1,1]^{n- \mathfrak k}\setminus \{0\}^{ \mathfrak k}\times[-1,1]^{n- \mathfrak k}\big)$  and
$x'=(x_1,\cdots,x_{ \mathfrak k})$. Suppose
there exists constants $M>0$ and  $\delta\in(0,\frac{1}{2}]$ such that, for any $x\in[0,1]^{ \mathfrak k}\times[-1,1]^{n- \mathfrak k}$,
$$|f(x)|\leq M |x'|^{\delta-\frac{1}{2}},$$ and for any $x\in[0,1]^{ \mathfrak k}\times[-1,1]^{n- \mathfrak k}\setminus \{0\}^{ \mathfrak k}\times[-1,1]^{n- \mathfrak k}$,
$$|\nabla f(x)|\leq M |x'|^{\delta-\frac{3}{2}}.$$
Then, for any $a=1,.., \mathfrak k$, the following estimate holds:
$$\|\sqrt{x_a}f\|_{C^{\delta}\big([0,1]^{ \mathfrak k}\times[-1,1]^{n- \mathfrak k}\big)}\leq C(n)M. $$
\end{lemma}

\begin{proof}
  \begin{itemize}
    \item[(1).]  $x''$ is fixed. Consider two points $x=(x',x'')$ and $y=(y',x'')$. Without loss of generality,  assume $|y'|\leq |x'|$. Consider the following two cases.
    \\ (1.1) $|x'-y'|<\frac{1}{2}|x'| $. Then, we have
    \begin{equation*}
    \begin{split}
    |\sqrt{x_a}f(x)-\sqrt{y_a}f(y)|& \leq\sqrt{y_a} |f(x)-f(y)|+ |\sqrt{x_a}-\sqrt{y_a}||f(x)|\\
    &\le|x'|^{\frac{1}{2}} |\nabla f\big(x'+\theta(y'-x'),x''\big)||x'-y'|+|x_a-y_a|^{\frac{1}{2}}M|x'|^{\delta-\frac{1}{2}}\\
    &\leq 2^{\frac{3}{2}-\delta}M |x'|^{\delta-1}|x'-y'|+2^{\frac{1}{2}-\delta}M|x'-y'|^{\delta}\\
    &\leq 2^{\frac{3}{2}}M|x-y|^{\delta},
\frac{}{}    \end{split}
    \end{equation*}
    where $\theta\in (0,1).$
    \\ (1.2) $|x'-y'|\geq\frac{1}{2}|x'| $. Then, we have
 \begin{equation*}
    | \sqrt{x_a}f(x)-\sqrt{y_a}f(y)|\le  |x'|^{\frac{1}{2}}|f(x)|+|y'|^{\frac{1}{2}}|f(y)| \leq 2M|x'|^{\delta}\leq 2^{1+\delta}M|x-y|^{\delta}.
    \end{equation*}

\item[(2).] $x'$ is fixed, i.e., $x=(x',x'')$ and $y=(x',y'')$. Consider the following two cases.
\\ (2.1) $|x''-y''|<|x'| $. Then, we have
    \begin{equation*}
    \begin{split}
     |\sqrt{x_a}f(x)-\sqrt{x_a}f(y)|&\le  |x'|^{\frac{1}{2}}|\nabla f\big(x',x''+\theta(y''-x'')\big)||x''-y''|\\
    &\leq M |x'|^{\delta-1}|x''-y''|\\
    &\leq M|x-y|^{\delta},
    \end{split}
    \end{equation*}
    where $\theta\in (0,1).$
\\ (2.2)  $|x''-y''|\geq|x'| $. Then, we have
 \begin{equation*}
 |\sqrt{x_a}f(x)-\sqrt{x_a}f(y)|\le  |x'|^{\frac{1}{2}}( |f(x)|+|f(y)|) \leq 2M|x'|^{\delta}\leq 2M|x-y|^{\delta}.
    \end{equation*}
\item[(3).] For general $x,y$, one gets
\begin{equation*}
 |\sqrt{x_a}f(x)-\sqrt{y_a}f(y)|\le 8M|x-y|^{\delta}.
\end{equation*}
  \end{itemize}
\end{proof}

\smallskip
 In the following, we state a simple interpolation inequality. This interpolation inequality can be derived from the Gagliardo-Nirenberg inequality in bounded domains. For convenience, we give an elementary proof here.

\begin{lemma}\label{lemma-interpolation}
Let $k\geq2$, and suppose $f\in C^k([0,1]^n)$ with 
$$\|f\|_{L^{\infty}([0,1]^n)}\leq A,\quad \max_{|\alpha|\leq k}\|D^{\alpha}f\|_{L^{\infty}([0,1]^n)}\leq B,$$
where  $A$ and $B$ are positive constants with $A\leq B$.
Then, for any multi-index $\alpha$ with $|\alpha|\leq k$,
\begin{equation}\label{ineq-inte}\|D^{\alpha}f\|_{L^{\infty}([0,1]^n)}\leq CA^{1-\frac{|\alpha|}{k}}B^{\frac{|\alpha|}{k}},\end{equation}where the constant $C$ depends only on $k$.
\end{lemma}

\begin{proof}
For any integer $l$ with $0\leq l\leq k$, define
$$A_l=\max_{|\alpha|\leq l}\|D^{\alpha}f\|_{L^{\infty}([0,1]^n)}.$$
Then, $A_l$ is nondecreasing with respect to $l$ and $A_l\leq A_{k}\leq B$. It suffices to prove that for any $\varepsilon\in[0,1]$, there holds
\begin{equation}\label{ineq-inte2}\varepsilon^{l}A_l\leq C(k)\big(A_0+\varepsilon^{k}A_{k}\big)\leq C(k)\big(A+\varepsilon^{k}B\big).\end{equation}
Then, \eqref{ineq-inte} follows by taking $\varepsilon=\sqrt[k]{\frac{A}{B}}$.

We will prove \eqref{ineq-inte2} by induction on  $k$.
Without loss of generality, assume $A_0>0$.
For any $i\in\{1,...,n\}$, and any $x\in[0,1]^n$, take $y=x+he_i\in[0,1]^n$ with $|h|=\sqrt{\frac{A_0}{4A_2}}\in(0,\frac{1}{2}]$. By Taylor's expansion, we have
$$f(y)-f(x)=f_{i}(x)h+\frac{f_{ii}(x+\theta(y-x))}{2}h^2$$
for some $\theta\in(0,1)$.  This implies
$$\sqrt{\frac{A_0}{4A_2}}|f_i(x)|\leq 2A_0+\frac{A_2}{2}\frac{A_0}{4A_2}.$$
Thus,
 \begin{equation}\label{ineq-inte3}
 A_1\leq  5\sqrt{A_0A_2}.
 \end{equation}
This confirms that \eqref{ineq-inte2} holds for $k=2$.

Suppose that \eqref{ineq-inte2} holds for $k=m$, where $m\geq2$. Now, let $k=m+1$. Similar to the proof of \eqref{ineq-inte3},  for any multi-index $\alpha$ with $|\alpha|= m$, we have $$ \|D^{\alpha}f\|_{L^{\infty}([0,1]^n)}\leq  5\sqrt{A_{m-1}A_{m+1}}.$$
Hence,
$$A_m\leq  5\sqrt{A_{m-1}A_{m+1}}.$$
Then, for any $\mu>0$ and any $\varepsilon\in[0,1]$, we have
  \begin{equation}\label{ineq-inte4}\varepsilon^{m}A_m\leq 3(\mu\varepsilon^{m-1}A_{m-1}+\frac{1}{\mu}\varepsilon^{m+1}A_{m+1})\leq 3\big[\mu C(m)\big(A_0+\varepsilon^{m}A_{m}\big)+\frac{1}{\mu} \varepsilon^{m+1}A_{m+1}\big].\end{equation}
Therefore, \eqref{ineq-inte2} holds for $k=m+1$ by taking $3\mu C(m)\leq\frac{1}{2}$ in \eqref{ineq-inte4}.
\end{proof}

\section{Appendix B-Existence of Alexandrov's solution}

In this Appendix, we discuss the Dirichlet problem of Monge-Amp\`ere equation. Consider the following Dirichlet problem.
\begin{equation}\label{app1}
\begin{cases}
\det D^2 u=\mu,&\quad \text{in}\quad \Omega,\\
u=g,&\quad \text{on}\quad \partial\Omega
\end{cases}
\end{equation}
where $\Omega$ is a bounded convex domain and $\mu$ is a non-negative Radon measure with $\mu(\Omega)<+\infty$.

 It is proved that \eqref{app1} admits a unique  Alexandrov solution $u\in C(\overline \Omega)$ provided that $g=0$([Theorem 2.13,\cite{Figalli2017book}]) or $\Omega$ strictly convex and $g\in C(\partial\Omega)$([Theorem 2.14,\cite{Figalli2017book}],[Theorem 1.6.2,\cite{Gutierrez2001book}]) or $g\in C(\overline{\Omega})$ convex in $\Omega$([Theorem 8.2.7,\cite{Han2016book}]). However, none of the above results can be applied directly in our present paper  since the domain $\Omega$ is a polytope and the boundary data $g\in C(\partial \Omega)$.

\begin{definition}
We call $g\in C(\partial\Omega)$ a convex function on $\partial\Omega$ iff for any $x,y,tx+(1-t)y\in \partial\Omega$, $t\in [0,1]$ there holds $g(tx+(1-t)y)\le tg(x)+(1-t)g(y)$.
\end{definition}
Then we have the following theorem.
\begin{theorem}\label{appthm}
Let $\Omega$ be a simple convex polytope in $\mathbb R^n$ and $\mu(\Omega)<+\infty$. Suppose $g\in C(\partial\Omega)$ is a convex function on $\partial\Omega$. Then there exists a unique Alexandrov solution $u\in C(\overline{\Omega})$ solves \eqref{app1}.
\end{theorem}
We can use Perron's method to prove Theorem \ref{appthm}. In fact, following the line of the proof of [Theorem 2.14,\cite{Figalli2017book}], the important thing is show the Alexandrov solution constructed by Perron's method is continuous up to the boundary.
In order to get the continuity, it is enough to show the existence of the following affine functions.
\begin{lemma}\label{lemdb1}
Let $\Omega$ be a   simple convex polytope in $\mathbb R^n$  and $g$ is a continuous convex function on $\partial\Omega$. Let $x_0$ be a point on $\partial\Omega$. Then $\forall \varepsilon>0$, there exist an affine functions $l_{x_0,\varepsilon}^{-}$ such that
\begin{equation}
\begin{split}
&l_{x_0,\varepsilon}^-(x)\le g(x),\quad x\in \partial\Omega,\\
& g(x_0)-\varepsilon\le l_{x_0,\varepsilon}^-(x_0).
\end{split}
\end{equation}
\end{lemma}
\begin{proof}
Without loss of generality, we assume $x_0=0$.
\par If $x_0$ is a vertex of $\Omega$. Since $\Omega$ is a polytope, without loss of generality, we may assume
$$\Omega\subset \{x_n>0\},\quad \partial \Omega\cap \{x_n=0\}=0.$$
By the continuity of $g$, $\exists \delta>0$, such that
$$|g(x)-g(0)|\le \varepsilon,\quad forall x\in \partial\Omega\cap B_\delta(0).$$
 By the assumption, there exists $\eta>0$ such that
$\partial\Omega\backslash B_{\delta}(0)\subset \{x_n\ge \eta\}$.  Let
\begin{equation*}
l_{0,\varepsilon}^{-}=g(0)- \varepsilon - \frac{\|g-g(0)\|_{L^\infty(\partial\Omega)}}{\eta}x_n.
\end{equation*}
\par If $0\in K=\partial\Omega\cap \{x_l=\cdots=x_n=0\}$ for some $2\le l\le n$ and $0\notin \partial K$(Here $\partial K$ is the relative boundary of $K$). Moreover, $\partial\Omega\cap \{x_n=0\}=K$.  Since $g(x',0'')$ is convex in $K$, $x'=(x_1,\cdots,x_{l-1})$, $x''=(x_l,\cdots,x_n)$. We can find an affine function $l_0(x')=a'\cdot x'+b$ such that $g(0)=l_0(0)$ and $g(x',0)\ge l_0(x')$ for $(x',0)\in K$. Let $l_1(x)=l_0(x')$. By the continuity of $g$ and $l_0$, $\exists \delta>0$, such that $|g(x)-g(y)|+|l_1(x)-l_1(y)|\le \varepsilon$, for any $x,y\in \partial\Omega$ and $|x-y|<\delta$. By the assumption, there exists $\eta>0$ such that
$\partial\Omega\backslash B_{\delta}(K)\subset \{x_n\ge \eta\}$.
Let
\begin{equation*}
l_{0,\varepsilon}^{-}=l_1(x)- \varepsilon - \frac{\|g-l_1\|_{L^\infty(\partial\Omega)}}{\eta}x_n.
\end{equation*}

Then, by direct computation, one knows the present lemma follows directly.
\end{proof}


\begin{thebibliography}{10}

\bibitem{Abreu1998}
M.~Abreu, \emph{K{\"a}hler geometry of toric varieties and extremal metrics},
  International {J}ournal of {M}athematics \textbf{9} (1998), no.~06, 641--651.

\bibitem{Caffarelli1989-1}
L.~Caffarelli, \emph{Interior a priori estimates for solutions of fully
  non-linear equations}, Annals of {M}athematics \textbf{130} (1989), no.~1,
  189--213.

\bibitem{Caffarelli1990-1}
\bysame, \emph{A localization property of viscosity solutions to the
  {M}onge--{A}mp{\`e}re equation and their strict convexity}, Annals of
  mathematics \textbf{131} (1990), no.~1, 129--134.

\bibitem{ChenHanLiSheng2014}
B.~Chen, Q.~Han, A.-M. Li, and L.~Sheng, \emph{Interior estimates for the
  n-dimensional {A}breu's equation}, Advances in {m}athematics \textbf{251}
  (2014), 35--46.

\bibitem{ChenLiSheng2014}
B.~Chen, A.-M. Li, and L.~Sheng, \emph{Extremal metrics on toric surfaces}, Advances in {m}athematics \textbf{340}
  (2018), 363--405.


\bibitem{ChenLiuWang2023}
S.~Chen, J.~Liu, and X.-J. Wang, \emph{Regularity of optimal mapping between
  hypercubes}, Advanced {N}onlinear {S}tudies \textbf{23} (2023), no.~1,
  20230087.
  \bibitem{ChenCheng12021}
X.~ Chen and J.~Cheng, \emph{On the constant scalar curvature K\"ahler metrics (I)-A priori estimates}, J. Amer. Math. Soc. 34 (2021), no.4, 937-1009.


 \bibitem{ChenCheng22021}
\bysame
 \emph{On the constant scalar curvature K\"ahler metrics (II)-Existence results}, J. Amer. Math. Soc. 34 (2021), no.4, 909-936.





\bibitem{Donaldson2005}
S.~Donaldson, \emph{Interior estimates for solutions of {A}breu's equation},
  Collectanea {M}athematica \textbf{56} (2005), no.~2, 103--142.

\bibitem{Donaldson2008}
\bysame, \emph{Extremal metrics on toric surfaces: a continuity method},
  Journal of {D}ifferential {G}eometry \textbf{79} (2008), no.~3, 389--432.

\bibitem{Donaldson2009}
\bysame, \emph{Constant scalar curvature metrics on toric surfaces}, Geometric
  and {F}unctional {A}nalysis \textbf{19} (2009), no.~1, 83--136.

\bibitem{Figalli2017book}
A.~Figalli, \emph{The {M}onge--{A}mp{\`e}re equation and its applications},
  European Mathematical Society, 2017.

\bibitem{GilbargTrudinger2001}
D.~Gilbarg and N.~Trudinger, \emph{Elliptic partial differential equations of
  second order}, vol. 224, springer, 2001.

\bibitem{Guillemin1994}
V.~Guillemin, \emph{K\"ahler structures on toric varieties}, Journal of
  {D}ifferential {G}eometry \textbf{40} (1994), no.~2, 285--309.

\bibitem{Guillemin2012}
\bysame, \emph{Moment maps and combinatorial invariants of {H}amiltonian
  ${T}^n$-spaces}, vol. 122, Springer Science \& Business Media, 2012.

\bibitem{Gutierrez2001book}
E.~Guti{\'e}rrez, \emph{The {M}onge--{A}mp{\`e}re equation}, vol.~44, Springer,
  2001.

\bibitem{Han2016book}
Q.~Han, \emph{Nonlinear elliptic equations of the second order}, vol. 171,
  American Mathematical Society, 2016.

\bibitem{HanJiang2023}
Q.~Han and X.~Jiang, \emph{Boundary regularity of minimal graphs in the
  hyperbolic space}, Journal f{\"u}r die reine und angewandte {M}athematik
  ({C}relle's {J}ournal) \textbf{2023} (2023), no.~801, 239--272.

\bibitem{HanJiangShen2020}
Q.~Han, X.~Jiang, and W.~Shen, \emph{The loewner-nirenberg problem in cones},
Journal of Functional Analysis, 287, (8), (2024), 110566.

\bibitem{Hisamoto2016}
T. ~Hisamoto, \emph{Stability and coercivity for toric polarizations}, arXiv:1610.0799


\bibitem{HongHuang2022}
J.~Hong and G.~Huang, \emph{Boundary h{\"o}lder estimates for a class of
  degenerate elliptic equations in piecewise smooth domains}, Chinese Annals of
  Mathematics, Series B \textbf{43} (2022), no.~5, 719--738.

\bibitem{HongHuangWang2011}
J.~Hong, G.~Huang, and W.~Wang, \emph{Existence of global smooth solutions to
  {D}irichlet problem for degenerate elliptic {M}onge-{A}mp\`ere equations},
  Comm. Partial Differential Equations \textbf{36} (2011), no.~4, 635--656.
  \MR{2763326}

\bibitem{Huang2023}
G.~Huang, \emph{The {G}uillemin boundary problem for {M}onge-{A}mp{\`e}re
  equation in the polygon}, Advances in Mathematics \textbf{415} (2023),
  108885.

\bibitem{HuangShen2023}
G.~Huang and W.~Shen, \emph{The {M}onge-{A}mp{\`e}re equation in cones and
  polygonal domains}, arXiv:2312.01405 (2023).


\bibitem{LiSheng2023}A.-M. Li and L.~Sheng, \emph{Extremal K{\"a}hler metrics of toric manifolds}, Chinese Annals of
  Mathematics, Series B \textbf{44} (2023) no.6, 827-836.


\bibitem{Jhaveri2019}
Y.~Jhaveri, \emph{On the (in) stability of the identity map in optimal
  transportation}, {C}alculus of {V}ariations and {P}artial {D}ifferential
  {E}quations \textbf{58} (2019), no.~3, 96.

\bibitem{JianLuWang2022}
H.~Jian, J.~Lu, and X.-J. Wang, \emph{A boundary expansion of solutions to
  nonlinear singular elliptic equations}, {S}cience {C}hina {M}athematics
  \textbf{65} (2022), 9--30.

\bibitem{JianWang2013}
H.~Jian and X.-J. Wang, \emph{Bernstein theorem and regularity for a class of
  {M}onge-{A}mp{\`e}re equations}, Journal of {D}ifferential {G}eometry
  \textbf{93} (2013), no.~3, 431--469.

\bibitem{LeSavin2017}
Nam~Q. Le and O.~Savin, \emph{Schauder estimates for degenerate
  {M}onge-{A}mp\`ere equations and smoothness of the eigenfunctions}, Invent.
  Math. \textbf{207} (2017), no.~1, 389--423. \MR{3592760}

\bibitem{LeSavin2021}
\bysame, \emph{Global ${C}^{2,\alpha}$ estimates for the {M}onge-{A}mp\`ere
  equation on polygonal domains in the plane}, Amer. J. Math. 145 (2023), no.1, 221-249.

\bibitem{LinWang1998}
F.~Lin and L.~Wang, \emph{A class of fully nonlinear elliptic equations with
  singularity at the boundary}, The {J}ournal of {g}eometric {a}nalysis
  \textbf{8} (1998), 583--598.

\bibitem{Rubin2015}
D.~Rubin, \emph{The {M}onge--{A}mp{\`e}re equation with {G}uillemin boundary
  conditions}, Calculus of Variations and Partial Differential Equations
  \textbf{54} (2015), no.~1, 951--968.

\bibitem{Savin2014}
O.~Savin, \emph{A localization theorem and boundary regularity for a class of
  degenerate {M}onge-{A}mpere equations}, J. Differential Equations
  \textbf{256} (2014), no.~2, 327--388. \MR{3121699}

\bibitem{SavinZhang2020}
O.~Savin and Q.~Zhang, \emph{Boundary regularity for {M}onge-{A}mp{\`e}re
  equations with unbounded right hand side}, Annali della {S}cuola {N}ormale
  {S}uperiore di {P}isa. {C}lasse di scienze \textbf{20} (2020), no.~4,
  1581--1619.

\bibitem{WangZhou2014}
 X.-J. Wang and B.~Zhou,
\emph{K-stability and canonical metrics on toric manifolds},
Bull. Inst. Math. Acad. Sin. (N.S.)9 (2014), no.1, 85-110.


\bibitem{WangZhu2004}
 X.-J. Wang and X.~Zhu, \emph{K{\"a}hler-Ricci solitons on toric manifolds with positive first Chern class},
Advances in Mathematics, 188 (2004), no.1, 87-103.




\bibitem{ZhouZhu2008}
B.~Zhou and X.~Zhu, \emph{Relative  K-stability and modified  K-energy on toric manifolds},
Advances in Mathematics, 219 (2008), no. 4, 1327-1362.

\end{thebibliography}
\end{document}